\documentclass[11pt]{article}%

\usepackage[margin=2cm]{geometry}
\usepackage[utf8]{inputenc}
\usepackage[T1]{fontenc}
\usepackage{lmodern}
\usepackage[french,english]{babel}
\usepackage{amsthm, amssymb, amsmath, amsfonts, mathrsfs, esint, textcomp, bm, xcolor}
\usepackage[colorlinks=true, pdfstartview=FitV, linkcolor=blue, citecolor=blue, urlcolor=blue,pagebackref=false]{hyperref}
\usepackage{mathtools}
\usepackage{tikz}
\usepackage{enumitem, imakeidx}
\usetikzlibrary{patterns}

\usepackage{microtype}

\definecolor{darkblue}{rgb}{0,0,0.7} 
\definecolor{darkred}{rgb}{0.9,0.1,0.1}
\definecolor{darkgreen}{rgb}{0,0.5,0}

\newtheorem{thm}{Theorem}[section]
\newtheorem{prop}[thm]{Proposition}
\newtheorem{lem}[thm]{Lemma}
\newtheorem{cor}[thm]{Corollary}
\theoremstyle{remark}
\newtheorem{rem}[thm]{Remark}
\theoremstyle{definition}

\renewcommand{\O}{\mathcal{O}}

\renewcommand{\leq}{\leqslant}
\renewcommand{\geq}{\geqslant}

\renewcommand{\subset}{\subseteq}

\newcommand{\N}{\mathbb{N}}

\newcommand{\R}{\mathbb{R}}
\newcommand{\C}{\mathbb{C}}
\newcommand{\Q}{\mathbb{Q}}
\newcommand{\Z}{\mathbb{Z}}

\newcommand{\eps}{\varepsilon}
\renewcommand{\d}{{\mathrm{d}}}

\newcommand{\T}{\mathbb{T}}

\newcommand{\vT}{\vec{T}}
\newcommand{\vN}{\vec{N}}

\newcommand{\iii}[1]{{\left\vert\kern-0.25ex\left\vert\kern-0.25ex\left\vert #1 
    \right\vert\kern-0.25ex\right\vert\kern-0.25ex\right\vert}}

\numberwithin{equation}{section}

\begin{document}

\begin{center}
{\Large Edge States for the magnetic Laplacian in domains with smooth boundary}\\
\bigskip
{ARIANNA GIUNTI, JUAN J.L. VEL\'AZQUEZ}
\end{center}

\bigskip

\textbf{Abstract:} We are interested in the spectral properties of the magnetic Schr\"odinger operator $H_\eps$ in
a domain $\Omega \subset \mathbb{R}^2$ with compact boundary and with magnetic field of intensity $\eps^{-2}$. We impose Dirichlet boundary conditions on $\partial\Omega$.
Our main focus is the existence and description of the so-called \textit{edge states}, namely 
eigenfunctions for $H_{\eps}$ whose mass is localized at scale $\eps$ along the boundary $\partial\Omega$.  When the intensity of the magnetic field is large (i.e. $\eps <<1$), we show 
that such edge states exist. Furthermore, we give a detailed description of their localization close to the boundary $\partial\Omega$, as well as how their mass is 
distributed along it.
From this result, we also infer asymptotic formulas for the eigenvalues of $H_\eps$.

\tableofcontents
\section{Introduction}
In this paper we are concerned with the structure of the eigenfunctions for the Schr\"odinger operator in the presence of a magnetic field. More precisely, for a simply connected domain $\Omega \subset \R^2$ with $C^4$ boundary, we study the spectrum and the eigenfunctions of the Hamiltonian
\begin{align}\label{hamiltonian.intro}
H := -(\nabla + i a ) \cdot (\nabla + i a ) 
\end{align}
on $L^2(\Omega)$, with domain $\mathcal{D}(H) =H^2(\Omega) \cap H^1_0(\Omega)$. Here, the \textit{magnetic potential} is a vector field $a: \Omega \to \R^2$  that corresponds to the magnetic field $\nabla \times a = b \, e_3$, with $b=b(x) \in C^{0}(\Omega ; \R)$ and $e_3\in \R^3$ being the canonical versor in the perpendicular direction to the plane containing the domain $\Omega$. 
\bigskip

The focus of our study are the localization properties of the eigenfunctions when the intensity of the magnetic field $b e_3$ is large. Specifically, given the Hamiltonian
$$
H_\eps := -(\nabla + i \eps^{-2} a)\cdot (\nabla + i \eps^{-2}a),
$$
we study the eigenvalue problem
\begin{align}\label{spectral.intro}
\begin{cases}
H_\eps \Psi  = \lambda \Psi \ \ \ &\text{in $\Omega$}\\
\Psi=0 \ \ \ \ &\text{on $\partial\Omega$}
\end{cases}, \ \ \ \lambda \in \R
\end{align} 
with a fixed magnetic potential $a \in C^{{1}}(\Omega ; \R^2)$ and when $\eps << 1$. For a large class of magnetic potentials $a$, many eigenvalues in \eqref{spectral.intro}
correspond to the so-called \textit{edge states}, namely the associated eigenfunctions have most of their $L^2$-norm concentrated along the boundary  $\partial\Omega$ at 
a distance of order $\eps$.

\bigskip

For the sake of simplicity, the main results of this paper are given in the case of a constant magnetic field $b= b_0 \in \R$. We stress, though, that our techniques
may easily be adapted also to non-constant magnetic fields  that do not oscillate too much. For more details about this ``smallness'' condition, we refer to Subsection
\ref{sub.generalizations}. 

\bigskip

In the absence of a boundary in \eqref{spectral.intro} (i.e. when $\Omega = \R^2$), it is well known that the spectrum of $H_\eps$ is pure point, has countably many gaps
of size $\eps^{-2}$ and may be written as $\sigma( H_\eps) = \eps^{-2}\sigma_{\text{Landau}}$, with $\sigma_{\text{Landau}}$ being the so called \text{Landau levels}
\begin{align}\label{landau.intro}
\sigma_{\text{Landau}} : = \biggl\{ b_0 (2 n +1) \, \colon \,  \, n \in \N  \biggr\}.
\end{align}
Here, the notation $\N$ stands for the natural numbers including $0$.
It is well-known that for each fixed value $\eps^{-2}b_0(2n + 1 ) \in \sigma(H_\eps)$, one may construct a countable family of eigenfunctions $\{\Psi_m\}_{m\in \N}$
having (finite) $L^2$-norm that concentrates at infinity when $m \to +\infty$. Our main objective is therefore to understand from an analytical point of view how the 
presence of a boundary in \eqref{spectral.intro} gives rise to eigenvalues in the gaps of $\eps^{-2}\sigma_{\text{Landau}}$ that correspond to edge states.

\bigskip

The main result  (see Theorems \ref{t.main.easy}-\ref{t.main}) that we prove in this paper may be outlined in the following way: If
$\{\lambda_\eps \}_{\eps > 0}$ is any family of eigenvalues for $H_\eps$ in \eqref{spectral.intro} such that 
$\eps^2\lambda_\eps \to \lambda$ and $\lambda \notin \sigma_{\text{Landau}}$, then the corresponding family of eigenfunctions 
$\{ \Psi_\eps \}_{\eps >0}$ is made of edge states. More precisely, we show that if $\Omega_\eps := \{ x \in \Omega \, \colon \, \mathop{dist}(x, \partial\Omega) < L\eps \}$,
with $L$ sufficiently large but independent from $\eps$, then
\begin{align}\label{localized.intro}
\| \Psi_\eps \|_{L^2(\Omega_\eps)} \simeq \| \Psi_\eps\| = 1,
\end{align}
and the probability density $|\Psi_\eps|^2$ is roughly homogeneously distributed along $\partial\Omega$ (see \eqref{cor.uniform.norm}). In addition, we provide a rigorous description of the asymptotic behaviour of $\Psi_\eps$ along the boundary and we derive an ODE describing the asymptotic variation of $|\Psi_\eps|^2$ along $\partial \Omega$ (see \eqref{ODE.amplitude}). Moreover, $\lambda > \min(\sigma_{\text{Landau}}) = b_0$. Reciprocally, we stress that it may be seen by classical asymptotic techniques and perturbation theory for self-adjoint operators that for every  $\lambda > b_0$ such that $\lambda \notin \sigma_{\text{Landau}}$, there exists a sequence of eigenvalues $\{ \lambda_\eps \}_{\eps>0}$ such that $\eps^2\lambda_\eps \to \lambda$.   {We also emphasize that our result \textit{is not} restricted to eigenvalues $\{\lambda_{\eps} \}_{\eps > 0}$ that cluster around the ground state. The limit value $\lambda$ for the rescaled family $\{\eps^2 \lambda_\eps\}_{\eps >0}$, indeed, may be any real number between two fixed Landau levels of $\sigma_{\text{Landau}}$.}

\smallskip

{ In addition to the previous result, in Corollary \ref{c.spectrum} we show that the eigenvalues $\lambda_\eps$ have form
\begin{align}\label{asympt.formula}
\lambda_\eps = \eps^{-2}\nu(k) + \eps^{-1}B(k)+ o(\eps^{-1}),  
\end{align}
where the functions $\nu$ and $B$ are evaluated on the wave number $k$ and do not depend on the geometry of the boundary $\partial\Omega$. On the other hand,
the corresponding eigenfunctions, approximated with the same level of accuracy, \textit{do} depend on the curvature of the domain $\partial\Omega$. We stress that the above formula is not restricted to the ground state.}

\smallskip

The function $e^{-i\lambda t}\Psi_\eps$, with $\Psi_\eps$ solving \eqref{spectral.intro}, corresponds to a solution for the Schr\"odinger equation $i\partial_t \rho= H_\eps\rho$ in $\Omega$ with Dirichlet boundary condition; therefore, the localization result for $\Psi_\eps$ translates into a result on the Fermi levels $\lambda$'s for which there are only currents that are concentrated along the boundary of the domain. Moreover, as it will become apparent in the statement of Theorem \ref{t.main}, the eigenvalue $\lambda$ uniquely determines the sign of the wave number $k$ (see also \eqref{asympt.formula}). This corresponds to a \textit{chiral} property for the Schr\"odinger operator $(i\partial_t - H_\eps)$, namely that the corresponding waves propagate in a preferred direction along the boundary. 

\bigskip

Operators with the form \eqref{hamiltonian.intro} have been extensively studied in both the physical and mathematical 
literature \cite{AvronSimon, Landau}. When $\Omega$ is sufficiently smooth (not necessarily bounded), the operators $H_\eps$ are self-adjoint
\cite{Helffer_book, Pushnitski_Rozenblum}. If $\Omega$ is bounded, the spectrum is discrete \cite{Helffer_book}; if $\Omega$ is an exterior domain then it has both discrete 
and essential spectrum. In this case, the essential spectrum coincides with the rescaled Landau levels set $\eps^{-2}\sigma_{\text{Landau}}$ \cite{Pushnitski_Rozenblum}.

\bigskip

The localization of the eigenfunctions $\Psi_\eps$ with $\eps^2\lambda_\eps$ sufficiently separated from the Landau levels set \eqref{landau.intro} might be expected 
from physical grounds and it has been proven for specific geometries in which the spectrum of $H_\eps$ may be explicitly computed; we refer to \cite{Halperin, Laughlin} for
the case of circular geometries and to \cite{Iwatsuka} for the case of $\Omega$ being the half-plane. This phenomenon is closely related to the \textit{Quantum Hall Effect},
namely that the transport of electrical current in the presence of a magnetic field takes place only along the boundary  when the Fermi energies
(in our setting the value $\lambda$) are different from the Landau levels (see, for instance, \cite{Prange_QHE}[Ch.1]). 
 
\bigskip
 
A result related to the one of this paper has been obtained in \cite{GrafFroelich}. In that paper, the authors study the Hamiltonian $H + V$ 
with $H$ as in \eqref{hamiltonian.intro} and $V$ bounded and satisfying a smallness condition depending on the magnetic field 
$|b|$.  They define the operator in an unbounded domain $\Omega \subset \R^2$, whose boundary is as well an unbounded curve. They consider
Dirichlet boundary conditions or other types of confining potentials. In this case, since the edge states are localized along an unbounded boundary,  they are expected
to correspond to the absolutely continuous part of the spectrum of $H + V$. By using the so-called Mourre  commutator estimates, they prove indeed that the spectrum between
the Landau levels is absolutely continuous if the domain $\Omega$ satisfies a suitable geometric condition (c.f. \cite{GrafFroelich}[assumption (GA)]) which excludes domains
having constant or shrinking thickness sufficiently far away from the origin.

\bigskip

{For what concerns the low-lying eigenvalues for $H_\eps$, there is a large literature focussed on studying their behaviour in the semiclassical limit $\eps \downarrow 0$. More generally, many works study the small eigenvalues for the Pauli-Dirichlet operator $P$ which yields the Hamiltonian for particles with spin in a magnetic field (c.f.  \cite{barbaroux.ecc, Hellfer_Sundqvist}). In particular, the estimates obtained in these papers for general magnetic fields of intensity $b=b(x)$, imply precise asymptotic bounds for the low-energy eigenvalues of $H_\eps$ when $b$ is constant. }

\bigskip

We also mention that the spectral properties of the Hamiltonian \eqref{hamiltonian.intro} with Neumann boundary conditions has
been extensively studied in the context of superconductivity problems (see, e.g. \cite{Helffer_book}). We mention that a remarkable 
feature that takes place in the case of Neumann boundary conditions is that the presence of a non-empty boundary 
$\partial\Omega$ can yield that the values of the ground state are smaller than the minimum of the Landau levels 
\eqref{landau.intro}. { On the other hand, in \cite{Frank_Dirichlet} and \cite{Frank_Neumann} the authors focus on the high-frequency eigenvalues for the operator $H_{\eps}$ above \eqref{spectral.intro}
with Dirichlet and Neumann boundary conditions, respectively. In these works, the main results are trace estimates for operators with the form $f(\eps^2 H_\eps)$, where 
$f$ is a Schwarz function. These give, in particular, a Weyl-like theorem yielding the number of eigenvalues of $H_\eps$ of order $\eps^{-2}$. 
The formula obtained for the counting function exhibits a main contribution due to the bulk part (namely the eigenvalues clustering around the Landau levels) and a
lower order correction due to the edge states. The approach developed in \cite{Frank_Dirichlet,Frank_Neumann} is based on detailed estimates on the contribution to the spectrum of the operator $f(\eps^2 H_\eps)$
of both the bulk and boundary parts. In the current paper, we only focus on the behaviour of the eigenvalues, and in particular of the corresponding eigenfunctions, that
are away from the Landau level and thus correspond to the part of the spectrum related to the presence of the boundary. Only for these eigenvalues, we obtain more accurate
estimates on their behaviour and on the propagation of the corresponding edge state along the boundary $\partial\Omega$.
  }

\bigskip

A class of Schr\"odinger operators in which the existence of (generalized) eigenfunctions localized along an interface has been investigated in a series of papers by
Fefferman, Lee-Thorp and Weinstein (see e.g. \cite{FeffWein12, FeffWein14, FeffWein18, FeffermanWeinstein_graphene}). In these works the authors consider operators of the form $-\Delta + V_\delta$ on $L^2(\R^2)$, with $V_\delta$ being asymptotically periodic at infinity, with the periodicity of an hexagonal lattice. This particular periodicity is motivated by the structure of
graphene. 

\bigskip

We also mention that the edge states studied in this papers may be interpreted in physical terms using the theory of topological insulators
\cite{BernevigHughes, Thouless_etall}. These ones are a class of materials in which the transport of waves cannot take place in the bulk of the material, but only along its
boundary or edges. Many aspects of this propagation along these specific parts of the material are very robust under perturbation of the Hamiltonian $H$ in question, and may
be explained by means of the topological properties of the spectral Floquet-Bloch bands for the spectrum $\sigma(H)$. 

\bigskip

For what concerns the techniques used in this paper, there are several analogies with the ones coming from the theory of homogenization for operators having oscillating coefficients. This is in particular true for what concerns the study of the distribution of the probability density $|\Psi_\eps|^2$ along the boundary $\partial\Omega$. For a more detailed explanation, we refer to Subsection \ref{subsection.main.ideas}. It is also interesting to remark that in the context of homogenization it was proved in \cite{AllaireConca_BlochWaves} that for a large class of second-order operators with highly oscillating periodic coefficients the eigenvalues can either be described in terms of Bloch eigenfunctions or are localized along the boundary.

\bigskip

We finally stress that the results of this papers may be extended with minor changes also to the case of a bounded scalar potential $V$ with
$\| V \|_{L^\infty} \leq C \eps^{-1}$ and, as already mentioned above, to non-constant magnetic fields that do not oscillate too much. Similarly, we may extend all the results
also to  any unbounded domain having compact and $C^4$ boundary. For further details on these generalizations, we refer to  Subsection \ref{sub.generalizations}.
We also mention that in the case of  a general magnetic field $b$ we expect to have eigenfunctions that are partially localized along the boundary and partially on the bulk.
We plan to address this issue in the near future.
\section{Setting and main results}

\subsection{The magnetic Laplacian in the half-plane}\label{s.half.plane.intro}
We begin with a first review of the properties of the spectrum for \eqref{spectral.intro} in the case of the half-plane and with $\eps=1$.  
This case will play a main role in the proof of our results for a general domain $\Omega$ with regular boundary. Indeed, since in this case the edge states are expected to be
localized at distance $\eps$ from $\partial\Omega$, provided that $\eps$ is small enough, the spectral problem 
\eqref{spectral.intro} at scale $\eps$ will resemble the one for the half-plane with constant and unitary magnetic field.

\bigskip

By using the notation $(x,y) \in \R^2$, we denote by $H_0$ the Hamiltonian 
\begin{align}\label{flat.hamiltonian}
H_0:= -\partial^2_x - (\partial_y + i x)^2.
\end{align}
This operator corresponds to the case of the magnetic field $e_3$ and with magnetic potential chosen as $A= -x e_2$. We thus define the eigenvalue problem 
\begin{equation}\label{eigenvalue.flat}
\begin{cases}
H_0 \Psi = \lambda \Psi \ \ \ \ &\text{ in $\{ x> 0\} \times \R$}\\
\Psi_\eps= 0 \ \ \ \  &\text{on $\{x = 0\} \times \R$.}
\end{cases}.
\end{equation}
Solutions to the previous problem, that are solely in $L^2_{loc}(\{ x> 0\} \times \R; \C ) \cap C^\infty(\{ x> 0\} \times \R; \C)$, are given by the (generalized) eigenfunctions
\begin{align}\label{generalized.eigenfunction}
\Psi_{k,n}(x,y) = e^{i k y} H_n(k, x), \ \ \ k \in \R, \,  n \in \N,
\end{align}
with $H_n(k, \cdot) \in H^2(\R_+ ; \R ) \cap C^\infty(\R_+ ; \R)$ solving
\begin{equation}\label{eigenvalue.oscillator}
\begin{cases}
( -\partial^2_x + (x + k)^2) H_n(k, \cdot) = \lambda H_n(k, \cdot) \ \ \ \ &\text{ in $\{ x> 0\}$}\\
H_n(k, 0)= 0
\end{cases}
\end{equation}
and normalized such that $\int_0^{+\infty}|H_n(k , x)|^2 \d x = 1$.

\smallskip

For each $k \in \R$ fixed, it is well-known that the  harmonic oscillator
\begin{align}\label{harmonic.oscillator}
\O(k):= ( -\partial^2_x + (x + k)^2)
\end{align}
on $\{ x > 0 \}$, with Dirichlet boundary conditions, has compact resolvent and therefore a discrete spectrum $\sigma( \O(k) ) := \{ \nu_n(k) \}_{n\in \N} \subset \R_+$
made of simple eigenvalues \cite{ReedSimon}. The functions $\{ H_n(k, \cdot) \}_{n \in \N}$ in \eqref{generalized.eigenfunction} are therefore the corresponding
eigenfunctions.

\smallskip

For every $n \in \N$, we may thus define the function 
\begin{align}\label{definition.branches}
\nu_n : \R \to \R_+, \ \ \ k \mapsto \nu_n(k), \ \ \text{$\nu_n(k)$ is the $n$-th eigenvalue of $\O(k)$.} 
\end{align}
This, together with \eqref{generalized.eigenfunction}, yields that the spectrum of $H_0$ may be written as 
\begin{align}\label{spectrum.H_0}
\sigma(H_0)= \overline{\{ \nu_n( k) \, \colon \, k \in \R, n \in \N \}}.  
\end{align}

More rigorously, we have the following: 

\begin{lem}\label{l.basic.flat} Let $H_0$ and $\O(\cdot)$ be as above. Then 
\begin{itemize}
\item For every $n\in \N$, the ``branch'' $\nu_n= \nu_n(k)$ defined in \eqref{definition.branches} is analytic in $\R$ and satisfies
\begin{align*}
\nu_n \nearrow, \ \ \ \ \ \ \lim_{k\to -\infty} \nu_n(k) =2 n + 1,\ \ \ \ \ \ \lim_{k \to +\infty} \nu_n(k) = +\infty.
\end{align*}
Furthermore, we have
\begin{align}\label{formula.derivative.nu}
\nu'_n(k) =2 \int_{\R_+} (x+ k) |H_n(k, x)|^2 \, \d x, 
\end{align}
where $H_n(k, \cdot)$ is the corresponding eigenfunction;

\item The spectrum of $H_0$ satisfies \eqref{spectrum.H_0} and coincides with the interval $[1 ; +\infty )$. Furthermore, it may be split into an
absolutely continuous part and an essential part. The essential part corresponds to the Landau levels $\sigma_{\text{Landau}}$ of  \eqref{landau.intro} (with $b_0=1$); 

\item For every $k\in \R$ and $n \in \N$, the eigenfunction $H_n(k, \cdot)$ corresponding to the eigenvalue $\nu_n(k)$ in \eqref{eigenvalue.oscillator} satisfies for every $m\in \N$
\begin{align*}
\sup_{R > 0} R^{2m}\int_{|x + k| > R}|H_n(k, x)|^2 \leq C(k, n, m).
\end{align*}
\end{itemize}
\end{lem}
The estimates of this lemma are standard and follow easily by classical results for harmonic oscillators \cite{Kato_perturbations, ReedSimon}. 
We remark that the properties of the functions $\nu_l$ may be inferred by standard comparison principles combined with Min-Max theorems for semi-bounded operators 
\cite{ReedSimon}[Theorem XIII.1]. Identity \eqref{formula.derivative.nu} may be easily proven by differentiating in $k$ the equation for each $H_n(k , \cdot)$ and testing 
with $H_n(k , \cdot)$ itself.

\begin{rem}\label{landau.in.zero}
We remark that the solutions of \eqref{eigenvalue.oscillator} for $k=0$ are all the Hermite functions, namely the eigenfunctions for $\O(0)$ in $\R$, that vanish at the origin. These correspond to all the Hermite functions having even quantum numbers $n \in \N$. By recalling that the spectrum of $\O(0)$ in $\R$ is given by the set $\sigma_{\text{Landau}}$ of \eqref{landau.intro}, this yields that $\nu_n(0) = 2n +  1$ for every $n\in \N$.
\end{rem}
\subsection{Main results}\label{sub.main.results}

Let us consider $H_\eps$ as in \eqref{spectral.intro}, with magnetic potential $a \in C^1(\Omega; \R^2)$ such that $\nabla \times a = e_3$ in $\Omega$.
Let $\{\Psi_\eps \}_{\eps >0} \subset H^2(\Omega ; \C) \cap H^1_0(\Omega; \C)$, $\| \Psi_\eps \|_{L^2(\Omega)}=1$, be a family of eigenfunctions for $H_\eps$ in $\Omega$ with Dirichlet boundary conditions. We denote by $\{ \lambda_\eps \}_{\eps > 0} \subset \R$ the associated eigenvalues.  We assume that there exist $\lambda, \delta >0$ and $N \in \N \backslash \{ 0\}$ such that
 with 
\begin{align}\label{distance.landau}
\eps^2 \lambda_\eps \to \lambda, \ \ \ \  \mathop{dist}(\lambda ; \sigma_{\text{Landau}}) \geq \delta, \ \ \lambda \in (2N- 1;  2N + 1).
\end{align}
We thus denote by $k_1, \cdots k_N \in \R$ the (unique) values such that 
\begin{align}\label{roots}
 \nu_1(k_{1}) = \cdots = \nu_{N}(k_{N}) = \lambda.
\end{align}
We remark that the existence and uniqueness of the previous solutions follows from the properties of the spectrum of $H_0$ enumerated in Lemma \ref{l.basic.flat}.  Furthermore, since $\eps^2\lambda_\eps \to \lambda$ and the functions $\nu_l(\cdot)$, $l \in \N$, are analytic, there exists an
$\eps_0=\eps_0(\lambda) > 0$ such that for every $\eps < \eps_0$ there are exactly $k_{\eps,1}, \cdots, k_{\eps, N} \in \R$  with $k_{\eps,j} \to k_j$ for all
$l=1, \cdots, N$ and such that 
\begin{align}\label{roots.branches}
\nu_{1}(k_{1,\eps}) = \cdots = \nu_N( k_{N,\eps}) = \eps^2\lambda_\eps.
\end{align}

The next two theorems provide a detailed description of the behaviour of the eigenfunctions $\Psi_\eps$ close to the boundary. In particular, they claim that such 
eigenfunctions are proper edge states since they are not only localized at distance $\sim \eps$ from $\partial\Omega$, but are roughly homogeneously distributed along the
boundary.  In view of this, we also fix the following gauge in \eqref{spectral.intro}: We set 
the magnetic potential $a=a(x): \Omega \to \R^2$ to be the rotated gradient
\begin{align}\label{magnetic.potential}
a:= (\nabla \phi)^T, \ \ \ \ \ \ \ \ \begin{cases}
-\Delta\phi = 1 \ \ \ \ &\text{ in $\Omega$}\\
\phi = 0 \ \ \ &\text{ on $\partial\Omega$.}
\end{cases}
\end{align}.

Since we aim at describing the behaviour of $\Psi_\eps$ close to the boundary, it is also useful to introduce the following (local) curvilinear coordinates:
Since $\Omega$ is bounded, simply connected and with $C^4$ boundary, we may parametrize its boundary $\partial\Omega$ by arc-length and write it as
\begin{align}\label{parametrized.boundary}
\partial\Omega = \{ f(\xi) \in \R^2 \, \colon \, \xi \in \T_L \}, \ \ \ f : \T_L \to \R^2,
\end{align}
with $L$ being the length of the curve $\partial\Omega$. Without loss of generality, throughout the rest of the paper, we assume that $L=1$. For every $\xi \in \T$ we thus
denote by
\begin{align}
T(\xi) = (f_1'(\xi), f_2'(\xi)), \ \ \ N(\xi) := (T(\xi))^\perp,
\end{align}
the tangent and the (outer) normal vector at $f(\xi) \in \partial\Omega$, respectively. Since $\partial\Omega$ is $C^4$, there exists
$\kappa >0$ such that the tubular neighbourhood
\begin{align}\label{tubular.neighbourhood}
\Omega_{2\kappa}:= \{ x \in \Omega \, | \, \mathop{dist}(x, \partial\Omega) \in (0, 2\delta) \} \simeq \{ (\xi, s) \in \T \times (0, 2\delta) \},
\end{align}
where the new coordinates $(\xi, s)$ are defined by
\begin{equation}\label{local.coordinates}
 x(\xi, s) = f(\xi) - \vN(\xi) s.
\end{equation}
Throughout the paper, for any given two points $\xi, \xi^* \in \T$, we denote by $d(\xi, \xi^*)$ the distance on $\T$  between $\xi $ and $\xi^*$.

\bigskip

For $f$ as in \eqref{parametrized.boundary}, we denote by $\alpha: \T \to \R$ the normal derivative of $\phi$ along $\partial\Omega$ and $\kappa: \T \to \R$ the signed curvature
of $\partial\Omega$. In other words, for every $\xi \in \T$
\begin{align}\label{functions.curvature}
\alpha(\xi):= \partial_n \phi(f(\xi)), \ \ \ \vT  '(\xi)= \kappa(\xi)\vN(\xi).
\end{align}
For every $\eps>0$, $l \in \N$ and $k\in \R$, we also set
\begin{align}\label{quantities.corollary}
\omega_\eps := \frac{|\Omega|}{2\pi \eps^2} - \lfloor \frac{|\Omega|}{2\pi \eps^2}\rfloor,\ \ \ \ \ \ B_l(k)= -\nu_l'(k)^{-1} \int_0^{+\infty} \mu ( (\mu + k)^2 + k^2) |H_l(k,\mu)|^2 \d\mu.
\end{align}
We remark that as an easy consequence of assumption \eqref{distance.landau}, it follows that the eigenfunctions are localized in a neighbourhood of $\partial \Omega$ of size $\eps$, namely that 
\begin{prop}\label{t.localization}
Let $(\Psi_\eps, \lambda_\eps)$ be as above and let $d_{\partial\Omega}=d_{\partial\Omega}(x)$ be the distance function from the boundary $\partial\Omega$. Then for every $n \in \N$ there exists a constant $C=C(N, \delta, n)$ such that
\begin{align}\label{localization}
 \| d_{\partial\Omega}^n \Psi_\eps \|_{L^2(\Omega ; \C)} + \eps  \| d_{\partial\Omega}^n( \nabla + i \eps^{-2}a )\Psi_\eps \|_{L^2(\Omega; \C)}+  \eps^2  \| d_{\partial\Omega}^n (\nabla + i \eps^{-2}a )^2\Psi_\eps \|_{L^2(\Omega; \C)} \leq C \eps^n.
\end{align}
\end{prop}
{We prove this statement in Section \ref{s.localization} by means of an easy rescaling argument that is similar to the one of \cite[Proposition 5.6]{AllaireConca_BlochWaves}.  We do not claim that this result is optimal. We expect to obtain exponential localization by adapting existing techniques for eigenfunctions corresponding to low-lying eigenvalues \cite[{Theorem 8.2.4.}]{Helffer_book}. However, the lemma above serves our purposes as it allows us to assume that the eigenfunctions $\Psi_\eps$ are supported in the tubular neighbourhood defined in \eqref{tubular.neighbourhood}, up to an error of order $\eps^n$, for $n \in \N$ arbitrary. This motivates the slight abuse of notation of the next two theorems, where the eigenfunctions $\Psi_\eps$ are expressed in the curvilinear coordinates $(s, \xi)$ which are assumed to be in $\T \times \R_+$ (see Step 1 in the proof of Theorem \ref{t.main}, for the rigorous proof of this claim). }

\bigskip

Our first main result focusses on the case of the rescaled eigenvalues $\eps^2\lambda_\eps$ being between the first and second Landau level 
(i.e. with $N=1$ and any $\delta>0$ in \eqref{distance.landau}).
\begin{thm}\label{t.main.easy}
Let $\{\Psi_\eps \}_{\eps>0} \subset H^2(\Omega ; \C) \cap H^1_0(\Omega ; \C)$ solve \eqref{spectral.intro} with $a$ as in \eqref{magnetic.potential}. Let, $\{\lambda_\eps\}_{\eps>0}$ be the associated eigenvalues satisfying \eqref{distance.landau} for some $\delta > 0$ and with $N=1$. Then, for every sequence $\eps \downarrow 0$ we may extract a subsequence $ \eps_j \downarrow 0$ and $\rho \in [0, 2\pi)$ such that the function
\begin{align*}
\Psi_{\text{flat},\eps}&:= e^{i( \theta_\eps(\xi,s)- k_{1,\eps} \frac \xi \eps + \rho)}H_1(k_{1,\eps}, \frac s \eps),\\
\theta_{\eps}(\xi, s)&:= \eps^{-2}(-\int_0^\xi \alpha(y) \d y+ \frac{1}{2}\alpha'(\xi)s^2)  + B_1(k_1) \int_0^\xi \kappa(y) \d y  + {\omega_\eps} \xi  \in \R
\end{align*}
satisfies
\begin{align}\label{closeness.to.flat.theorem.easy}
\lim_{j \uparrow +\infty}\sup_{\xi^* \in \T}\bigl(\fint_{d(\xi; \xi^*)< \eps_j} \int_0^{+\infty} | \Psi_{\eps_j}(\xi, s) - \frac{1}{\sqrt{\eps_j}}\Psi_{\text{flat}, \eps_j}(\xi, s)|^2 \d\xi \, \d s \bigr)^{\frac 1 2} = 0.
\end{align}
\end{thm}
We stress that in the previous statement the integral in the variable $\xi \in \T$ is over the set $\{\xi \in \T \, \colon \, d(\xi, \xi_*) < \eps \}$ and averaged by its the measure. We stress that, as it will be the case in most parts of this paper, since $\eps \downarrow 0$, the set locally is diffeomorphic to the Euclidean space $\R$. 

\smallskip

The next theorem extends the previous result to the general case of the limit value $\lambda$ in \eqref{distance.landau} being between any two Landau levels. In contrast with the previous case, there is more than one value $k \in \R$ satisfying \eqref{roots}. Hence, at scale $\eps$ the eigenfunctions are expected to behave as a linear combination of the functions
$e^{ik_{j,\eps} \frac{\xi}{\eps}} H_j(k_{j,\eps}, \frac s \eps)$. In order to understand how each one of the eigenfunctions in the sum 
propagates along the boundary, i.e. proving the existence of $\theta_{l,\eps} \in \R$ above for each $l=1, \cdots, N$, one needs to deal with possible resonances between
the plane waves $e^{ik_{j,\eps} \frac{\xi}{\eps}}$, $j=1, \cdots, N$. For general values $k_{1,\eps}, \cdots, k_{N,\eps} \in \R$, indeed, the ``interaction'' between the
previous waves becomes negligible only when averaging in $\xi$ over lengthscales of order bigger than $\eps$. More precisely, for every $\alpha > 0$, one may find $M > +\infty$ such that for all
$j, l = 1, \cdots, N$ with $j \neq l$ 
\begin{align}\label{diophantine.main}
|\fint_{|\xi| < M\eps} e^{i (k_{j,\eps} - k_{l,\eps}) \frac{\xi}{\eps}} \, \d\xi | \leq \alpha.
\end{align}
Note that this implies in particular that for all $c_1, \cdots, c_N \in \C$ with $\sum_{l=1}^N |c_l|^2 \leq 1$
\begin{align}\label{diophantine.main.sum}
\bigl| \fint_{|\xi| < M\eps} |\sum_{l=1}^N c_l e^{i k_{l,\eps} \frac{\xi}{\eps}}|^2 \, \d\xi  - \sum_{l=1}^N |c_l|^2 \bigr| \leq \alpha N^2.
\end{align}
This motivates the need of a mesoscale $M_\eps \eps$, $M_\eps \to +\infty$ in the convergence result of the next theorem. However, we stress that the sequence $M_\eps$ may
be chosen to diverge at infinity as slowly as needed.

\smallskip

\begin{thm}\label{t.main}
Let $\{\Psi_\eps, \lambda_\eps\}_{\eps> 0}$ be as in the previous theorem. Let the family of associated eigenvalues $\{ \lambda_\eps\}_{\eps}$ satisfy \eqref{distance.landau} for some $\delta > 0$ and $N \in \N$. Let $\{ M_\eps \}_{\eps > 0} \subset \R_+$ be such that 
$$
\lim_{\eps \downarrow 0} M_\eps = +\infty , \ \ \ \ \lim_{\eps \downarrow 0} \eps M_\eps = 0.
$$
Then, there exists a sequence $\eps_j \downarrow 0$, constants $C_1, \cdots, C_{N} \in \C$ with $\sum_{j=1}^{N}|C_j|^2 =1$ such that the function
\begin{align*}
\Psi_{\text{flat},\eps}(\xi, s)&:= \sum_{j=1}^{N} C_l e^{i( \theta_l(\xi,s)- k_{l,\eps} \frac \xi \eps)}H_j(k_{l,\eps}, \frac s \eps),\\
\theta_{l,\eps}(\xi, s)&:= \eps^{-2}\bigl(-\int_0^\xi \alpha(y) \d y  + \frac{1}{2}\alpha'(\xi) s^2\biggr) + B_j(k_l) \int_0^\xi \kappa(y) \d y + \omega_\eps \xi 
\end{align*}
satisfies
\begin{align}\label{closeness.to.flat.theorem}
\lim_{j \uparrow +\infty}\sup_{\xi^* \in \T}\bigl(\fint_{d(\xi, \xi^*)<  M_j \eps_j} \int_0^{+\infty} |\Psi_{\eps_j}(\xi, s) -\frac{1}{\sqrt{ \eps_j}} \Psi_{\text{flat}, \eps_j}(\xi,s)|^2 \d\xi \, \d s \bigr)^{\frac 1 2}= 0.
\end{align}
\end{thm}

\bigskip

\begin{rem}
We remark that if the values $k_1, \cdots, k_N$ are pairwise commensurable, namely there exist $\{ q_l \}_{l=2}^N \subset \Q$ such that $k_l = q_l k_1$ for all $l=2, \cdots, N$, then the family $\{M_\eps \}_{\eps >0}$ above may be chosen also to be constant, thus recovering the exact same estimate of Theorem \ref{t.main.easy}. In this case, indeed, we may find $M < +\infty$ such that all the integrals in \eqref{diophantine.main} vanish and the proof of Theorem \ref{t.main.easy} adapts to this case only with trivial modifications.
\end{rem}

\bigskip

It immediate to see that the description of the eigenfunctions $\{\Psi_\eps \}_{\eps >0}$ in the previous theorems in particular implies that for every family $\{M_\eps \}_{\eps>0}$ be as in Theorem \ref{t.main} and for every $\xi^* \in \T$ it holds
\begin{align}\label{cor.uniform.norm}
\lim_{\eps \downarrow 0} \fint_{\{ |\xi - \xi^*| < \eps M_\eps \}} \int_0^{+\infty} |\Psi_\eps |^2 \d s\,  \d\xi = 1.
\end{align}
In other words, the eigenfunction $\Psi_\eps$ propagates along the boundary $\partial \Omega$ and the mass $|\Psi_\eps|^2$ is basically uniformly distributed on the boundary $\partial \Omega$. Another consequence of the asymptotic estimate \eqref{closeness.to.flat.theorem} is the following corollary that describes, up to order $\eps^{-1}$ the behaviour of the spectrum of $H_\eps$ away from the rescaled Landau levels $\eps^{-2}\sigma_{\text{Landau}}$.
\begin{cor}\label{c.spectrum} 
Any eigenvalue $\lambda_\eps \in \sigma(H_\eps)$ satisfying \eqref{distance.landau} is such that for some $l \in \N$ and $q_\eps \in 2\pi\eps\Z$
\begin{align}\label{characterization.spectrum}
\lim_{\eps \downarrow 0} \eps^{-1}|\eps^2\lambda_\eps - \nu_{l}(q_\eps) - \eps 2\pi \nu'_{l}(q_\eps) B_{l}(q_\eps)| =0.
\end{align}
 Here, $B_l( \cdot )$ and $\omega_l$ as in \eqref{quantities.corollary}.  Vice versa, for all $l \in \N$ and $q_\eps \in  {2\pi}\eps \Z$ such that $\eps^{-2}\nu_{l}(q_\eps)$ 
 satisfies \eqref{distance.landau}, there exists $\lambda_\eps \in \sigma(H_{\eps})$ satisfying \eqref{characterization.spectrum}.
\end{cor}

\subsection{Main ideas in the proofs of  Theorems \ref{t.main.easy}-\ref{t.main}}\label{subsection.main.ideas}

The proofs of Theorems \ref{t.main.easy} and \ref{t.main} rely on the same strategy and they
only differ in the parts where the resonance of the waves $ \{ e^{i k_{j,\eps}\frac \xi \eps} \}_{j=1}^N$ gives some technical challenges (c.f. \eqref{diophantine.main}).
We therefore discuss the strategy in the case of Theorem \ref{t.main.easy}. We remark that Proposition \ref{t.localization} yields that $\Psi_\eps$ is localized close to the
boundary but does not give any information on how the norm is distributed along it. It does not exclude, indeed, that $\Psi_\eps$ is localized only around a portion of the
boundary $\partial\Omega$. In order to show that $\Psi_\eps$ is an edge state, we need to rule out this scenario. This is the real challenge in the
proof of Theorems \ref{t.main}-\ref{t.main.easy} as it requires to tie together the microscopic behaviour at scale $\eps$ of $\Psi_\eps$ with its macroscopic behaviour
along $\partial\Omega \sim \T$. 

\smallskip

By the localization estimates of Proposition \ref{t.localization}, it is natural to believe that if we fix a point $(\xi^*, 0)$ of the boundary $\partial\Omega$ and rescale
the local coordinates $(\xi, s) \mapsto (\eps \theta + \xi^*, \eps \mu)$ around such point, the magnified problem \eqref{spectral.intro} ``resembles'' the one in
\eqref{eigenvalue.flat}  in the new coordinates  $(\theta, \mu) \in \R \times \R_+$. In other words, when $\eps<<1$, the behaviour of the rescaled function
$$
\tilde\Psi_\eps(\theta, \mu):= \sqrt\eps \Psi_\eps(\eps \theta, \eps\mu)
$$
in the sets $\{ |\theta| < R \} \times \R_+$, $R> 0$, is expected to be close to a multiple of the eigenfunction $e^{i k_{\eps, 1} \theta} H_1(k_{1,\eps}, \mu)$ 
 corresponding to the eigenvalue $\eps^2 \lambda_\eps$ for $H_0$ in the half-plane (c.f. Lemma \ref{l.basic.flat}).
To rigorously prove this, the main technical challenge is to characterize the limit for $\tilde \Psi_\eps$: We do this by proving a rigidity/Liouville statement for 
(very) weak solutions to \eqref{eigenvalue.flat} that have a certain growth in variable $\theta$ (see Lemma \ref{l.identification}). We note that the scaling $\sqrt \eps$ in
the definition of $\tilde\Psi_\eps$ is the correct one once we post-process the estimate
of the main theorem. In the proof, we need to first rescale $\Psi_\eps$ by $\frac{\eps}{m_\eps}$, with $m_\eps$ being the maximum of the $L^2$ norm of $\Psi_\eps$ on the
sets $\{\xi \, \colon \,  d(\xi; \xi^*) < \eps \}\times \R_+$, for $\xi^* \in \T$.

\smallskip

By the reasoning of the previous paragraph, for every point $(\xi^*, 0)$ on the boundary, we may find a suitable constant
$A_\eps(\xi) \in \C$, $|A_\eps(\xi^*)| \leq 1$, such that  for every $R> 0$
\begin{align}\label{qualitative.convergence}
\int_{|\theta|< R}\int_0^{+\infty} |\tilde\Psi_\eps - A_\eps(\xi^*) e^{i k_{\eps, 1} \theta} H_1(k_{1,\eps}, \mu)|^2 \d\theta \, \d\mu \to 0.
\end{align}
The next step, and the main challenge, is therefore to obtain a description of the behaviour of the amplitude $A_{\eps} \in L^\infty( \T)$ along $\partial\Omega$.
In particular, we prove that $A_\eps(\xi) \to F(\xi)$ uniformly on $\T$, with $F$ solving for every $\xi^* \in \T$ 
the boundary value problem
\begin{align}\label{ODE.amplitude}
\begin{cases}
F' = i \kappa B_1(k_1) F \ \ \text{in $\{ \xi \, \colon \, d(\xi; \xi^*) < \frac 1 2 \}$}\\
|F(\xi^*)| = 1
\end{cases}
\end{align}
This implies that the function $A_{\eps} \in L^\infty( \partial\Omega)$ approximately behaves as the phase $e^{i B_1(k_1)\int_0^\xi \kappa(x)}$ and that $|\Psi_\eps|^2$ is roughly homogeneous along
$\partial\Omega$. 

\smallskip

To show this, we need to push \eqref{qualitative.convergence} also up to lengthscales $R \sim \frac 1 \eps$. In the original coordinates this means passing from an 
approximation for $\Psi_\eps$ in neighbourhoods of size $\eps$ of the boundary to neighbourhoods of size $1$. We thus need to consider the error term
$$
\Psi_{\eps, 1}(\xi^*, \theta, \mu):= \frac{\tilde \Psi_\eps(\theta, \mu) - A_\eps(\xi^*) e^{i k_{\eps, 1} \theta } H_1(k_{1,\eps}, \mu)}{\eps}
$$
and show that it grows at most linearly (with constant independent from $\eps$) in the angular variable $\theta \in \T_{\frac 1 \eps}$.
This is the content of Proposition \ref{p.main}; the proof of \eqref{ODE.amplitude} from the sublinearity result of Proposition \ref{p.main} is contained instead in Lemma
\ref{p.main.2}. 

\smallskip

We stress that the proof of the previous results crucially relies on Proposition \ref{proposition.growth} in Section \ref{s.auxiliary}: Roughly speaking, this
result states that if $\Psi$ solves \eqref{eigenvalue.flat} with a further right-hand side that grows as a polynomial of degree $m$ in the variable
$\theta$, then $\Psi$ grows in $\theta$ at most like a polynomial of degree $m+1$. This result combines techniques of Fourier analysis with resolvent estimates and allows
to deal also with the frequencies where the operator $H_0 - \lambda$ becomes singular (i.e. the values of $k$ such that $\nu_l(k)=\lambda$). We mention that a tool that is extensively used 
throughout all the previous results is the standard relation between the order of growth of a function at infinity and the regularity of its Fourier transform.

\smallskip

In order to obtain a description of the derivative of $A_\eps$ and prove \eqref{ODE.amplitude}, we need to perform the previous steps up to the second-order error:
More precisely, for every $\xi^* \in \T$ corresponding to a point on $\partial\Omega$, we find the approximation for $\tilde\Psi_\eps$ in the sense of 
\eqref{qualitative.convergence} up to a linear correction in $\eps\theta (= \xi)$. In other words, we define
$$
(A_\eps(\xi^*) + i B_\eps(\xi^*)\eps \theta) e^{i k_{\eps, 1} \theta} H_1(k_{1,\eps}, \mu),
$$
and show that the second-order error 
$$
\Psi_{\eps, 2}(\xi^*, \theta, \mu):= \frac{\tilde \Psi_\eps(\theta, \mu) - (A_\eps(\xi^*) + i B_\eps(\xi^*)\eps \theta) e^{i k_{\eps, 1} \theta} H_1(k_{1,\eps}, \mu)}{\eps^2}
$$
grows at most quadratically in $\theta \in \T_{\frac 1 \eps}$. This allows us not only to give a macroscopic characterization of $A_\eps$, but also to describe its (approximate) 
derivative via the slope $B_\eps$. 

\subsection{Generalizations}\label{sub.generalizations}
\noindent{\bf Additional electric potential $V$. } By inspecting the proofs of the main theorems it is easy to see that the arguments can be adapted
to the case of the Hamiltonian $H_\eps + V_\eps$, for any uniformly bounded scalar (real) potential $\|V_\eps \|_{L^\infty(\R^2; \R)} \leq C$, or for $V_\eps= \eps^{-1} V$, with 
$V \in L^\infty(\R^2; \R)$. In the first case, the result is the same of the case $V_\eps \equiv 0$, with the constants in the estimates of the theorems and corollary also depending
on $\sup_{\eps > 0}\| V \|_{L^\infty}$. In the second case, the effect of the potential $V_\eps$ will appear in the function $\theta_\eps$ and in  the asymptotic estimates
for the eigenvalues $\lambda_\eps$. The function $\theta_\eps$ contains, indeed, also the term {$\int_0^{\xi} V((y,0)) \d y$}. In the asymptotic estimate for $\lambda_\eps$, in
 the term of order $\eps^{-1}$ the constant $\nu'_{l}(q_\eps)(B_{l}(q_\eps))$ is substituted by  $\nu'_{l}(q_\eps)\bigl(B_{l}(q_\eps) +  \int_{\partial\Omega} V(x) \d x \bigr)$.

\bigskip

\noindent{\bf Exterior domains.} In the case of an exterior domain $\Omega = \R^2 \backslash K$, with $K \subset \R^2$ compact, simply connected and having $C^4$ boundary,
all the main results hold with trivial modifications. In this case, we stress that the vector field $\phi$ in the definition \eqref{magnetic.potential} of the magnetic 
potential $a$ may be chosen as the solution to the Poisson problem \eqref{magnetic.potential} in the exterior domain $\Omega$ such that
$$
\limsup_{|x|\to +\infty} |\phi- \frac{1}{4}|x|^2 | < +\infty.
$$
Furthermore, in the definition of $\omega_\eps$ in \eqref{quantities.corollary}, the term $|\Omega|$ has to be replaced by $|K| < +\infty$.

\bigskip

\noindent{\bf Non-constant magnetic fields.} Let us consider a magnetic field $\nabla \times a = b e_3$, with 
$b \in C^1(\R^2) \cap L^\infty(\R^2)$. We define the subset of $\R$ 
\begin{align*}
\sigma_{\text{Landau}, b}:= \{ b(x)( 2n + 1) \, \colon \, x \in \R^2, \, n\in \N \}. 
\end{align*}
Then, if the previous subset of $\R$ has gaps, namely  $\sigma_{\text{Landau}, b}$ is not connected, our main results hold provided that the limit value 
$\lambda$ for the sequence $\eps^2\lambda_\eps$ is chosen inside a gap. In this case, the values $k_{1}, \cdots, k_N$, as well as $k_{1,\eps}, \cdots, k_{N,\eps}$ do depend
on $x \in \partial\Omega$. Therefore, in the local coordinates, they are functions of the angular variable $\xi\in \T$. Note that the assumption that $\lambda$ is in a gap
of $\sigma_{\text{Landau}, b}$ yields that each $k_{l, \eps}= k_{l,\eps}(\xi)$ satisfies, for $\eps$ small enough, the identity
$\nu_l( k_{l,\eps}(\xi)) = \eps^2 \lambda_\eps$ for every $\xi \in \T$, with the index $l\in \N$ being fixed. 

\bigskip

\subsection{Notation}
Throughout this paper we employ the notation $\lesssim$ or $\gtrsim$ for $\leq C$ or $\geq C$ with the constant $C$ depending on the domain $\Omega$ and the parameters $\delta$, $N$ in \eqref{distance.landau}. When no ambiguity occurs, we skip omit the target space in the notation for the function spaces used in this paper and write, for instance, $\Psi_\eps \in L^2(\Omega)$ instead of $\Psi_\eps \in L^2(\Omega ; \C)$. 

\section{Proof of Proposition \ref{t.localization} and of Theorem \ref{t.main.easy}}\label{s.localization}

\begin{proof}[Proof of Proposition \ref{t.localization}]
Since $\lambda$ is assumed to satisfy \eqref{distance.landau} and $\eps^2 \lambda_\eps \to \lambda$, we may select $\eps_0$ such that for all $\eps \leq \eps_0$ also
$\eps^2\lambda_\eps$ satisfies \eqref{distance.landau}. We prove the statement for all such $\eps$. Throughout this proof the constant implied in the notation $\lesssim$ or $\gtrsim$ also depends on the exponent $n\in \N$ in \eqref{l.localization.step.1}. In addition, we denote by $\| \cdot \|$ the $L^2$-norm in the domain $\Omega$ and write $d=d(x)$ instead of $d_{\partial\Omega}(x)$ for the distance function from the boundary of $\Omega$.

\bigskip

We first argue that it suffices to prove
\begin{align}\label{l.localization.step.1}
 \| \phi^n \Psi_\eps \| + \eps \| \phi^n (\nabla + i \eps^{-2}a) \Psi_\eps \| + \eps^2 \| \phi^n (\nabla + i \eps^{-2}a)^2&\Psi_\eps\| \lesssim \eps^n  \|\Psi_\eps\|,
\end{align}
where $\phi$ is the solution to \eqref{magnetic.potential}. Note that, since $\Omega$ has $C^4$- boundary, by standard elliptic regularity we have that {$\phi \in C^{2,\alpha}(\bar \Omega)$}, $0< \alpha < 1$. In addition, by the maximum principle, $\phi > 0$ in $\Omega$. 

\smallskip

For $r >0$, let 
$$
 \Omega_{r}:= \{ x \in \Omega \, \colon \, \mathop{dist}(x, \partial\Omega) < r \}
$$
and let $\eta_r$ be a cut-off function for $\Omega_r$ in $\Omega_{2r}$. Since $\Omega$ is bounded, by the
maximum principle the function $\phi$ attains a positive minimum in $\Omega \backslash \Omega_{r}$. Hence, there exists a constant $C(r)> 0$ such that
\begin{align*}
 \| (1 -\eta_r) d^n \Psi_\eps \|^2 \leq C(r, n) \| \phi^n \Psi_\eps \|^2 \stackrel{\eqref{l.localization.step.1}}{\lesssim}\eps^{n}. 
\end{align*}
Since the same may be done for the other terms $(1-\eta_r)( \nabla + i\eps^{-2}a) \Psi_\eps$ and $(1-\eta_r)(\nabla + i\eps^{-2}a)^2\Psi_\eps$, we prove Proposition \ref{t.localization} from \eqref{l.localization.step.1} provided that we show that we may fix $r> 0$ such that
\begin{align}\label{l.localization.scale.r}
  \| d^n\eta_r\Psi_\eps \| + \eps  \| d^n\eta_r( \nabla + i \eps^{-2}a)\Psi_\eps \|+  \eps^2  \| d^n\eta_r (\nabla + i \eps^{-2}a)^2\Psi_\eps \| \lesssim \eps^n.
\end{align}
This is an easy consequence of the regularity of the domain $\Omega$ and the definition of $\phi$: Restricting to $r < \frac{\delta_0}{2}$  (see \eqref{tubular.neighbourhood}), we may pass into the local curvilinear coordinates $(\xi,s)$ and simply rewrite $d(x) = s$ in $\Omega_{r}$. By the regularity of $\partial\Omega$ and {the Implicit Function Theorem}, there exists a $0<r < \frac{\delta_0}{2}$ such that $\phi(\xi,s) = \alpha(\xi) s + \epsilon(\xi, s)$,  for every $(\xi, s) \in \T \times [0,2r]$, with $|\epsilon(\xi, s)| \lesssim s^2$ and $\alpha$ as in \eqref{functions.curvature}. On the one hand, $|\alpha(\xi)| < C$ for some $C< +\infty$ by the regularity of $\partial\Omega$. On the other hand, Hopf's lemma and the compactness of the boundary imply that there exists a positive constant $c > 0$ for which $\alpha(\xi) > c$ for all $\xi \in \T$. Hence, there exist $C, c >0$ (depending on the domain $\Omega$) such that
\begin{align}\label{comparable.to.distance}
c d \leq \phi \leq C d, \ \ \ \ \text{in $\Omega_r$}
\end{align}
This inequality implies \eqref{l.localization.scale.r} and, in turn, allows to establish \eqref{t.localization} from \eqref{l.localization.step.1}.

\smallskip

We now turn to \eqref{l.localization.step.1} and begin by showing that
\begin{align}\label{l.localization.2}
 \eps \|\phi^n (\nabla + i\eps^{-2}a)\Psi_\eps \| + \eps^2 \|\phi^n (\nabla + i \eps^{-2}a)^2 \Psi_\eps \| \lesssim \|\phi^{n}\Psi_\eps\| + \eps\|\phi^{n-1} \Psi_\eps \|,
\end{align}
so that \eqref{l.localization.step.1} reduces to prove that for all $n\in \N$
\begin{align}\label{l.localization.step.1.b}
\| \phi^n \Psi_\eps \|_{L^2(\Omega)} \lesssim \eps^n  \|\Psi_\eps\|_{L^2(\Omega)}.
\end{align}
By testing the equation \eqref{spectral.intro} with $\Psi_\eps$ itself, we obtain immediately that
\begin{align}\label{energy.estimate.localization}
 \|(\nabla + i\eps^{-2}a)\Psi_\eps\|^2 \leq \lambda_\eps \|\Psi\|^2 \stackrel{\eqref{distance.landau}}{\lesssim}\eps^{-2}\|\Psi_\eps\|^2.
\end{align}
Hence, if we test \eqref{spectral.intro} with $\phi^{2n} \Psi_\eps$ for $n\in \N$, the previous energy estimate, the fact that $\nabla\phi$ is bounded and again 
\eqref{distance.landau} yield
\begin{align}\label{l.localization.1}
\|\phi^n (\nabla + i\eps^{-2}a)\Psi_\eps \| \lesssim \eps^{-1}\|\phi^{n}\Psi_\eps\| + \|\phi^{n-1} \Psi_\eps \|.
\end{align}
Furthermore, since $\Psi_\eps \in H^2(\Omega)$, we may bound as well
\begin{align}\label{l.localization.2}
\|\phi^n (\nabla + i \eps^{-2}a)^2 \Psi_\eps \| \lesssim \lambda_\eps \|\phi \Psi_\eps \| \stackrel{\eqref{distance.landau}}{\lesssim} \eps^{-2} \| \phi^n \Psi_\eps\|.
\end{align}
This inequality, together with \eqref{l.localization.1}, implies \eqref{l.localization.2} and thus reduces the proof of \eqref{l.localization.step.1} to \eqref{l.localization.step.1.b}.

\bigskip

We argue \eqref{l.localization.step.1.b} as follows: By gauge invariance, we prove the inequality for the solution $\bar \Psi_\eps$ to \eqref{spectral.intro},
with magnetic potential given by $\eps^{-2}a_0, \ \ \ \ a_0(x) = - x_1 e_2$, where we denote $(x_1, x_2) \in \R^2$ and where $e_2 \in \R^2$ is the second versor of the standard canonical base of $\R^2$. To keep the notation lean, throughout the proof of \eqref{l.localization.step.1.b} we write $\Psi_\eps$ instead of $\bar \Psi_\eps$. 

\smallskip

We prove \eqref{l.localization.step.1.b} by induction over $n\in \N$. We begin with the case $n=1$: By rewriting the equation for $\Psi_\eps$ in the new variables $\tilde x = \eps^{-1}x$, we have that the function
$\tilde \Psi_\eps (\tilde x )= \Psi_\eps(\eps \tilde x)$ satisfies in the domain $\frac 1 \eps \Omega:=\{ y \in \R^2 \, \colon \eps y \in \Omega \}$ the boundary value problem
\begin{align}
\begin{cases}
 -(\nabla + i a_0) \cdot (\nabla + i a_0)\tilde\Psi_\eps = \eps^2\lambda_\eps \tilde \Psi \ \ \ \ &\text{in $\frac{1}{\eps}\Omega$}\\
 \tilde\Psi_\eps= 0 \ \ \ &\text{on $\partial (\frac{1}{\eps}\Omega)$}.
\end{cases}
\end{align}
If $\phi$ is as in \eqref{magnetic.potential}, then $\tilde \phi (\tilde x )= \phi(\eps \tilde x)$ satisfies
\begin{align}\label{bounds.landscape.rescaled}
\|\nabla \tilde\phi \|_{L^\infty(\Omega)} +\eps^{-1}\| \nabla^2 \tilde\phi \|_{L^\infty(\Omega)} \lesssim \eps.
\end{align}
Then, the function $F_\eps:= \tilde\phi \tilde \Psi_\eps \in H^1(\R^2)$ weakly solves
\begin{align*}
-(\nabla + i a_0) \cdot (\nabla + i a_0) F_\eps = \eps^2\lambda_\eps F_\eps + \tilde R_\eps \ \ \ \ \text{in $\R^2$},
\end{align*}
where the error term $\tilde R_\eps:= 2\nabla \tilde \phi \cdot (\nabla + i a_0)\tilde\Psi_\eps - \Delta\tilde\phi \tilde \Psi_\eps$ satisfies
\begin{equation}\label{close.to.eigenvalue}
\| \tilde R_\eps \| \stackrel{\eqref{bounds.landscape.rescaled}-\eqref{energy.estimate.localization}}{\lesssim} \eps \| \tilde \Psi_\eps\|.
\end{equation}
This implies that
\begin{align}\label{resolvent.distance}
\|\phi \Psi_\eps \| \lesssim \frac{1}{\mathop{dist}(\eps^2\lambda_\eps, \sigma_{\text{Landau}})} \| R_\eps \|. 
\end{align}
Note that $F_\eps$ is in the domain of the operator, by standard elliptic regularity theory and since the domain $\frac 1 \eps \Omega$ is bounded.
 
\medskip

Let us now assume that \eqref{l.localization.step.1.b} is true for all $n \leq n_0$ with $n_0 \in \N$. Then, the function 
$$
F_{n_0+1, \eps}:= \tilde\phi^{n_0+1}\tilde\Psi_\eps
$$ 
solves
\begin{align*}
-(\nabla + i a_0) \cdot (\nabla + ia_0)F_{n_0+1, \eps} = \eps^2\lambda_\eps  F_{n_0+1, \eps} + R_{n_0,\eps} \ \ \ \ \ \text{ in $\R^2$},
\end{align*}
with
\begin{align*}
R_{n_0,\eps} =2 (n_0+1) \tilde\phi^{n_0} \nabla \tilde\phi \cdot (\nabla + i a_0) \tilde\Psi_\eps + \Delta( \tilde\phi^{n_0+1}) \tilde\Psi_\eps.
\end{align*}
By \eqref{bounds.landscape.rescaled} and the rescaled version of \eqref{l.localization.2}, the error $R_{n_0,\eps}$ satisfies
\begin{align*}
\|\tilde R_{n_0,\eps} \| \lesssim  \eps (\|F_{n_0,\eps}\| + \eps\|F_{n_0,\eps}\| + \eps \| F_{n_0-1,\eps}\|).
\end{align*}
By arguing again as in the case $n=1$ via the resolvent estimate in \eqref{resolvent.distance} we get that
$$
\| F_{n_0+1, \eps} \| \lesssim  \eps (\|F_{n_0,\eps}\| + \eps \| F_{n_0-1,\eps}\|).
$$
From this, the induction hypothesis and the definition of the functions $F_{n, \eps}$ yield \eqref{l.localization.step.1.b} for $n= n_0+1$. 
This establishes \eqref{l.localization.step.1.b} and concludes the proof of the theorem.
\end{proof}

\subsection{The Hamiltonian $H_\eps$ in curvilinear coordinates}
Before giving the proof of the Theorem \ref{t.main.easy}, we need some technical lemmas allowing to express the Hamiltonian $H_\eps$ of \eqref{spectral.intro} in the local curvilinear  coordinates $(\xi, s)$ introduced in \eqref{local.coordinates}. Since, as already shown in Proposition \ref{t.localization}, the eigenfunctions are localized at scale $\eps$ from the boundary, it we mostly work with the rescaled coordinates around any point $(\xi^*, 0)$ of the boundary $\partial\Omega$ 
 \begin{equation}\label{rescaling.local}
(\mu, \theta) \in \T_{\frac 1 \eps} \times (0, \frac \delta \eps), \ \ \ \ \  \begin{cases}
 \mu = \frac{1}{\eps} s\\
 \theta= \frac{1}{\eps}( \xi - \xi^*).
 \end{cases}
 \end{equation}
 
We remark that the Lam\'e coefficients associated to the change of coordinates \eqref{local.coordinates} equal
\begin{align*}
 h_\xi := |\partial_\xi x(\xi, s)| = |\vT - s(\vN(\xi))'|, \ \ \ \ \ h_s:= |\partial_s x(\xi, s)| = |\vN(\xi)|=1,
\end{align*}
so that by the standard formula $(\vN(\xi))'= -\kappa \vT$, we obtain
\begin{align}\label{lame.coeff}
 h_\xi := (1 + \kappa(\xi) s), \ \ \ \ \ h_s:=1.
\end{align}
These, in particular, allow us to rewrite the gradient $\nabla$ in the Cartesian variables $x\in \R^2$ as
\begin{align}\label{gradient.curvilinear}
 \nabla g =\frac{1}{1+\kappa(\xi) s} \partial_\xi g \, \vT + \partial_s g \, \vN.
\end{align}
 \begin{lem}\label{l.local.magnetic.potential}
Let $a$ be as in \eqref{magnetic.potential} and the functions $\alpha, \kappa$ as in \eqref{functions.curvature}.
We consider their periodic extensions from $\T$ to the whole line $\xi \in \R$. Finally, let $\delta$ be as in \eqref{tubular.neighbourhood} and $\omega_\eps$ be as in  \eqref{quantities.corollary}. Then, there exist $\delta_1= \delta_1(\partial\Omega) < 2\delta$ and $C=C(\partial \Omega) < +\infty$ such that the function $\rho_\eps: \R^2 \to \R$
\begin{align}\label{multivalued.gauge.function}
\rho_\eps(\xi, s) := -\int_0^\xi \alpha(\tilde \xi) \d \tilde \xi + \frac 1 2 \alpha'(\xi) s^2 +\eps^2 \omega_\eps \xi, \ \ \ \ (\xi, s) \in  \R \times [0, \delta_1]
\end{align}
has (Cartesian) gradient $\nabla \rho_\eps$ that is 1-periodic in the variable $\xi$ and  satisfies for all $(\xi, s)\in \T \times  [0, \delta_1]$
\begin{align}\label{local.magnetic.potential}
\bigl | a(\xi, s) + \nabla\rho_\eps - \biggl(s - \frac{\kappa(\xi)}{2} s^2\biggr) \vT(\xi) + \frac 1 2 \biggl(3\alpha'(\xi) \kappa(\xi) +\alpha(\xi)\kappa'(\xi) \biggr) &s^2 \vN(\xi)\bigr|\\
&\leq C (s^3 + \eps^2).
\end{align}
Moreover, while $\rho_\eps$ is not periodic in $\xi$ and thus has to be defined on the entire slab $(\xi, s) \in \R \times [0, \delta_1]$, its exponential $e^{i\eps^{-2}\rho}$ is $1$-periodic in $\xi$ and thus is a well-defined change of gauge in $\T \times  [0, \delta_1]$. 
\end{lem}

\bigskip

\begin{lem}\label{l.local.hamiltonian}
Let  $\delta_1$ and $\rho$ be as in Lemma \ref{l.local.magnetic.potential}. Let $\xi^* \in \T$ be fixed and let us consider the  rescaled coordinates $(\mu, \theta)$ introduced in \eqref{rescaling.local}. Then, the Hamiltonian 
$$
\tilde H_\eps:= - (\nabla + i\eps^{-2}(a + \nabla \rho_\eps)) \cdot (\nabla + i\eps{-2}(a + \nabla \rho_\eps))
$$
may be written  in the set $\Omega_{\delta_1} \simeq \{(\theta, \mu)  \in \T_{\frac 1 \eps} \times [0, 2\delta_1] \}$ as 
$$
\tilde H^\eps =  \eps^{-2} ( H_0 + \eps H_{1,\eps} + \eps^2 H_{2,\eps})
$$ 
with
\begin{align*}
&H_0 = -\partial_\mu^2 -(\partial_\theta + i\mu)^2,\\
&H_{1,\eps}= - \kappa(\xi^* + \eps \theta) \partial_\mu + 2 \kappa(\xi^* +\eps\theta) \mu \partial_\theta^2 + i 2 \kappa(\xi^*+\eps \theta) \mu^2 \partial_\theta \\
&\quad\quad\quad\quad - i \bigl(3\alpha'(\xi^* +\eps\theta) \kappa(\xi^* +\eps\theta) +\alpha(\xi^* +\eps\theta)\kappa'(\xi^* +\eps\theta) \bigr)\bigl(\mu + \mu^2 \partial_\mu \bigr) - \kappa (\xi^* +\eps\theta) \mu^3,
\end{align*}
and $H_{2,\eps}: H^2(\Omega_{\delta_1})\cap H^1_0(\Omega_{\delta_1}) \to L^2(\Omega_{\delta_1})$ satisfying for every $\rho \in H^2(\Omega_{\delta_1})\cap H^1_0(\Omega_{\delta_1})$
\begin{align}\label{boundedness.H.2}
 \|H_{2,\eps} \rho \|_{L^2(\Omega_{\delta_1})}& \leq C(\Omega) \biggl(\|(1 +\mu) (\partial_\theta + i \mu)^2 \rho \|_{L^2(\Omega_{\delta_1})}+ \|(1 +\mu)^3 (\partial_\theta + i \mu) \rho \|_{L^2(\Omega_{\delta_1})}\notag \\
 &\quad \quad \quad\quad\quad\quad+ \|(1 + \mu)^3\partial_\mu \rho\|_{L^2(\Omega_{\delta_1})}+ \|(1 + \mu)^4 \rho \|_{L^2(\Omega_{\delta_1})} + \eps^2\|(1 + \mu)^6 \rho \|_{L^2(\Omega_{\delta_1})}\biggr).
\end{align}
\end{lem}

\begin{proof}[Proof of Lemma \ref{l.local.magnetic.potential}]
By rewriting the equation \eqref{magnetic.potential} for in the local coordinates $(\xi, s)$ (see also \eqref{lame.coeff}-\eqref{gradient.curvilinear}),
the Implicit Function Theorem implies that there exists $\delta_1 \leq 2\delta$ such that for all $s \leq \delta_1$ and $\xi \in T$, it holds
\begin{align*}
\phi(\xi, s) = -\bigl(\alpha(\xi)s + \beta(\xi)s^2 + \gamma(\xi)s^3\bigr) + \epsilon(\xi,s),
\end{align*}
with the function $\alpha$ defined as in \eqref{functions.curvature} and the remaining coefficients $\beta, \gamma$ satisfying for a constant $C < +\infty$
\begin{align}\label{coefficients.phi}
 \begin{cases}
 \beta = \frac{1}{2}(1 - \kappa\alpha)\\
 \gamma = \frac{1}{6}(2\kappa^2 \alpha - \alpha''-\kappa)\\
  |\epsilon(\xi, s)| \leq C s^4.
 \end{cases}
\end{align}
 Note that since $\partial\Omega$ is assumed to be $C^4$, all the above quantities are well-defined. Furthermore, by \eqref{magnetic.potential}, the vector field $ a$ may be written in the local coordinates $(\xi, s)$ as
\begin{align*}
a(\xi, s) = \bigl(\alpha's - (\kappa\alpha' - \beta')s^2 \bigr)\vN(\xi) - \bigl(\alpha + 2\beta s + 3\gamma s^2\bigr)\vT(\xi) + V(\xi, s), \ \ \ |V(\xi, s)|\leq Cs^3. 
\end{align*}
By using the definition of $\rho_\eps$ in \eqref{multivalued.gauge.function}, \eqref{gradient.curvilinear} and the relations \eqref{coefficients.phi}, it is an easy computation to show that $a + \nabla \phi$ satisfies \eqref{local.magnetic.potential}. 

\bigskip

To conclude the proof of the lemma, it remains to show that the definition of $C_\eps$ in \eqref{multivalued.gauge.function} implies that for every $s \in [0, \delta_1]$  and $m\in \Z$ we have
\begin{align}\label{periodicity}
e^{i \eps^{-2}\rho_\eps(s, m)} = e^{i \eps^{-2}\rho_\eps(s, 0)}.
\end{align}
Since $\alpha$ is periodic and the function $\int_0^\xi \alpha(x) \, \d x + \eps^2\omega_\eps \xi$ is linear in the variable $\xi$, it suffices to prove that
$$
\eps^{-2}\int_0^1 \alpha(x) \d x + \omega_\eps \in 2\pi \Z.
$$
This, in turn, immediately follows by the definition \eqref{quantities.corollary} of $\omega_\eps$ and the fact that by definition \eqref{functions.curvature} of $\alpha$,
the divergence theorem and the equation in \eqref{magnetic.potential} it holds
$$
\int_0^1 \alpha(x) \d x = - \int_{\Omega} \Delta \phi = |\Omega|.
$$
\end{proof}

\begin{proof}[Proof of Lemma \ref{l.local.hamiltonian}] Since the proof of this lemma follows by the results of Lemma \ref{l.local.magnetic.potential} and standard computations for change of coordinates in differential operators, we skip its proof and give below only the main steps:

\begin{itemize}
 \item We rewrite the gradient $\nabla$ in curvilinear coordinates as in \eqref{gradient.curvilinear} and use
 \eqref{local.magnetic.potential} for $a + \nabla\rho$.

 \item We then change coordinates according to the rescaling \eqref{rescaling.local} so that $(\partial_\xi, \partial_s) \mapsto \eps^{-1}(\partial_\theta, \partial_\mu)$ and compare the terms obtained at each order of $\eps$. The orders $1$ and $\eps$ correspond to  $H_0$ and $H_{1,\eps}$.  
 
 \item To conclude the proof of the lemma it remains to show that the remainder $\eps^2 H_{2,\eps}:= \eps^{2}\tilde H_\eps- (H_o + \eps H_1)$ satisfies estimate
\eqref{boundedness.H.2}. This easily follows thanks to the assumptions on the regularity and compactness of $\partial\Omega$ that allow to uniformly bound $\alpha, \kappa$ and their derivatives up to the fourth order.
\end{itemize}
\end{proof}

\subsection{Proof of Theorem \ref{t.main.easy}}
We begin this section by showing that by Proposition \ref{t.localization} we may assume that $\Psi_\eps$ is supported in a neighbourhood of $\partial \Omega$ of order 1. This, in particular,  allows us to work in the local coordinates $(\xi, s)$ of \eqref{local.coordinates} without facing problems of well-definiteness. 

\smallskip

Let $\delta_1>0$ be as in Lemma \ref{l.local.magnetic.potential} and Lemma \ref{l.local.hamiltonian}. Let $\eta$ be a smooth cut-off function for $\Omega_{\frac{\delta_1}{2}}$ in $\Omega_{\delta_1}$. Then, the function $\eta\Psi_\eps$ satisfies
\begin{align}\label{cut.problem}
\begin{cases}
H_\eps( \eta\Psi_\eps) = \lambda_\eps \eta\Psi_\eps + (\nabla + i\eps^{-2}a)\cdot \bigl( \nabla\eta \Psi_\eps\bigr) \ \ \ \ &\text{in $\Omega$}\\
\eta\Psi_\eps = 0 \ \ \ \ \ &\text{ on $\partial\Omega$}.
\end{cases}
\end{align}
By Theorem \ref{localization}, we may bound for every $n\in \N$
\begin{align}
 \|(1-\eta)\Psi_\eps\|_{L^2(\Omega)} +  \|(\nabla + i \eps^{-2}a)\cdot \bigl(\nabla\eta \Psi_\eps\bigr) \|_{L^2(\Omega)} \lesssim C(n) \eps^{n},
 \end{align}
which implies that the modified function $\eta\Psi_\eps$ is  close in $L^2(\Omega)$ to the original function $\Psi_\eps$  and satisfies \eqref{spectral.intro} up to an error $\eps^n$, with  $n\in \N$ arbitrary. Throughout the proof of the theorem, we may thus assume that $\Psi_\eps$ is supported in $\Omega_{\delta_1}$. We may also perform the change of gauge $\eta\Psi_\eps \mapsto e^{i\eps^{-2}\rho_\eps}\eta\Psi_\eps$, with $\rho_\eps$ introduced in Lemma \ref{l.local.magnetic.potential}. This new function thus satisfies \eqref{cut.problem} with the magnetic potential $A_\eps$ substituted by $\eps^{-2}(a + \nabla \rho)$ and an error term $R_\eps$ which may be again bounded by an arbitrary power of $\eps$. As long as no ambiguity occurs, throughout this section we keep the same notation $\Psi_\eps$ for the previous approximated eigenfunction $e^{i\eps^{-2}\rho_\eps}\eta\Psi_\eps$. In view of this, we also redefine
the function $\Psi_{\text{flat},\eps}$ as
$$
\Psi_{\text{flat},\eps}:= e^{i (- k_{1} \frac \xi \eps + B_1(k_1) \int_0^\xi \kappa(y) \d y + \rho) }H_j(k_{1,\eps}, \mu), \ \ \rho \in [0, 2\pi)
$$
so that the proof of Theorem  \ref{t.main.easy} reduces to showing \eqref{closeness.to.flat.theorem.easy}  with these new definitions of $\Psi_\eps$ and $\Psi_{\text{flat}, \eps}$.

\bigskip

For $\eps > 0$ we define
\begin{align}\label{def.m.eps}
m_\eps:= \max_{\xi \in \T}\biggl( \frac{1}{2}\int_{ d(\tilde \xi; \xi)< \eps} \int_{\R_+} | \Psi_\eps(\tilde\xi, s)|^2 \d\tilde \xi \, \d s\biggr)^{\frac 1 2}.
\end{align}\bigskip

The next two propositions give a more quantitative information on the convergence of the eigenfunctions. As becomes apparent in the statement
of the next result, since we did not assume any quantitative information on the convergence of the eigenvalues $\eps^2\lambda_\eps$ to $\lambda$, the price to pay to quantify the convergence of the eigenfunctions is a more implicit definition of their first-order approximation (which we call $\tilde \Psi_{\text{flat},\eps}$ in comparison with $\Psi_{\text{flat},\eps}$).

\smallskip

As explained in Subsection \ref{subsection.main.ideas}, Proposition \ref{p.main} states that for each $\xi^* \in \T$, corresponding to a point of the boundary $\partial \Omega$, the function $\Psi_\eps$ at scales $\eps$ around such point may be approximated in a weighted $L^2$-norm by a multiple of the eigenfunction for the half-plane. Moreover, the error of the approximation around a point $\bar \xi$ of the boundary grows at most linearly  with the distance $d(\bar \xi; \xi^*)$. Similarly, we may find a higher-order approximation for the behaviour of $\Psi_\eps$ by multiplying the eigenfunction of the half plane by a linear function in the angular variable. In this case, the error grows quadratically with the distance from the point $\xi^*$.

\begin{prop}\label{p.main}
Let $\eps_0$ and $k_{1,\eps}$ be as in \eqref{roots.branches}. For each $ \xi^* \in \T$, there exist $A_{\eps}(\xi^*) \in \C$  with $|A_\eps(\xi^*)| \leq 1$ such that for every $\bar \xi \in \T$ with $d(\xi^*+ \bar \xi; \bar \xi) < \frac 1 4 $
\begin{align}\label{first.order.error}
\biggl(\fint_{d(\xi; \bar \xi) < \eps} \int (1+ \frac s \eps)^{-6} |\frac{\sqrt \eps}{m_\eps} \Psi_\eps(\xi^* + \xi, s) - \frac{1}{\sqrt \eps} A_{\eps}(\xi^*) H_1(k_{1,\eps}, \frac s \eps) e^{-i k_{1,\eps} \frac{\xi}{\eps}}|^2 \d\xi \, \d s \biggr)^{\frac 1 2}\lesssim d(\xi^*+ \bar \xi; \xi^*) + \eps
\end{align}
and
\begin{align}\label{second.order.error}
\biggl(\fint_{d(\xi; \bar \xi) < \eps} \int &(1+ \frac s \eps)^{-12} |\frac{\sqrt \eps}{m_\eps} \Psi_\eps(\xi^* + \xi, s)\notag  \\
& \quad\quad -\frac{1}{\sqrt \eps} A_{\eps}(\xi^*)( 1 + i B_l(k_{\eps,1}) \kappa(\xi^*) \xi) H_1(k_{1,\eps}, \frac s \eps) e^{-i k_{1,\eps}\frac \xi \eps}|^2 \d\xi \, \d s \biggr)^{\frac 1 2}\lesssim d(\xi^*+ \bar \xi;  \xi^*)^2 + \eps.
\end{align}

\end{prop}

\bigskip

\begin{lem}\label{p.main.2} Let $\eps_0$ and $k_{1,\eps}$ be as in the previous lemma. For every $\eps < \eps_0$ there exists  a function $F_\eps\in C^0( \R; \C)$ such that
\begin{align*}
\tilde \Psi_{\text{flat},\eps}(\xi, s):= F_\eps(\xi) e^{-i k_{1,\eps} \frac \xi \eps}H_1(k_{1,\eps}, \frac s \eps),
\end{align*}
satisfies
\begin{align}\label{closeness.to.flat.improved.easy}
\sup_{\xi^* \in \T}\bigl(\fint_{d(\xi;\xi^*)< \eps} \int_{0}^{+\infty}|\frac{\sqrt \eps}{m_\eps}\Psi_\eps(\xi, s) -
\frac{1}{\sqrt{\eps}}\tilde \Psi_{\text{flat}, \eps}(\xi, s)|^2 \d\xi \d s \bigr)^{\frac 1 2} \lesssim \sqrt\eps.
\end{align}
Furthermore, for every $\eps_j \to 0$, there exists a subsequence and $\rho \in [0, 2\pi)$ such that 
\begin{align}\label{closeness.to.amplitude}
F_\eps \to e^{i B_1(k_1) \int_0^\xi \kappa(x) \d x + \rho} \ \ \ \ \text{uniformly on  $\T$.}
\end{align}
\end{lem}

We begin by showing how to prove Theorem \ref{t.main.easy} from the above statements.
\begin{proof}[Proof of Theorem \ref{t.main.easy}]
We use \eqref{closeness.to.flat.improved.easy} to argue that 
\begin{align}\label{uniform.norm}
\frac{m_\eps}{\sqrt\eps} \to 1.
\end{align}
On the one hand, the normalization assumption $\|\Psi_\eps\|_{L^2}=1$ implies that
\begin{align}\label{first.lower.bound}
 m_\eps \geq \sqrt \eps + o(1).
\end{align}
On the other hand, by the triangle inequality and the normalization
of the functions $H_1(k_{\eps,1}, \cdot)$ (see line after \eqref{eigenvalue.oscillator}) in the rescaled variable $\mu= \frac s \eps$, we may use   \eqref{closeness.to.flat.improved.easy} to bound for any $\xi^* \in \T$
\begin{align}\label{hom.norm}
\bigl(\int_{d(\xi ; \xi^*)< \eps} \int |\Psi_\eps|^2 \bigr)^{\frac 1 2} 
\geq  m_\eps \biggl( \bigl( \fint_{d(\xi ; \xi^*)< \eps} |F_{\eps}(\xi)|^2 \d\xi  \bigr)^{\frac 1 2} - C{\sqrt \eps}\biggr).
\end{align}
Using again the normalization assumption  $\|\Psi_\eps\|_{L^2}=1$, we  get that
\begin{align*}
\frac{m_\eps}{\sqrt \eps}  \leq \biggl( \bigl( \fint_{d(\xi ;\xi^*)< \eps} |F_{\eps}(\xi)|^2 \d\xi  \bigr)^{\frac 1 2} - C{\sqrt \eps}\biggr)^{-1} + o(1).
\end{align*}
Since by \eqref{closeness.to.amplitude}, we have  that $|F_{\eps}| \to 1$ uniformly, the term  on the right-hand side converges to 1. This, together with \eqref{first.lower.bound}, implies \eqref{uniform.norm}. We remark that \eqref{hom.norm} also implies that 
\begin{align}
\sup_{\xi^* \in \T}\biggl( \fint_{d(\xi ;\xi^*) < \eps} \int_{\R_+} | \Psi_\eps(\xi, s)|^2 \d\xi \, \d s\biggr)^{\frac 1 2} \to 1,
\end{align}
i.e. for $\eps$ small the $L^2$-norm of the eigenfunction at scale $\eps$ is distributed almost uniformly along the boundary.

\bigskip

Equipped with \eqref{uniform.norm}, we turn to the main estimate of Theorem \ref{t.main.easy}: By combining \eqref{closeness.to.flat.improved.easy}  with \eqref{uniform.norm} and \eqref{hom.norm}, we have that
\begin{align}\label{t.main.1}
\bigl(\fint_{d(\xi ;\xi^*)< \eps} \int  |\Psi_\eps(\xi, s) - \frac{1}{\sqrt{\eps}}\tilde\Psi_{\text{flat}, \eps}(\xi, s)|^2 \d\xi \, \d s \bigr)^{\frac 1 2} \to 0.
\end{align}
Hence, by the triangle inequality we conclude the proof of Theorem \ref{t.main.easy} provided we argue that also
\begin{align}
\bigl(\fint_{d(\xi ;\xi^*)< \eps} \int |\Psi_{\text{flat},\eps}(\xi, \eps\mu) - \tilde \Psi_{\text{flat}, \eps}(\xi, \eps \mu)|^2 \d\xi \, \d \mu \bigr)^{\frac 1 2} \to 0,
\end{align}
where we selected any sequence $\eps_j \downarrow 0$ such that $F_{\eps}$ converges as in \eqref{closeness.to.amplitude} to $F$ with a fixed phase $\rho$.  This limit easily follows by combining the triangle inequality with \eqref{closeness.to.amplitude} of Proposition \ref{p.main}.
\end{proof}

Before giving the argument for Proposition \ref{p.main} we introduce the following notation: For $m \in \N$ fixed, we define the norm  $\iii{\cdot}_{m}$ acting on any function $g: \R \times \R_+ \to \C$ as
\begin{align}\label{norm.3.bars.easy}
\iii{g}_{m}:= \sup_{\Theta \in \R}( 1+ |{\Theta}|)^{-m}\bigl(\fint_{|\theta - \Theta | < 1} \int_{\R_+} |g(\theta, \mu)|^2 \d\theta \, \d\mu \bigr)^{\frac 1 2}.
\end{align}

\begin{proof}[Proof of Lemma \ref{p.main.2}] \, We start by appealing to Proposition \ref{p.main} to argue that there exists $F_\eps \in C^0(\T; \C)$ such that for all $\xi^* \in \T$
\begin{align}\label{p.main.2.1}
\biggl(\fint_{d(\xi ;\xi^*) < \eps} \int_0^{+\infty} (1+ (\frac s \eps)^6)^{-2} |\frac{\sqrt\eps}{m_\eps} \Psi_\eps(\xi , \eps \mu) -\frac{1}{\sqrt \eps} F_{\eps}(\xi) H_1(k_{1,\eps}, \frac s \eps) e^{-i k_{1,\eps}\frac{\xi}{\eps}}|^2 \biggr)^{\frac 1 2} \lesssim  \eps.
\end{align}

\bigskip

Let $ \tilde \Psi_\eps(\xi^*,\mu, \theta)$ be as in \eqref{rescaled.zero.order} in the proof of Proposition \ref{p.main}. For  $\xi, \xi^* \in \T$, let us define
 \begin{align*}
f_{ \eps}(\xi, \xi^*) := \fint_{d(\eps \theta ; \xi) < \eps} \int_0^{+\infty} \tilde \Psi_\eps(\xi^*,\mu, \theta) e^{i k_{1,\eps} \theta} H_1(k_{1,\eps}, \mu) \d \mu \d\theta.
\end{align*} 
By the assumption $\|H_1(k_{1,\eps}, \cdot) \|_{L^2(\R_+)}=1$,  Cauchy-Schwarz's inequality and \eqref{first.order.error} we have that for all $\xi^*, \xi \in \T$ with $d(\xi^*+ \bar \xi; \bar \xi)< \frac 1 4$
$$
| f_{\eps}(\xi, \xi^*) - A_{\eps}(\xi^*)| \leq \bigl(  \int_0^{+\infty} (1+ \mu^6)^{2}|H_1(k_{1,\eps}, \mu)|^2 \d \mu  \bigr)^{\frac 1 2}( d(\xi + \xi^* ;\xi^*)+ \eps ).
$$
Since $k_{1,\eps} \to k_1$, the decay of  $H_1(k_{1,\eps}, \cdot)$ at infinity (see Lemma \ref{l.basic.flat})  implies that
\begin{equation}
\begin{aligned}\label{distance.A.1}
| f_{\eps}(\xi, \xi^*) - A_{\eps}(\xi^*)| \lesssim d(\xi + \xi^* ;\xi^*)+ \eps.
\end{aligned}
\end{equation}
Similarly, this time thanks to \eqref{second.order.error}, we also have that
\begin{equation*}
\begin{aligned}
| f_{\eps}(\xi, \xi^*) - A_{\eps}(\xi^*) ( 1+ i \xi B_{1}(k_{1,\eps}) \kappa(\xi_*) )| \lesssim  d(\xi + \xi^* ;\xi^*)^2 + \eps.
\end{aligned}
\end{equation*}
These yield, by using \eqref{distance.A.1} with $\xi =0$, that
\begin{equation}
\begin{aligned}\label{eps.differential.1}
&|f_\eps(\xi, \xi^*) - f_\eps(0, \xi^*)| \lesssim d(\xi + \xi^* ;\xi) + \eps,\\
&| f_{\eps}(\xi, \xi*) - f_{\eps}(0, \xi^*) ( 1+ i \xi B_{1}(k_{1,\eps}) \kappa(\xi_*) )| \lesssim d(\xi + \xi^* ;\xi^*)^2 + \eps .
\end{aligned}
\end{equation}
Since by a simple change of variables we rewrite
\begin{align}\label{change.phase}
f_{\eps}(\xi, \xi^*) = e^{-i \frac{\xi^*}{\eps} k_{1,\eps}} f_{l,\eps}(\xi + \xi^*, 0).
\end{align}
This and \eqref{eps.differential.1} yield that the function 
$$
F_{\eps}(\cdot):= f_{\eps}( \xi, 0) = \fint_{d(\eps\theta; \xi) < \eps} \int_0^{+\infty} \tilde \Psi_\eps(0 ;\mu, \theta) e^{i k_{1,\eps} \theta} H_1(k_{1,\eps}, \mu) \d \mu \d\theta.
$$ 
satisfies
\begin{align}\label{eps.differential}
| F_{\eps}(\xi) - F_{\eps}(\xi^*)| \lesssim d(\xi; \xi^*) + \eps, \ \ \ \ | F_{\eps}(\xi) - F_{\eps}(\xi^*) ( 1+ i \xi B_{1}(k_{1,\eps}) \kappa(\xi_*) )| \lesssim d(\xi; \xi^*)^2 + \eps
\end{align}
 for every $\xi, \xi^* \in \T$ with $d(\xi^*; \bar \xi)< \frac 1 4$. By \eqref{second.order.error} with $\xi^*\in \T$ and $\bar \xi =0$ and again \eqref{distance.A.1} with $\xi = 0$ we also get
\begin{align*}
\biggl(\fint_{d(\eps\theta; 0) < \eps} \int_0^{+\infty} (1+ \mu^6)^{-2} |\tilde \Psi_\eps(\xi^*, \theta,\mu) - f_{\eps}(0, \xi^*)H_1(k_{1,\eps}, \mu)e^{-i k_{1,\eps}\theta}|^2 \biggr)^{\frac 1 2} \lesssim \eps.
\end{align*}
Since, by the definition of $F_\eps$ above we may write $f_{\eps}(0, \xi^*)=  e^{-i \frac{\xi^*}{\eps} k_{1,\eps} }F_{\eps}(\xi^*)$, we obtain
\begin{align*}
\biggl(\fint_{d(\eps\theta; 0) < \eps} \int_0^{+\infty}  (1+ \mu^6)^{-2} |\tilde \Psi_\eps(\xi^*, \theta,\mu) - F_{\eps}(\xi^*) H_1(k_{1,\eps}, \mu) e^{-i k_{1,\eps}\frac{(\xi^* + \eps \theta)}{\eps}}|^2 \biggr)^{\frac 1 2} \lesssim  \eps.
\end{align*}
By \eqref{eps.differential} also
\begin{align*}
\biggl(\fint_{d(\eps\theta; 0) < \eps} \int_0^{+\infty}  (1+ \mu^6)^{-2} |\tilde \Psi_\eps(\xi^*, \theta,\mu) -F_{\eps}(\xi^*+ \eps \theta) H_1(k_{1,\eps}, \mu) e^{-i k_{1,\eps}\frac{(\xi^* + \eps \theta)}{\eps}}|^2 \biggr)^{\frac 1 2} \lesssim  \eps,
\end{align*}
so that if we use the definition \eqref{rescaled.zero.order} of $\tilde\Psi_\eps$ and change variables $\xi = \xi^* + \eps \theta$ we reduce to \eqref{p.main.2.1}.

\bigskip

Equipped with \eqref{p.main.2.1}, we argue how to upgrade this estimate to  \eqref{closeness.to.flat.improved.easy}: Let any $\frac {23}{24} \leq \alpha < 1$ fixed. Then
\begin{align}\label{closeness.to.flat.improved.1}
\bigl(\fint_{d(\xi ; \xi^*) < \eps} \int_{0}^{\eps^{\alpha}} |\frac{\sqrt \eps}{m_\eps}\Psi_\eps(\xi, s) -
\frac{1}{\sqrt{\eps}}\tilde \Psi_{\text{flat}, \eps}(\xi, s)|^2 \d\xi \d s \bigr)^{\frac 1 2}\stackrel{ \eqref{p.main.2.1}}{\lesssim} \sqrt\eps.
\end{align}

\bigskip

To conclude \eqref{closeness.to.flat.improved.easy} it remains to prove that also
\begin{align}\label{closeness.to.flat.improved.external}
\bigl(\fint_{d(\xi ; \xi^*) < \eps} \int_{\eps^{\alpha}}^{+\infty}(1 + (\frac{s}{\eps})^6)^{-2} |\frac{\sqrt \eps}{m_\eps}\Psi_\eps(\xi, s) -
\frac{1}{\sqrt{\eps}}\tilde \Psi_{\text{flat}, \eps}(\xi, s)|^2 \d\xi \d s \bigr)^{\frac 1 2} \lesssim \sqrt\eps.
\end{align}
Since by definition of $m_\eps$, it holds $m_\eps \geq \sqrt \eps$ (see \eqref{first.lower.bound}), Proposition \ref{t.localization} allow us to bound for every $n \in \N$
\begin{align}
 \bigl(\fint_{d(\xi ; \xi^*) < \eps} \int_{\eps^{\alpha}}^{+\infty} |\frac{\sqrt \eps}{m_\eps}\Psi_\eps(\xi, s)|^2 \d\xi \d s \bigr)^{\frac 1 2} \leq \eps^{-\alpha n-\frac 1 2} \|d_{\partial\Omega}^n \Psi_\eps\|_{L^2}\lesssim C(n) \eps^{(1-\alpha)n - \frac 1 2}.
\end{align}
Similarly, for all $n\in \N$
\begin{align*}
 \bigl(\fint_{d(\xi ; \xi^*) < \eps} \int_{\eps^{\alpha}}^{+\infty} |\frac{1}{\sqrt \eps} \Psi_{\text{flat}, \eps}(\xi, s)|^2 \d\xi \d s\bigr)^{\frac 1 2}&= \bigl(\fint_{|\xi-\xi^*|< \eps} \int_{\mu \geq \eps^{\alpha - 1}} |\Psi_{\text{flat}, \eps}(\xi, \eps \mu)|^2 \d\xi \d \mu \bigr)^{\frac 1 2} \\
 &\lesssim  \biggl( \int_{\eps^{\alpha-1}}^{+\infty} |H_1(k_1, \mu)|^2\d \mu \biggr)^{\frac 1 2} \stackrel{\text{Lemma \ref{l.basic.flat}}}{\lesssim} \eps^{(1-\alpha)n}.
\end{align*}
Since $\alpha < 1$, we may choose $n$ big enough in the previous inequality so that these imply \eqref{closeness.to.flat.improved.external} by the triangle inequality. It remains to combine \eqref{closeness.to.flat.improved.external} with \eqref{closeness.to.flat.improved.1}  to establish \eqref{closeness.to.flat.improved.easy} of Lemma \ref{p.main.2}.

\bigskip

To conclude the proof of this lemma, it remains to show that $F_\eps$ constructed above converges uniformly to the function $F$ defined in \eqref{closeness.to.amplitude}. Let us fix $\{\xi_l\}_{l=1}^{9} \subset \T$ such that the sets $I_l= \{ \xi \in \T \, \colon \, d( \xi_j; \xi) < \frac 1 8\}$ with $j=1, \cdots, 9$ provide a covering for $\T$. Let $j=1$. By the first inequality in 
\eqref{eps.differential}, we infer that $|F_\eps(\xi)| \lesssim 1$ for every $\xi \in \T$  and it is equicontinuous (up to $\eps$) on $I_1$. We may thus apply Ascoli-Arzel\'a's theorem \footnote{Inequality  \eqref{eps.differential} provides an equicontinuity estimate only up to a constant vanishing when $\eps \to 0$. This allows nonetheless for Ascoli-Arzel\'a's theorem to hold.} and, up to subsequences, infer that $F_\eps$ converges uniformly in $I_1$ to a function $F$. In addition, by passing to the limit $\eps \to 0$ in the second inequality of \eqref{eps.differential} and using that $k_{1,\eps } \to k_1$, we obtain for every $\xi,\xi^* \in \T$
$$
| F(\xi) - F(\xi^*) ( 1+ i \xi B_{1}(k_{1}) \kappa(\xi_*) )| \lesssim |\xi- \xi^*|^2.
$$
In other words, $F$ solves  the ODE
\begin{align*}
F' = i B_1(k_1) \kappa  F \ \ \ \ \text{in $I_1$},
\end{align*}
i.e.  for some $C \in \C$ and all $\xi \in I_1$
\begin{align}\label{def.limit.F}
F(\xi)=C \exp\bigl(i B_1(k_1) \int_0^\xi \kappa(x) \d x \bigr)
\end{align}
Since the same argument holds for all the other sets $I_l$, $l=2, \cdots, 9$ and they provide a covering for $\T$, we may infer that $F$ above is the uniform limit of $F_\eps$ for every $\xi \in \T$.
To conclude the proof of \eqref{closeness.to.amplitude} it remains to argue that 
\begin{align}\label{complex.sphere}
|C|=1.
\end{align}

\bigskip

We prove \eqref{complex.sphere} by appealing to  \eqref{closeness.to.flat.improved.easy}: For every $\eps$ we denote by $\xi^*_\eps$ one of the points where $m_\eps$ in \eqref{def.m.eps} is attained. Note that since the boundary $\partial\Omega$ is compact, and therefore the maximum of the $L^2$-norms  are taken over the torus $\T$, such $\xi^*_\eps$ exists. Hence, by \eqref{closeness.to.flat.improved.easy} with $\xi^* = \xi^*_\eps$, the triangle inequality and the definition of $m_\eps$, there exists a constant $C$ (independent from $\xi_\eps^*$) such that
\begin{align*}
1 - C\sqrt\eps \leq \bigl(\fint_{d(\xi;\xi^*_\eps)< \eps} \int_{0}^{+\infty}|\frac{1}{\sqrt{\eps}}\tilde \Psi_{\text{flat}, \eps}(\xi, s)|^2 \d\xi \d s \bigr)^{\frac 1 2} \leq 1 + C\sqrt\eps.
\end{align*}
By the definition of $\tilde \Psi_{\text{flat}, \eps}$ this turns into
\begin{align*}
1 - C\sqrt\eps \leq \bigl(\fint_{d(\xi;\xi^*_\eps)< \eps}|F_\eps(\xi)|^2 \d\xi \bigr)^{\frac 1 2} \leq 1 + C\sqrt\eps.
\end{align*}
Let $\{ \eps_j\}_{j \in \N}$ be any sequence such that $F_{\eps_j} \to F$ uniformly as above. Then, there exists another sequence $\beta_j \to 0$ such that
\begin{align*}
1 - (C\sqrt{\eps_j} + \beta_j) \leq |F(\xi_{\eps_j})| \leq 1 + C\sqrt{\eps_j} + \beta_j.
\end{align*}
Using that thanks to \eqref{def.limit.F} we have $|F(\xi_\eps)|=|F(0)|$, we may send $j \uparrow +\infty$ and conclude \eqref{complex.sphere} for the sequence $\{ \eps_j\}_{j\in \N}$. Since the same argument holds regardless of the converging sequence considered,  we infer \eqref{closeness.to.amplitude}. This concludes the proof of Lemma \ref{p.main.2}.
\end{proof}

\begin{proof}[Proof of Proposition \ref{p.main}] We divide the proof into the following steps:

\bigskip

\noindent {\em Step 1. Blow-up limit for $ \Psi_\eps$.\,}  For each $\eps>0$, let $m_\eps$ be as in \eqref{def.m.eps}. For a fixed point $\xi^* \in \T$, we define the rescaled functions
\begin{align}\label{rescaled.zero.order}
\tilde \Psi_\eps(\xi^*; \eps\theta, \eps\mu):= \frac{\eps}{m_\eps} \Psi(\xi^* + \eps \theta, \eps \mu).
\end{align}
The goal of this step is to show that for every sequence $\eps_j \to 0$  there exists $A=A(\xi^*) \in \C$ with $|A(\xi)| \leq 1$ such that for every $R> 0$
\begin{align}\label{convergence.psi.0}
\tilde\Psi_\eps \to A(\xi^*) H_1(k_1, \mu) e^{-i k_1 \theta}, \ \ \ \ \text{in $L^2( \{ |\theta| < R \} \times \R_+ )$}.
\end{align} 
Here, the value $k_1 \in \R$  is the one defined in \eqref{roots}. Note that this in particular implies that the limit class for $\tilde \Psi_{\eps}(\xi^*; \cdot, \cdot)$  are functions like the one above which only differ by the choice of the constant $A(\xi^*)$. With no loss of generality we give the proof for $\xi^*=0$. In order to keep our notation lean, we drop the argument $\xi^*$ in the notation for $\tilde\Psi_\eps$, $\Psi_0$ and $A$.
 
\bigskip

We begin the proof of this step by observing that by Lemma \ref{l.local.magnetic.potential} and Lemma \ref{l.local.hamiltonian}, the function $\tilde \Psi_\eps$ solves the boundary value problem
\begin{align}\label{rescaled.eigenvalue.problem}
\begin{cases}
(H_0- \eps^2\lambda_\eps) \tilde \Psi_\eps = \eps f_\eps \ \ \ &\text{ in $\T_{\frac 1 \eps} \times \R_+$}\\
\tilde\Psi(\xi, 0) =0 \ \ \ \ &\text{ on $\T_{\frac 1 \eps} \times \{ \mu = 0\}$}
\end{cases}
\end{align}
with 
\begin{align}\label{f.eps.first.step}
f_\eps(\mu, \theta) := H_{1,\eps}\tilde \Psi_\eps + \eps H_{2,\eps}\tilde \Psi_\eps.
\end{align}
By recalling the definition \eqref{norm.3.bars.easy} of the norms $\iii{\cdot }_m$, by construction of $\tilde \Psi_\eps$ and definition \eqref{def.m.eps} we have that the periodic extension of $\tilde \Psi_\eps$ to the half-plane $\R \times \R_+$ satisfies
\begin{align}\label{boundedness}
\iii{\tilde \Psi_\eps(\theta, \mu)}_{0}\leq 1.
\end{align}

\medskip

From \eqref{boundedness} and weak lower semicontinuity, we infer that, up to a subsequence $\eps_j\to 0$, there exists a function $\Psi_0 \in L^2_{loc}( \R \times \R_+)$ that satisfies
\begin{align}\label{boundedness.psi.0}
\iii{ \Psi_0}_{0} \lesssim 1,
\end{align}
and which is the weak limit of $\tilde \Psi_\eps$ in the sense of \eqref{convergence.psi.0}. We prove that
\begin{align}\label{characterization.psi.0}
\Psi_0 (\xi^*; \theta, \mu)= A(\xi^*) H_1(k_1, \mu) e^{-i k_1 \theta}.
\end{align}
 To keep a lean notation, we fix any subsequence  $\{ \eps_j\}_{j \in \N}$ and drop the lower index $j$ in the notation. We first show that  \eqref{boundedness} also implies that
\begin{align}\label{bounded.energy}
 \iii{\partial_\mu\tilde \Psi_\eps }_{0} + \iii{(\partial_\theta + i \mu)\tilde \Psi_\eps }_0 \lesssim 1, 
\end{align}
so that for every $R> 0$
\begin{equation}
\begin{aligned}\label{further.convergence}
(\partial_\theta + i\mu)\tilde\Psi_\eps(\theta, \mu) &\rightharpoonup (\partial_\theta + i\mu)\Psi_0(\theta, \mu), \ \ \ \ \ \partial_\mu\tilde \Psi_\eps(\theta, \mu) \rightharpoonup \partial_\mu\Psi_0(\theta, \mu)
\end{aligned}
\end{equation}
in $L^2( \{ |\theta| < R\} \times \R_+)$.
To show \eqref{further.convergence} we apply  Lemma \ref{l.energy.estimate} with $m=n =1$ to $\tilde\Psi_\eps$: Since this function solves \eqref{rescaled.eigenvalue.problem}, we infer that 
\begin{align}\label{step2.a}
\iii{(\partial_\theta + i\mu)\tilde\Psi_\eps}_{0} + \iii{\partial_\mu\tilde \Psi_\eps}_{0} \leq (\eps^2\lambda_\eps)^{\frac 1 2} \iii{\tilde\Psi_\eps}_{0} + \eps \iii{f_\eps}_{0} \stackrel{\eqref{distance.landau}-\eqref{boundedness}}{\lesssim} 1 +  \eps \iii{f_\eps}_{0}.
\end{align}
We show that also the second term on the right-hand side is bounded: By the inequality of Theorem \eqref{t.localization} rewritten in the local coordinates $(\mu, \theta)$ we get that for each $n\in \N$
\begin{align}\label{localization.rescaled}
\int_{\T_{\frac 1 \eps}} \int_{\R_+} \mu^n  |\tilde\Psi_\eps|^2  \leq C(n) \int_{\T_{\frac 1 \eps}} \int_{\R_+} |\tilde\Psi_\eps|^2 \stackrel{\eqref{boundedness}}{\lesssim}\eps^{-1}.
\end{align}
Similarly, we may rewrite in the rescaled coordinates $(\mu, \theta)$ also the second and third term on the left-hand side of \eqref{localization}.
By appealing to Lemma \ref{l.local.magnetic.potential} for $a+ \nabla \rho_\eps$ we conclude that \eqref{localization.rescaled} also holds if $\tilde \Psi_\eps$ is replaced by the terms $\partial_\mu\tilde\Psi_\eps$, $\partial_\theta\tilde\Psi_\eps, \partial_\mu^2\tilde\Psi_\eps$ and $\partial_\theta^2\tilde\Psi_\eps$. We stress that it is crucial that, as shown before Step 1, we reduced to consider $\Psi_\eps$ supported only on the set $\Omega_{\delta_1}$. Thanks to  the definition \eqref{f.eps.first.step} of $f_\eps$ and Lemma \ref{l.local.hamiltonian}, estimate \eqref{localization.rescaled} and its analogue for the higher derivatives of $\tilde\Psi_\eps$ imply that $\eps \iii{f_\eps} \lesssim 1$. By inserting this into \eqref{step2.a}  we infer \eqref{bounded.energy}.

\bigskip

Equipped with \eqref{boundedness}, \eqref{bounded.energy} and \eqref{further.convergence}, we now turn to prove identity \eqref{characterization.psi.0}: Ppassing to the limit $\eps \to 0$ in \eqref{rescaled.eigenvalue.problem}, the function $\Psi_0$ solves 
\begin{align*}
\begin{cases}
(H_0 -\lambda) \Psi_0 = 0 \ \ \ \ \ &\text{ in $\R \times \R_+$}\\
\Psi_0 =0 \ \ \ \ \ &\text{ on $\R \times \{ \mu = 0\}$}.
\end{cases}
\end{align*}
This, together with \eqref{boundedness.psi.0}  and \eqref{further.convergence}, allows us to apply Lemma \ref{l.identification} with $M=1$ and $n=0$ and obtain \eqref{characterization.psi.0}. We stress that by lower semicontinuity and \eqref{boundedness} it immediately follows that $|A(\xi)| \leq 1$.

\bigskip

We conclude this step by observing that \eqref{convergence.psi.0} and \eqref{bounded.energy} imply also that (up to supsequences) for every $R> 0$
\begin{equation}\label{strong.convergence.L2}
\tilde \Psi_\eps \to \Psi_0 \ \ \ \text{ in $L^2(\{ |\theta| < R \} \times \R_+)$}.
\end{equation}
This may be seen by observing that \eqref{bounded.energy} implies that, if $\eta_R$ is a cut-off for $\{ |\theta| < R\}$ in $\{ |\theta | < 2R\}$, the Fourier transform of $\eta_R \tilde \Psi_\eps$ is uniformly bounded in $H^1(\R \times \R_+)$ and thus admits a strong limit in $L^2(\R \times \R_+)$. This, Plancherel's identity and \eqref{convergence.psi.0} immediately yield \eqref{strong.convergence.L2}.

\bigskip

\noindent{\em Step 2. First-order correction.\,} As in Step 1, we consider the case $\xi^*=0$ and drop the argument $\xi^*$ in the notation for $\tilde\Psi_\eps$, $\Psi_0$ and $A$. As in the previous step, we extend periodically the function $\tilde \Psi_{\eps}$ from $\T_{\frac 1 \eps} \times \R_+$ to the whole half-plane $\R \times \R_+$.

\bigskip

In this step we approximate the function $\tilde\Psi_\eps$ of the previous step to a higher order: Let  $k_{1,\eps} \in \R$  be as in \eqref{roots.branches}. We define
\begin{align}\label{def.psi.0.eps}
\Psi_{0,\eps}:= A_{\eps} H_1(k_{1,\eps}, \mu) e^{-i k_{1,\eps} \theta},
\end{align}
with
\begin{align}\label{def.A.l.eps}
A_{\eps}= \fint_{|\theta|< 1} \int \tilde\Psi_\eps(\theta, \mu) e^{ik_{1,\eps}\theta} H_1(k_{1,\eps}, \mu) \d\mu \d\theta.
\end{align}
Note that by Cauchy-Schwarz, \eqref{boundedness} and the assumption $\int_0^{+\infty}|H_1(k_{1,\eps}, \mu)|^2 \d\mu =1$ , it follows that
$$ 
|A_\eps| \leq 1.
$$
We pick a (smooth) cut-off function $\eta_\eps=\eta_\eps(\theta)\in C^\infty_0(\R)$ of $[-\frac{ 1}{2\eps} ; \frac {1}{2\eps}]$ in $[-\frac{ 1}{ \eps} ; \frac{1}{\eps}]$. We claim that for every sequence $\{\eps_j\}_{j\in \N}$ such that $\tilde\Psi_\eps$ converges as in Step 1 for some $A \in \C$,  then also
\begin{align}\label{right.scaling.error}
 \iii{ \frac{(1 + \mu)^{-3}\eta_{\eps_j} (\tilde\Psi_{\eps_j} - \Psi_{0,\eps_j})}{\eps_j}}_{1} \lesssim 1, \ \ \ \ \eta_{\eps_j} \frac{\tilde\Psi_{\eps_j} - \Psi_{0,\eps_j}}{\eps_j} \rightharpoonup \Psi_1 \ \ \ \ \ \text{in $L^2_{\text{loc}}(\R \times \R_+)$,}
\end{align}
where
\begin{align}\label{first.order.approx}
\Psi_1(\theta, \mu) := A \bigl( i B_1(k_1) \theta H_1(k_1, \mu)  + W_1(\mu) ) e^{-i k_1 \theta}
\end{align}
with $B_1( \cdot)$ as defined in \eqref{quantities.corollary} and
\begin{align*}
\int_0^{+\infty} W_1(\mu) H_1(k_1, \mu) \d \mu = 0, \ \ \ \| W_1\|_{L^2((0, +\infty))} \lesssim 1.
\end{align*}
From now on, when no ambiguity occurs, we skip the index $j$ in the sequence $\eps_j$. We remark that throughout this step, the function $\eta \tilde \Psi_\eps(\theta, \mu)$ is supported only on  $\{ \theta \in \T_{\frac 1 \eps} \, \colon \, d(\eps \theta, 0) < \frac 1 \eps \}$ and therefore may be trivially extended to the whole line $\R$. 

\bigskip

We start the proof of this step by observing that the definitions of $A_\eps$ and $\Psi_{0,\eps}$ immediately imply that
\begin{align}\label{orthogonality}
\fint_{|\theta|< 1} \int (\tilde	\Psi_\eps - \Psi_{0,\eps}) e^{ik_{1,\eps}\theta} H_1(k_{1,\eps}, \mu) \d\mu \d\theta = 0.
\end{align}
We tackle the first estimate in \eqref{right.scaling.error}: Let us define the sequence
\begin{align*}
\beta_\eps:= \biggl(\int_{|\theta|< 1}\int (1 + \mu)^{-6}|\Psi_\eps - \Psi_{0,\eps}|^2 \biggr)^{\frac 1 2}
\end{align*}
and the functions 
$$
\Psi_{1,\eps}:= \eta_\eps \frac{\Psi_\eps - \Psi_{\eps,0}}{\beta_\eps \vee \eps} \in H^1_0(\R \times \R_+).
$$
By construction, it holds that $\Psi_{1,\eps}$ satisfies
\begin{align}\label{norm.psi.1.one}
\int_{|\theta|< 1}\int (1 +\mu)^{-6}|\Psi_{1,\eps}|^2 = 1 
\end{align}
and it solves 
\begin{align}\label{equation.psi.1.eps.cut}
\begin{cases}
(H_0 - \eps^2\lambda_{\eps}) \Psi_{1}^\eps= \frac{\eps}{\beta_\eps\vee \eps}(f_\eps + R_\eps) \ \ \ \ \ &\text{ in $\R \times \R_+$}\\
\Psi_1^\eps =0 \ \ \ \ \ &\text{ on $\R \times \{ \mu=0 \}$}.
\end{cases}
\end{align}
with 
\begin{align*}
f_\eps:= \eta_\eps (H_{1,\eps} \tilde \Psi_\eps + \eps H_{2,\eps}\tilde\Psi_\eps), \ \ \ R_\eps(\theta, \mu):= -2\frac{\eta'}{\eps} \partial_\theta  \tilde \Psi_\eps - \frac{\eta''}{\eps}\tilde \Psi_\eps
\end{align*}
We now argue that
\begin{align}\label{bounded.rhs.psi.1}
\iii{(1+ \mu)^{-3} f_\eps + R_\eps}_{0} \lesssim 1.
\end{align}
If this holds, indeed, we may appeal to Proposition \ref{proposition.growth} and \eqref{norm.psi.1.one} and infer that
\begin{align}\label{sublinearity.psi.beta}
\iii{(1+ \mu)^{-3} \Psi_{1,\eps}}_{1}\lesssim \frac{\eps}{\beta_\eps \vee \eps} + 1 \lesssim 1.
\end{align}

\bigskip

We show \eqref{bounded.rhs.psi.1} by treating separately the terms $\eps H_{2,\eps}\tilde\Psi_\eps$ and $H_{1,\eps}\tilde\Psi_\eps + R_\eps =f_\eps + R_\eps - \eps H_{2,\eps}\tilde\Psi_\eps$: For the first term we may exploit the factor $\eps$ in front of it to argue exactly as was done in Step 1 for the 
the right-hand side of \eqref{step2.a}. We infer from this that
\begin{align*}
 \iii{\eps H_{2,\eps}\tilde \Psi_\eps}_{0} \lesssim 1.
\end{align*}
We tackle the term $H_{1,\eps}\tilde\Psi_\eps + R_\eps$ in a different way: By using the explicit formulation for the operator $H_{1,\eps}$ (c.f. Lemma \ref{l.local.hamiltonian}), and observe that we may bound
\begin{align}
\iii{(1 + \mu)^{-3}(H_{1,\eps}\tilde\Psi_\eps + R_\eps)}_{0}& \lesssim 
\biggl( \|\alpha\|_{C^1(\T)} +  \|\kappa\|_{C^1(\T)} +  \|\frac{\eta'}{\eps}\|_{C^{1}(\T)}\biggr)\\
&\times \biggl(\iii{\tilde\Psi_\eps}_{0} + \iii{(\partial_\theta + i\mu) \tilde\Psi_\eps}_{0} + \iii{\partial_\mu \tilde\Psi_\eps}_{0} + \iii{(\partial_\theta + i\mu)^2 \tilde\Psi_\eps}_{0}\biggr).
\end{align}
By construction we have $\|\frac{\eta'}{\eps}\|_{C^1(\T)} \lesssim 1$; the same holds for the other $C^1$-norms since we assumed that the boundary $\partial\Omega$ is $C^4$. Finally, the
remaining factor on the right hand side above is bounded thanks to \eqref{bounded.energy}. This establishes \eqref{bounded.rhs.psi.1} and, in turn, inequality 
\eqref{sublinearity.psi.beta} for $\Psi_{1,\eps}$.

\bigskip

Equipped with \eqref{sublinearity.psi.beta}, we may argue again as in Step 1 to obtain \eqref{further.convergence} from \eqref{boundedness} and obtain uniform bounds in the norm $\iii{\cdot}_1$ also for $(1 + \mu)^{-3}\partial_\mu \Psi_{1,\eps}$, $(1 + \mu)^{-3}(\partial_\theta + i\mu)\Psi_{1,\eps}$, as well as $(1 + \mu)^{-3}\partial_\mu^2 \Psi_{1,\eps}$ $(1 + \mu)^{-3}(\partial_\theta + i\mu)^2\Psi_{1,\eps}$. By  standard weak compactness these bounds imply that, up to subsequences, there exists a function $\Psi_1 \in H^2_{loc}(\R \times \R_+)$ satisfying
\begin{align}\label{psi.1.sublinear}
\iii{(1+ \mu)^{-3}\Psi_{1}}_{1}\lesssim  1
\end{align}
and such that for every $R> 0$
\begin{align}
(1 + \mu)^{-3}\Psi_{1,\eps} \rightharpoonup (1 + \mu)^{-3} \Psi_1 \ \ \ \text{in $L^2(\{|\theta| < R\} \times \R_+)$.}
\end{align}
Arguing as done in \eqref{strong.convergence.L2}, we may upgrade the previous convergence to 
\begin{align}\label{strong.convergence.psi.1}
\lim_{\eps \downarrow 0}\int_{|\theta| < R} \int (1 + \mu)^{-6}|\Psi_{1,\eps} - \Psi_1|^2 = 0, \ \ \ \ \forall R > 0.
\end{align}
This time, we use the bounds in the norm $\iii{ \cdot }_1$ for both $(1 + \mu)^{-3}\Psi_{1,\eps}$ and $(1 + \mu)^{-3}\partial_\mu \Psi_{1,\eps}$. 
 
\smallskip

To conclude the proof of this step it remains to show that in the definition of $\Psi_{1,\eps}$ we may substitute $\beta_\eps \vee \eps$ with $\eps$, and that  $\Psi_1$ above is as in \eqref{first.order.approx}. We prove the first claim by proving that
\begin{align}\label{contraddiction.assumption}
\liminf_{\eps \downarrow 0}\frac{\eps}{\beta_\eps} > 0.
\end{align}
We argue in favour of this by contradiction: Let us assume indeed that there exists $\{\eps_j\}_{j\in \N}$ we have $\frac{\eps_j}{\beta_{\eps_j}} \to 0$ and therefore that, for $\eps_j$ small enough, $\Psi_{1,\eps_j}= \frac{\Psi_{\eps_j} - \Psi_{\eps_j,0}}{\beta_{\eps_j}}$. Also here, we drop the index $j$ in the notation.  On the one hand, by \eqref{norm.psi.1.one} and \eqref{strong.convergence.psi.1} with $R=1$ we have that
\begin{align}\label{norm.1.psi}
 \int_{|\theta|< 1}\int (1 +\mu)^{-6}|\Psi_{1}|^2 = 1.
\end{align}
On the other hand, since we assumed that $\frac{\eps}{\beta_\eps} \to 0$, we may pass to the limit in the equation 
 \eqref{equation.psi.1.eps.cut} and obtain that $\Psi_1$ solves in the sense of distributions
 \begin{align*}
\begin{cases}
(H_0 - \lambda_0)\Psi_{1}= 0 \ \ \ \ \ &\text{ in $\R \times \R_+$}\\
\Psi_1 =0 \ \ \ \ \ &\text{ on $\R \times \{ \mu =0 \}$}.
\end{cases}
\end{align*}
Using \eqref{psi.1.sublinear}, we may appeal to Lemma \ref{l.identification} with $m=1$ and infer that there exists $b \in \C$
\begin{align*}
\Psi_1(\theta, \mu) = b e^{-ik_1 \theta} H_1(k_1, \mu).
\end{align*}
Furthermore, by \eqref{strong.convergence.psi.1} and \eqref{orthogonality}, we have that 
\begin{align}\label{orthogonality.limit}
b = \int_{|\theta|< 1} \int \Psi_{1}(\theta, \mu) e^{ik_1 \theta} H_1(k_1, \mu) \d\mu \d\theta = 0.
\end{align}
This contradicts identity \eqref{norm.1.psi} for $\Psi_1$ and therefore implies that \eqref{contraddiction.assumption} necessarily holds. The proof of \eqref{contraddiction.assumption} is complete.

\bigskip

Since we established \eqref{contraddiction.assumption}, the bounds and the weak convergence result obtained above for $\Psi_{1,\eps}$, hold
also when $\beta_\eps$ is substituted by $\eps$. These correspond to \eqref{right.scaling.error}. To conclude this step, it remains to show that the weak limit 
$\Psi_1$ is actually the function defined in \eqref{first.order.approx}. The argument is similar to the part above: We pass to the limit in \eqref{equation.psi.1.eps.cut}, this time with $\beta_\eps = \eps$, and use \eqref{convergence.psi.0} so that $\Psi_1$ solves (in the sense of distributions)
\begin{align*}
\begin{cases}
(H_0- \lambda)\Psi_1 = H_1\Psi_0 \ \ \ \ &\text{in $\R \times \{\mu > 0\}$}\\
\Psi_1=0 &\text{on $\R \times \{\mu = 0\}$},
\end{cases}
\end{align*}
with 
\begin{align*}
H_1 \Psi_0 &:= \biggl( - \kappa(0) \partial_\mu + 2 \kappa(0) \mu \partial_\theta^2 + i 2 \kappa(0) \mu^2 \partial_\theta 
- i \bigl(3\alpha'(0) \kappa(0) +\alpha(0)\kappa'(0) \bigr)\bigl(\mu + \mu^2 \partial_\mu \bigr) - \kappa (0) \mu^3\biggr)\Psi_0\\
&\stackrel{\eqref{characterization.psi.0}}{=} -\biggl(\kappa(0)\partial_\mu H_1(k_1 , \mu) + 2 \kappa(0)\mu \bigl(k_1^2 + k_1 \mu + \frac{1}{2}\mu^2\bigr)H_1(k_1,\mu)\\
&\quad\quad + i\bigl(3\alpha'(0) \kappa(0) +\alpha(0)\kappa'(0) \bigr)\bigl(\mu H_1(k_1,\mu) + \mu^2 \partial_\mu H_1(k_1,\mu) \biggr)e^{-ik_1 \theta}.
\end{align*}
We may appeal again to Lemma \ref{l.identification} with $m=1$ and conclude \eqref{first.order.approx}. We remark that, thanks to \eqref{orthogonality.limit}, the limit  $\Psi_1$ does not contain any multiple of $e^{-i k_1\theta}H_1(k_1, \mu)$. The proof of Step 2 is complete.

\bigskip

\noindent {\em Step 3. Second-order error. \,} As in Step 2, we consider the case $\xi^*=0$ and drop the argument $\xi^*$ in the notation for $\tilde\Psi_\eps$, $\Psi_0$ and $A$. As in the previous step, we fix any sequence $\{\eps_j\}_{j\in \N}$ such that $\tilde \Psi_\eps$ converges as in Step 1 but drop the index $j \in \N$ in all the notation. 

\bigskip

We argue similarly to Step 2, this time to compute the second-order error for $\tilde\Psi_\eps$. We thus only
give a sketch of how to prove this step: We define
\begin{align}
 \Psi_1^\eps = A_{\eps} \bigl( i B_1( k_{1,\eps})\kappa(0) \theta H_1(k_{1,\eps}, \mu)  + W_{\eps,1}(\mu) ) e^{i k_{1,\eps} \theta},
\end{align}
i.e. the analogue of $\Psi_1$ above, this time solving 
\begin{align*}
\begin{cases}
(H_0- \eps^2\lambda_\eps)\Psi_{1}^\eps = H_1\Psi_{0,\eps} \ \ \ \ &\text{in $\R \times \{\mu > 0\}$}\\
\Psi_{1}^\eps=0 &\text{on $\R \times \{\mu = 0\}$}.
\end{cases}
\end{align*}
By construction of $\Psi_{1,\eps}$ and by the definition of $\Psi_1^\eps$ it holds
\begin{align*}
 \int_{|\theta|<1}\int (\Psi_{1,\eps} -\Psi_{1}^\eps)  e^{i k_{1,\eps} \theta} H_1(k_{1,\eps}, \mu) = 0.
\end{align*}
Let $\eta_\eps$ the same cut-off function of the previous step.  The difference solves
\begin{align*}
\begin{cases}
(H_0 - \eps^2\lambda_\eps)(\Psi_{1,\eps} -\eta_\eps\Psi_{1}^\eps) = \eps (f_\eps + R_\eps) \ \ \ \ \ &\text{ in $\R \times \R_+$}\\
\tilde \Psi_{\eps,1} =0 \ \ \ \ \ &\text{ on $\R \times \{ \mu = 0\}$}
\end{cases}
\end{align*}
with 
\begin{align*}
f_\eps := \eta_\eps( H_1\Psi_{1,\eps} + \frac{(H_1 - H_{1,\eps})}{\eps}\Psi_{0,\eps} + H_{2,\eps}\Psi_{0,\eps}), \ \ \ \ R_\eps(\theta, \mu):= -2\frac{\eta'}{\eps} \partial_\theta  \tilde \Psi_{1,\eps} - \frac{\eta''}{\eps}\tilde \Psi_{1,\eps}
\end{align*}
As for \eqref{bounded.rhs.psi.1} for the functions $f_\eps$ of Step 2, we may argue that
\begin{align}\label{bounded.rhs.psi.2}
\iii{(1 + \mu)^{-6}( f_\eps  + R_\eps)}_{1} \lesssim 1.
\end{align} 
The argument for the first term in $f_\eps$ and $R_\eps$ is similar to the one in Step 2 with the only difference that we need to consider the norm $\iii{ \cdot }_{1}$ instead of $\iii{ \cdot }_{0}$. The last term may be bounded as well by the same reasoning of Step 2: In this case it suffices to consider the norm $\iii{ \cdot }_{0}$.  The term in the middle may be treated similarly after comparing the definitions of $H_1$ and $H_{1,\eps}$ and using the assumption on the regularity of the domain which
allows us to bound
\begin{align*}
|\kappa(\eps \theta) - \kappa(0)| + |\kappa'(\eps\theta) - \kappa'(0)| + |\alpha(\eps\theta) - \alpha(0)| + |\alpha'(\eps\theta) -\alpha'(0)| \lesssim \eps |\theta|.
\end{align*}
We remark that in this case it is exactly the above estimate which requires  also for this term the use of the norm $\iii{ \cdot }_{1}$ instead of $\iii{\cdot }_{0}$. 

\bigskip

Equipped with \eqref{bounded.rhs.psi.2},  we thus define the quantity 
\begin{align}\label{right.scaling.psi.2}
\beta_\eps := \int_{|\theta|< 1}\int (1+\mu)^{-12}|\Psi_{1,\eps} - \Psi_1|^2 \, \d\mu \, \d\theta.
\end{align}
and argue as in Step 2 and show that also in this case \eqref{contraddiction.assumption} holds.

\bigskip

Hence, the function
\begin{align*}
\tilde \Psi_2^\eps= \frac{\Psi_{1,\eps} - \eta_\eps\Psi_1^\eps}{\eps},
\end{align*}
satisfies
\begin{align*}
 \int_{|\theta|< 1}\int (1+\mu)^{-12}|\tilde \Psi_{2,\eps}|^2 \, \d\mu \, \d\theta \lesssim 1,
\end{align*}
and solves (in the sense of distributions)
\begin{align*}
\begin{cases}
(H_0 - \eps^2\lambda_\eps) \tilde \Psi_{\eps,2}= f_\eps + R_\eps \ \ \ \ \ &\text{ in $\R \times \R_+$}\\
\tilde \Psi_{\eps,2} =0 \ \ \ \ \ &\text{ on $\R \times \{ \mu=0 \}$}.
\end{cases}
\end{align*}
By the same argument of Step 2, with the only difference that now we apply Proposition \ref{proposition.growth} for $m=1$ and $M=1$,
we obtain that
\begin{align}\label{quadratic.growth.psi.2.eps}
\iii{(1 +\mu)^{-6} \tilde \Psi_{2,\eps}}_{2} \lesssim 1. 
\end{align}

\bigskip

\noindent {\em Step 4. Conclusion \,} We now wrap up the previous steps to show \eqref{first.order.error} and \eqref{second.order.error}. For every $\xi^* \in \T$ fixed, we apply Step 1-3 to $\tilde\Psi(\xi^*, \cdot, \cdot)$ defined in \eqref{rescaled.zero.order} of Step 1 and obtain that there exists $A_\eps(\xi^*) \in \C$, $|A_\eps(\xi^*)| \leq 1$ such that for every $\bar \xi \in \T$ with $d(\bar \xi, \xi^*) < \frac 1 4$
\begin{align*}
\biggl(\fint_{d(\eps\theta; \bar\xi) < \eps}\int_0^{+\infty} (1+ \mu)^{-6} |\tilde \Psi_\eps(\xi^*, \theta,\mu) - A_{\eps}(\xi^*)H_1(k_{1,\eps}, \mu) e^{-i k_{1,\eps}\theta}|^2 \biggr)^{\frac 1 2}\stackrel{\eqref{right.scaling.error}}{\lesssim} \eps + d(\xi^* +\bar \xi , \xi^*).
\end{align*}
This corresponds to \eqref{first.order.error} if we  rescale the variables $(\theta, \mu) \to (\frac \xi \eps ,\frac s \eps)$.

\bigskip

Similarly, we use Step 1 and 3 and \eqref{quadratic.growth.psi.2.eps} instead of \eqref{right.scaling.error} to bound for all $\xi \in \T$ with $d(\xi, \xi^*)< \frac 1 4$
\begin{align*}
\biggl(\fint_{d(\eps \theta; \bar\xi) < 1} \int_0^{+\infty}& (1+ \mu^6)^{-2} |\tilde \Psi_\eps(\xi^*, \theta,\mu)\\
&  - \bigl( A_{\eps}(\xi^*)( 1 + i B_1(k_{1,\eps})\kappa(\xi^*) \eps \theta) H_1(k_{1,\eps}, \mu) +\eps A_{\eps}(\xi^*) W_{1,\eps}(\mu)\bigr) e^{-i k_{1,\eps}\theta}|^2 \biggr)^{\frac 1 2}\notag\\
&\lesssim  d(\xi^* + \bar \xi, \xi^*)^2 + \eps^2. 
\end{align*}
 This turns as well into \eqref{second.order.error} if we use the triangle inequality, together with the boundedness of $|A_\eps(\xi^*)|$ and of $\| W_{1,\eps}\|_{L^2((0,+\infty))}$ (see the display after \eqref{first.order.approx}),  and switch back to the original variables $(\xi, s) = (\eps\theta, \eps\mu)$. The proof of Proposition \ref{p.main} is therefore complete.
\end{proof}

\section{Proof of Theorem \ref{t.main} and Corollary \ref{c.spectrum}}
As argued at the beginning of the previous section, we may reduce to prove \eqref{closeness.to.flat.theorem} for the functions $e^{i\rho}\eta\Psi_\eps$ and $\sum_{l=1}^N C_l e^{i(B_l(k_l)\int_0^\xi \kappa(y) \d y -  k_{l,\eps}\frac \xi \eps )} H_l(k_l, \mu)$.  Also in this case, we keep the notation $\Psi_\eps$, $\Psi_{\text{flat},\eps}$ for the previous two functions.

\bigskip

As in the case of Theorem \ref{t.main.easy}, we rely on the following two propositions, which generalize of Propositions \ref{p.main}-\eqref{p.main.2} to the case $N \geq 1$.  To do so, we need to define the following generalization of $m_\eps$ in \eqref{def.m.eps}:  for each $M \geq 1$, we denote by $m_\eps(M)$ the maximum
\begin{align}\label{definition.omega}
m_\eps(M) = \sup_{\xi^* \in \T} \biggl(M^{-1}\int_{|\xi -\xi^*|< M\eps} \int_{\R^+}|\Psi_\eps|^2 \d s \d \xi \biggr)^{\frac 12} \leq 1.  
\end{align}
In other words, the quantity $m_\eps$ in \eqref{def.m.eps} corresponds to $m_\eps(1)$ above.
Similarly, for $M \geq 1$ and $m\in \N$ we write
\begin{align}\label{norm.3.bars}
\iii{g}_{m, M}:= \sup_{\Theta \in \R}( 1+ |\frac{\Theta}{M}|)^{-m}\bigl(\fint_{|\theta - \Theta | < M} \int_{\R_+} |g(\theta, \mu)|^2 \d\theta \, \d\mu \bigr)^{\frac 1 2},
\end{align}
the rescaled version of $\iii{\cdot }_m$ in \eqref{norm.3.bars.easy}.

\bigskip

\begin{prop}\label{p.main.main}
Let $\{\Psi_\eps \}_{\eps>0}, \{\lambda_\eps\}_{\eps> 0}$, $\lambda, \delta$ and $N$ be as in \eqref{distance.landau}. Let $\eps_0$ be as in \eqref{roots.branches}. For every $M \geq 1$ fixed and all $\eps \leq \eps_0$, the following holds:

\smallskip

For each $ \xi^* \in \T$, there exist $\{A_{\eps, M}^{(l)}(\xi^*) \}_{l=1}^N \subset \C$, $|A_{\eps, M}^{(l)}| \leq 1$, such that 
for every $\bar \xi \in \T$ with $d(\xi^*, \bar \xi) < \frac 14$ 
\begin{align}\label{closeness.to.flat.1.general}
\biggl(\fint_{d(\xi; \bar \xi) < M\eps} \int_0^{+\infty} &(1+ (\frac s \eps)^3)^{-2} |\frac{\sqrt \eps}{m_\eps(M)} \Psi_\eps(\xi^* + \xi, s)\notag\\
& \quad\quad - \frac{1}{\sqrt \eps} \sum_{l=1}^N A_{\eps,M}^{(l)}(\xi^*) H_1(k_{1,\eps}, \frac s \eps) e^{-i k_{1,\eps} \frac{\xi}{\eps}}|^2 \d\xi \, \d s \biggr)^{\frac 1 2} \lesssim d(\bar\xi+ \xi^*; \xi^*) + M\eps,
\end{align}
and
\begin{align}\label{closeness.to.flat.2.general}
\biggl(\fint_{d(\xi; \bar \xi) < M\eps} \int_0^{+\infty} &(1+ (\frac s \eps)^6)^{-2} |\frac{\sqrt \eps}{m_\eps(M)} \Psi_\eps(\xi^* + \xi, s) \notag \\
& \quad\quad -\frac{1}{\sqrt \eps} \sum_{l=1}^N A_{\eps,M}^{(l)}(\xi^*)( 1 + i B_l(k_{\eps,l}) \kappa(\xi^*) \xi) H_1(k_{1,\eps}, \frac s \eps) e^{-i k_{1,\eps}\frac \xi \eps}|^2 \biggr)^{\frac 1 2}\lesssim d(\bar\xi+ \xi^*; \xi^*)^2 + M\eps.
\end{align}
\end{prop}

\bigskip

\begin{lem}\label{p.main.main.2} Let $\eps_0$ be as in the previous lemma. For every $\eps \leq \eps_0$ there exist $F_\eps^{(1)}, \cdots , F^{(N)}_\eps \in C^0([0, 1])$ such that for every family $M_\eps \to +\infty$ satisfying $\lim_{\eps \downarrow 0}\eps M_\eps = 0$ the function
\begin{align*}
\tilde \Psi_{\text{flat},\eps}(\xi, s):=
\sum_{l=1}^N F_\eps^{(l)}(\xi) e^{-i k_{l,\eps} \frac \xi \eps}H_l(k_{l,\eps}, \frac s \eps),
\end{align*}
satisfies
\begin{align}\label{closeness.to.flat.improved.N}
\lim_{\eps \to 0}\bigl(\fint_{d(\xi;\xi^*)< M_\eps \eps} \int_{0}^{+\infty}|\frac{\sqrt \eps}{m_\eps(M_\eps)}\Psi_\eps(\xi, s) -
\frac{1}{\sqrt{\eps}}\tilde \Psi_{\text{flat}, \eps}(\xi, s)|^2 \d\xi \d s \bigr)^{\frac 1 2}= 0.
\end{align}
Furthermore, for every $\eps_j \to 0$, there exists a subsequence and $C_1, \cdots, C_N$ with $\sum_{l=1}^N |C_l|^2 = 1$ such that for every $l=1, \cdots, N$
\begin{align}\label{closeness.to.amplitude.N}
F_\eps^{(l)} \to C_l e^{i B_l(k_l) \int_0^\xi \kappa(x) \d x} \ \ \ \ \text{uniformly in $\T$.}
\end{align}
\end{lem}

\bigskip

\begin{proof}[Proof of Theorem \ref{t.main}]   Let $\{M_\eps\}_{\eps> 0}$ be any family satisfying the assumption of Theorem \ref{t.main}.
 
\bigskip

As was done for the proof of Theorem \ref{t.main.easy}, we rely on  Lemma \ref{p.main.2} to show that
\begin{align}\label{uniform.N}
\frac{\sqrt\eps}{m_\eps(M_\eps)} \to 1.
\end{align}
The proof for this limit is similar to the one for \eqref{uniform.norm}: The only difference in this case is that we combine condition $\sum_{l=1}^N |C_l|^2=1$ of Lemma \ref{p.main.main.2} with the fact that, since $M_\eps \to +\infty$,
\begin{align}\label{almost.orthogonality.lemma}
\lim_{\eps \downarrow 0}|\fint_{|\theta|< M_\eps} e^{i(k_j - k_l) \theta} \d\theta | = 0, \ \ \ \ \text{ for all $l,j = 1,\cdots, N$ with $l \neq j$.}
\end{align}

\bigskip

Again as in the proof of Theorem \ref{t.main.easy}, we use \eqref{uniform.N} and \eqref{closeness.to.flat.improved.N} to reduce the proof of Theorem \ref{t.main}  to show that, if we select a sequence $\eps_j \to 0$ such that \eqref{closeness.to.amplitude.N} holds, then
\begin{align*}
\lim_{j \uparrow +\infty}\sup_{\xi^* \in \T}\bigl(\fint_{d(\xi;\xi^*)<  M_j \eps_j} \int_{\R_+} | \frac{1}{\sqrt{\eps_j}}\tilde \Psi_{\text{flat}, \eps_j}(\xi, s) -\frac{1}{\sqrt{ \eps_j}} \Psi_{\text{flat}, \eps_j}(\xi,s)|^2 \d\xi \, \d s \bigr)^{\frac 1 2}= 0.
\end{align*}
By the triangle inequality and condition \eqref{almost.orthogonality.lemma}, we reduce the limit above to 
$$
\lim_{j \uparrow +\infty}\sup_{\xi^* \in \T}\bigl(\fint_{d(\xi;\xi^*)<  M_j \eps_j} | F_\eps^{(|)}(\xi) - F^{(l)}(\xi)|^2 \d\xi \bigr)^{\frac 1 2}= 0,
$$
that is true thanks to \eqref{closeness.to.amplitude.N} of Lemma \ref{p.main.main.2}. This concludes the  proof of Theorem \ref{t.main}. 
\end{proof}

\begin{proof}[Proof of Lemma \ref{p.main.main.2}]
Let $ \{M_\eps \}_{\eps>0}$ be a fixed family such that $M_\eps \to \infty$ and
\begin{align}\label{M.not.too.fast}
\lim_{\eps \downarrow 0}  \eps M_\eps = 0.
\end{align}

\bigskip

We begin by arguing that there exist $F_\eps^{(1)}, \cdots, F_\eps^{(N)} \in C^0([0,1])$ for which  
\begin{align}\label{general.1}
\biggl(\fint_{d(\xi;\bar \xi) < M_\eps \eps} \int_0^{+\infty} &(1+ (\frac s \eps)^6)^{-2} |\frac{\sqrt \eps}{m_\eps(M)} \Psi_\eps(\xi, s) - \frac{1}{\sqrt \eps} \sum_{l=1}^N F_\eps^{(l)}(\xi) H_1(k_{1,\eps}, \frac s \eps) e^{-i k_{1,\eps} \frac{\xi}{\eps}}|^2 \d\xi \, \d s \biggr)^{\frac 1 2} \lesssim  M_\eps \eps.
\end{align}
The proof of this is similar to the one for \eqref{p.main.2.1} of Lemma \ref{p.main.2} and we sketch below only the differences: We define for each $l=1, \cdots, N$
 \begin{align*}
f_{l, \eps}(\bar \xi, \xi^*) := \fint_{|\theta -\frac{\bar \xi}{\eps}| < M_\eps} \int_0^{+\infty} \frac{ \eps}{m_\eps(M_\eps)} \Psi_\eps(\xi^* +\eps\theta, \eps \mu)   e^{i k_{l,\eps}\theta} H_l(k_{l,\eps}, \mu) \d \mu \d\theta
\end{align*} 
and use \eqref{closeness.to.flat.1.general} and the bound $|A^{(l)}_{\eps, M_\eps}| \leq 1$ to infer that
\begin{align}\label{distance.A.1.1}
| f_{l,\eps}(\bar\xi, \xi^*) - A_{l,\eps}(\xi^*)| \lesssim d(\bar\xi+ \xi^*; \xi^*) +M_\eps \eps + \sum_{j\neq l} |\fint_{|\theta - \frac {\bar\xi}{\eps}| < M_\eps} e^{i (k_{j,\eps}- k_{l,\eps})\theta} \d\theta|.
\end{align}
Similarly, using this time \eqref{closeness.to.flat.2.general}, we also have that
\begin{align}\label{distance.A.2}
| f_{l,\eps}(\bar\xi, \xi^*) - A_{l,\eps}(\xi^*) ( 1+ i \bar\xi B_l( k_{l,\eps}) \kappa(\xi^*) )| &\lesssim d(\bar\xi+ \xi^*; \xi^*)^2 \notag \\
& \quad\quad + M_\eps \eps + (1 + M_\eps \eps) \sum_{j\neq l} |\fint_{|\theta - \frac {\bar\xi}{\eps}| < M_\eps} e^{i (k_{j,\eps}- k_{l,\eps})\theta} \d\theta|.
\end{align}
By using that $M_\eps \to +\infty$, condition \eqref{M.not.too.fast} and  that $k_{j,\eps} \to k_j$ for every $j=1, \cdots, N$  the constants in the right-hand sides of \eqref{distance.A.1.1} and \eqref{distance.A.2}  vanish in the limit $\eps \downarrow 0$. We may thus argue for \eqref{general.1} exactly as done in the case $N=1$ in the proof of Lemma \ref{p.main.2}.

\bigskip

As done in Lemma \ref{p.main.2}, we upgrade \eqref{general.1} into \eqref{closeness.to.flat.improved.N} and use \eqref{distance.A.1.1}-\eqref{distance.A.2} and Ascoli-Arzel\'a's theorem to show that each $F_\eps^{(l)}$ converges uniformly to a limit $F_l$ satisfying
\begin{align*}
F_l' = i B_l(k_l)\kappa(\xi) F_l.
\end{align*}
This implies that, for every subsequence, there exist $C_l \in \C$ such that
\begin{align}\label{limit.amplitude}
F_l(\xi)= C_l \exp\bigl(i \int_0^\xi B(x) \d x )\bigr).
\end{align}

\bigskip

To conclude the proof of the theorem, it remains to show that the constants $C_1, \cdots C_N \in \C$ satisfy
the constraint 
\begin{align}\label{complex.sphere.N}
\sum_{l=1}^N |C_l|^2 = 1.
\end{align}
The proof for this identity follows the same lines of \eqref{complex.sphere} in Lemma \ref{p.main.2}: We stress that
\begin{align*}
\lim_{\eps \downarrow 0} \sup_{\bar \xi \in \T}|\fint_{d(\xi; \bar \xi) < M_\eps \eps} \int_0^{+\infty} |\frac{1}{\sqrt \eps} \sum_{l=1}^N F^{(l)}(\xi) H_1(k_{1,\eps}, \frac s \eps) e^{-i k_{1,\eps} \frac{\xi}{\eps}}|^2 \d\xi \, \d s- \sum_{l=1}^N |F_\eps^{(l)}(\bar \xi)|^2|= 0.
\end{align*}
This indeed follows by the triangle inequality, the uniform continuity estimate for each $F_\eps^{(l)}$ (i.e. the equivalent of the first line in \eqref{eps.differential}) and
\begin{align*}
| \fint_{d(\xi; \bar \xi) < M_\eps \eps} \int_0^{+\infty} |\frac{1}{\sqrt \eps} \sum_{l=1}^N F^{(l)}(\bar\xi) H_1(k_{1,\eps}, \frac s \eps) e^{-i k_{1,\eps} \frac{\xi}{\eps}}|^2 \d\xi \, \d s- \sum_{l=1}^N |F_\eps^{(l)}(\bar \xi)|^2 | {\lesssim}  \sum_{j\neq l} |\fint_{|\theta - \frac {\bar\xi}{\eps}| < M_\eps} e^{i (k_{j,\eps}- k_{l,\eps})\theta} \d\theta|.
\end{align*}
\end{proof}

\begin{proof}[Proof of Proposition \ref{p.main.main}] 
The proof of this proposition follows the same lines of the one of Proposition \ref{p.main} and we enumerate below the only parts of the proof which require a non-trivial modification. 

\bigskip

Let $k_{l, \eps}, k_l$, $l=1, \cdots, N$, be as in \eqref{roots.branches} and \eqref{roots}, respectively.  Let $\gamma > 0$ be a (small) parameter that will be fixed later (see Steps 2-3). Since $k_{j,\eps } \to k_j$ with $k_j \neq k_i $ for each $j\neq i$, we may find  $\eps_1>0$ and $M_0 \geq 1$ such that the following holds: For all $M \geq M_0$, $\eps \leq \eps_1$ we may bound
\begin{align}\label{diophantine}
|\fint_{|\theta|< M} e^{i (k_i - k_j) \theta} |\leq \gamma, 
\end{align}
for all $l, j = 1, \cdots N$ with $j \neq l$.

\bigskip

We observe that is suffices to show the result of Proposition \ref{p.main.main} for any $M \geq M_0$. Let indeed assume that \eqref{closeness.to.flat.1.general}-\eqref{closeness.to.flat.2.general} hold 
for $M \geq M_0$. Since for $1\leq M_1 \leq M_2$ we have that
\begin{align}
\frac{M_1}{M_2} m_\eps(M_1) \leq m_\eps( M_2) \leq \frac{M_2}{M_1} m_\eps(M_1).
\end{align}
Then, if $1 \leq \tilde M < M_0$ it suffices to define $A_{\tilde M, \eps}^{(l)}= \frac{m_\eps(M_0)}{m_\eps(M)} A_{M_0, \eps}^{(l)}$ and use the inequality above together with 
\eqref{closeness.to.flat.1.general}-\eqref{closeness.to.flat.2.general} for $M_0$.

\bigskip

\noindent{\em Step 1.\, } Let  $M\geq M_0$ be fixed. For any point $\xi^* \in \T$ we consider the rescaled function 
\begin{align*}
\tilde \Psi_\eps(\theta, \mu) = \frac{\eps}{m_\eps(M)}\Psi_\eps( \xi^* + \eps\theta, \eps \mu).
\end{align*}
We stress that $\tilde\Psi_\eps$ here is the analogue of the one of Proposition \ref{p.main}, with the only difference that we divided for $m_\eps(M)$ instead of $m_\eps$
($=m_\eps(1)$ according to definition \eqref{definition.omega}). Furthermore, this also means that in this case $\tilde\Psi_\eps$ depends not only on the origin $\xi^*$ of the blow-up, but also on the choice of $M$. 

\bigskip

By arguing as in proposition of Proposition \ref{p.main} the function $\tilde \Psi_\eps$ converges in the same spaces to a function $\Psi_0$. This time, Lemma \ref{l.identification} yields that
\begin{align}\label{characterization.psi.0.N}
\Psi_0 (\xi; \theta, \mu)=\sum_{l=1}^{N} A_{M}^{(l)}(\xi^*) H_l(k_l, \mu) e^{i k_l \theta},
\end{align}
for some  $A_M^{(1)}(\xi^*), \cdots, A_{M}^{(N)}(\xi^*) \in \C$. 

\bigskip

\noindent{\em Step 2. \,} We define for each $l=1, \cdots, N$
\begin{align}\label{def.A.l.eps}
A_{l, M, \eps}(\xi^*)= \fint_{|\theta|< M} \int \tilde \Psi_\eps(\xi^*; \theta, \mu) e^{ik_j\theta} H_l(k_j, \mu) \d\mu \d\theta.
\end{align}
We prove the analogue of Step 2 of Proposition \ref{p.main}, this time with 
\begin{align*}
\Psi_{0,\eps}:= \sum_{l=1}^{N} A_{l, M,\eps}(\xi^*) H_l(\mu, k_{l,\eps}) e^{i k_{l,\eps} \theta},
\end{align*}
and
\begin{align}
 \Psi_1 = \sum_{l=1}^{N}A_{M, l}(\xi^*) \bigl( i B_{l}(k_l) \kappa(\xi^*) \theta H_l(k_{l,\eps}, \mu)  + W_{\eps,l}(\mu) ) e^{i k_{l,\eps} \theta},
\end{align}
with $B_l(\cdot)$ as in \eqref{quantities.corollary} and each $W_{\eps,l}$ satisfying the analogue of $W_\eps$ in Step 2 of Proposition \ref{p.main} with $H_1(k_1, \cdot)$ substituted by $H_l(k_l, \cdot)$.

\bigskip

The argument is the same of Proposition \ref{p.main}, with the only difference that, when proving \eqref{contraddiction.assumption}, we need to pick a $\gamma$ small enough (but of order $\sim 1$) in  \eqref{diophantine} to get a contradiction from
$$
\fint_{|\theta| > M}\int_0^{+\infty} (1 + \mu^3)^{-6}|\Psi_1|^2 =1 
$$
for $\Psi_1(\theta, \mu) = \sum_{l=1}^{N} b_l e^{-ik_l \theta} H_l(k_l, \mu)$ with  $b_1, \cdots , b_N \in \C$ and conditions
\begin{align*}
|\fint_{|\theta|< M} \int \Psi_{1,\eps} e^{ik_l \theta} H_l(k_l, \mu) \d\mu \d\theta| \leq \gamma N, \ \ \ \text{for every $l= 1, \cdots, N$.}
\end{align*}

\bigskip

\noindent{\em Step 3.\,} This step is again the analogue of  Step 3 of Proposition \ref{p.main}, this time with
\begin{align*}
 \Psi_1^\eps = \sum_{l=1}^{N}A_{\eps,M, l} \bigl( i B_{l,\eps} \theta H_l(k_{l,\eps}, \mu)  + W_{\eps,l}(\mu) ) e^{i k_{l,\eps} \theta}.
\end{align*}
and the same changes outlined above in Step 2.

\bigskip

\noindent{\em Step 4.\,} Step 4 is exactly as in Proposition \ref{t.main}.
\end{proof}

\subsection{Proof of Corollary \ref{c.spectrum}} 
The proof of this corollary relies on Theorem \ref{t.main}.

\begin{proof}
We begin by showing that for every $m \in 2\pi \Z$ and $l \in \N$ such that $\nu_l(\eps m)= \lambda$ with $\lambda \in C$ and dist$(\lambda, \sigma_{\text{Landau}}) > \delta$, 
there exists $\lambda_\eps \in \sigma( H_\eps)$ satisfying
\begin{align}\label{asympt.analysis}
|\eps^2\lambda_\eps - \bigl(\lambda + \eps \nu'_l(\eps m) B_l(\eps m)\bigr)| < C(N, \delta) \eps^2.  
\end{align}
Let $\eta$ be any cut-off function for the set $\Omega_{\delta_1}$ with $\delta_1$ as in Lemma \ref{l.local.hamiltonian}. Let $\rho_\eps$ be as in Lemma \ref{l.local.magnetic.potential}.
By Lemmas \ref{l.local.hamiltonian}-\ref{l.local.magnetic.potential} and using classical asymptotic methods, we may find a function
$\bar \Psi_\eps: \T \times \T_{\frac 1 \eps} \times \R_+ \to \C$ 
$$
\Psi_\eps(\xi; \theta, \mu):= \bigl( e^{\int_0^\xi (\kappa(x) - 2\pi)\d x } H_l( \eps m; \mu) + \eps W_\eps(\xi, \mu) \bigr) e^{i \eps m \theta},
$$
with $\sup_{\xi \in \T} \int_0^{+\infty} |W_1(\xi,\mu)|^2 \d\mu \lesssim 1$ such that  $\bar\Psi_\eps(\xi, s)= \eta(\xi, s)e^{-i\eps^{-2} \rho_\eps(\xi, s)} \bar\Psi_\eps(\xi; \frac \xi \eps, \frac s \eps)$ solves 
\begin{align}
\begin{cases}
(H_\eps - \lambda_{0,\eps})\bar\Psi_\eps =  f_\eps\ \ \ \ &\text{in $\Omega$}\\
\bar\Psi_\eps =0 \ \ \ &\text{on $\partial\Omega$.}
\end{cases}
\end{align}
with $\| f_\eps \|_{L^2(\Omega)}\lesssim 1$ and $\lambda_{0,\eps}= \eps^{-2}\lambda + \eps^{-1}\nu_l'(\eps m)2\pi B_l(\eps m)$. This immediately implies that
\begin{align*}
\mathop{dist}(\lambda_{0,\eps} , \sigma( H_\eps)) \leq \eps^2 \|f_\eps\|\lesssim 1.
\end{align*}
Since the spectrum $\sigma( H_\eps)$ is discrete,  inequality \eqref{asympt.analysis} follows if we multiply by $\eps^2$. 

\bigskip

We now turn to the other implication: Let $ \{\lambda_\eps \}_{\eps>0}$ satisfy \eqref{distance.landau}. By compactness, we may divide the family into sequences 
such that $\eps^2 \lambda_\eps \to \lambda$ with $\lambda$ satisfying \eqref{distance.landau} for some $N, \lambda$. By applying Theorem \ref{t.main} there exists a subsequence $\{\eps_j\}_{j\in \N}$  and $C_1, \cdots, C_N \in \C$ on the complex $N$-dimensional sphere defining the limit $\Psi_{\text{flat},\eps_j}$ for $\Psi_{\eps_j}$ in \eqref{closeness.to.flat.theorem}. For the sake of a leaner notation, we forget about the index $j\in \N$ in $\{\eps_j\}_{j\in \N}$. We consider \eqref{closeness.to.flat.theorem} for any fixed $M_\eps \uparrow +\infty$ satisfying  $\eps M_\eps \to 0$.

\bigskip

By the triangle inequality and the periodicity of $\Psi_\eps$ in the angular variable, \eqref{closeness.to.flat.theorem} yields that
\begin{align*}
\lim_{\eps \downarrow 0}\bigl(\fint_{d(\xi; 0)<  M_\eps \eps} \int \frac{1}{\eps} | \Psi_{\text{flat}, \eps}(\xi + 1, s) - \Psi_{\text{flat}, \eps}(\xi,s)|^2 \d\xi \, \d s \bigr)^{\frac 1 2}= 0.
\end{align*}
After rescaling the variable $s \mapsto \eps \mu$, the definition of $\Psi_{\text{flat},\eps}$ allows to rewrite the previous limit as
\begin{align*}
\lim_{\eps \downarrow 0}\bigl(\fint_{d(\xi; 0) <  M_\eps \eps} \int |\sum_{j=1}^N C_l \biggl(e^{i( \theta_l(\xi + 1,\eps \mu)- \frac{k_{j,\eps}}{\eps})} - e^{i \theta_l(\xi,\eps\mu)} \biggr) e^{-i k_{l,\eps} \frac{\xi}{\eps}} H_j(k_j, \mu)|^2 \d\xi \, \d s \bigr)^{\frac 1 2}= 0.
\end{align*}
By the definition of $\theta_{l,\eps}(\xi, s)$ in Theorem \ref{t.main} and  \eqref{quantities.corollary} for $\omega_\eps$ the limit above further turns into
\begin{align*}
\lim_{\eps \downarrow 0}\bigl(\fint_{d(\xi; 0)<  M_\eps \eps} \int |\sum_{j=1}^N C_l \biggl(e^{iB_j(k_j) \int_0^{1} \kappa(y) \d y - i\frac{k_{l,\eps}}{\eps}} -1 \biggr) e^{-i k_{l,\eps} \frac{\xi}{\eps}+ i \theta_l(\xi, \eps\mu)} H_j(k_j, \mu)|^2 \d\xi \, \d s \bigr)^{\frac 1 2}= 0.
\end{align*}
By Gauss-Bonnet theorem $\int_0^{1}\kappa(y) \d y = 2\pi$, we  also infer that
\begin{align}\label{spectrum.1}
\lim_{\eps \downarrow 0}\bigl(\fint_{d(\xi; 0)<  M_\eps \eps} \int |\sum_{j=1}^N C_l \biggl(e^{2i\pi B_j(k_j)- i\frac{k_{l,\eps}}{\eps}} - 1 \biggr) e^{-i k_{l,\eps} \frac{\xi}{\eps}+ i \theta_l(\xi, \eps\mu)} H_l(k_l, \mu)|^2 \d\xi \, \d s \bigr)^{\frac 1 2}= 0.
\end{align}
Since $\eps M_\eps \to 0$, for all $j, l =1, \cdots, N$ we have 
\begin{align*}
\lim_{\eps \downarrow 0}&|\fint_{d(\xi; 0)< M_\eps \eps} \int  e^{i(k_{j,\eps} - k_{i,\eps}) \frac \xi \eps  +i ( \theta_l(\xi, \eps\mu)- \theta_j(\xi, \eps \mu))} H_l(k_l, \mu) H_j(k_j, \mu) \d\mu \, \d\xi |\\
& \lesssim \lim_{\eps \downarrow 0}|\fint_{d(\eps\theta; 0)< M_\eps}  e^{i(k_{j,\eps} - k_{i,\eps})\theta  +i ( B_l(k_l)- B_j(k_j)) \int_0^{\eps\theta} \kappa(y) \d y }| = 0.
\end{align*}

\bigskip

By expanding the inner square in  \eqref{spectrum.1}, using the above limits and that $\sum_{l=1}^N |C_l|^2 =1$ and $\int_0^{+\infty}|H_l(k_l, \mu)|^2 \d\mu = 1$ for all $l=1, \cdots, N$
 we conclude that
\begin{align}
 \lim_{\eps \downarrow 0} \sum_{j=1}^N |C_l \bigl(e^{2i\pi B_j(k_j)- i\frac{k_{l,\eps}}{\eps}} - 1 \bigr)|^2 = 0.
\end{align}
Furthermore,  by condition $\sum_{l=1}^N |C_l|^2 =1$ there exists $l \in \{ 1,\cdots N\}$ such that
\begin{align*}
 \lim_{\eps \downarrow 0}|e^{2i\pi B_l(k_l)- i\frac{k_{l,\eps}}{\eps}} - 1|^2 = 0.
\end{align*}
From this condition, it follows that 
\begin{align*}
k_{l,\eps} = q_\eps + \eps 2\pi B_l(k_l)  + \eps \epsilon_\eps, \ \ \  \epsilon_\eps \to 0,
\end{align*}
with $q_\eps := 2\pi \eps \lfloor \frac{k_{l,\eps}}{2\pi \eps}\rfloor \in 2\pi \eps \Z$. Hence, by \eqref{roots.branches}, it holds
\begin{align*}
\eps^2 \lambda^\eps = \nu_l( k_{l,\eps} ) = \nu_l(q_\eps + \eps 2\pi B_l(k_l)  +\eps \epsilon_\eps ),
\end{align*}
and, since $\nu_l$ is analytic (see Lemma \ref{l.basic.flat}), also
\begin{align*}
|\eps^2 \lambda^\eps - \bigl(\nu_l(q_\eps) + \eps \nu_l'(q_\eps) 2\pi B_l(k_l) \bigr) | \lesssim  \eps \epsilon_\eps + \eps^2.
\end{align*}
It now remains to substitute the term $B_l(k_l)$ with $B_l(q_\eps)$. This is allowed since by definition $q_\eps \to k_l$ and therefore
$$
\lim_{\eps \downarrow 0}\nu_l'(q_\eps) 2\pi |B_l(k_l) - B(q_\eps)| \leq C \lim_{\eps \downarrow 0} |B_l(k_l) - B(q_\eps)| = 0.
$$
This establishes Corollary \ref{c.spectrum}.
\end{proof}

\section{Auxiliary Lemmas}\label{s.auxiliary}

\subsection{Auxiliary results on the magnetic Laplacian $H_0$ in the half-plane}\label{s.half.plane}

In the next results, we fix a value $M \geq 1$ in the definition \eqref{norm.3.bars} of the norms $\iii{ \cdot }_{m,M}$. 

\smallskip

The next proposition plays a crucial in the proofs of Theorems \ref{t.main.easy}-\ref{t.main}. Roughly speaking, it states that if $\Psi=\Psi(\theta, \mu)$ is a locally bounded solution to $(H_0 - \lambda)\Psi= g$ in $\R \times \{\mu > 0\}$ with Dirichlet boundary conditions, then if the right-hand side $g$ grows like a polynomial of degree $m$ in the variable $\theta$, the solution $\Psi$ grows at most like $m+1$ in $\theta$.

\begin{prop}\label{proposition.growth}
Let $\lambda \in \R$ satisfy \eqref{distance.landau}. Let $\{\Psi_\eps \}_{\eps > 0}, \{ g_\eps\}_{\eps > 0} \subset H^2( \R \times \R_+)$ be a family of functions which are compactly supported in the variable $\theta$ and such that
\begin{align}
\begin{cases}
(H_0- \lambda)\Psi_\eps = g_\eps \ \ \ \ &\text{ in $\R \times \R_+$}\\
\Psi = 0  \ \ \ \ &\text{ on $\R \times \{ \mu = 0 \}$.}
\end{cases} 
\end{align}
Let us assume that for some $m, n \in \N$ it holds
\begin{align}\label{growth.g}
\limsup_{\eps \to 0} \iii{(1 + \mu)^{-n} g_\eps }_{m, M} \leq 1.
\end{align}
Then
\begin{align}\label{growth.psi}
\limsup_{\eps \to 0} \iii{(1 + \mu)^{-n} \Psi_\eps }_{m +1, M} \lesssim  \biggl(M + \limsup_{\eps \to 0}\fint_{|\theta|<2 M}\int  (1-\mu)^{-2n} |\Psi_\eps(\theta,\mu)|^2 \d\theta \, \d\mu \biggr).
\end{align}

\end{prop}

\bigskip

The following is Liouville-type of statement for solutions to $H_0 -\lambda$ having a certain order of growth in the variable $\theta \in \R$. We give the statement only 
for the cases that we need in the proofs of our main results, namely when the solution grows at most quadratically. It is an easy exercise to extend the statement to 
a general order of integer growth.

\begin{lem}\label{l.identification}
Let $\lambda \in \R_+$ be such that $\lambda \in (2N -1; 2N + 1)$ for some $N \in \N$. We denote by $k_1, \cdots, k_{N}$ the (unique) solutions to
$\nu_1(k_1)= \cdots \nu_{N}(k_{N}) = \lambda$. For $m \in \{1, 2 \}$, let   $\Psi \in H^2_{loc}( \R \times \R_+)$ solve
\begin{align}\label{eq.l.identification}
\begin{cases}
(H- \lambda)\Psi(\mu, \theta) =  \sum_{l=1}^{N}\bigl( \sum_{j=1}^m (i\theta)^{m-1} f_{l,j}(\mu) \bigr) e^{-i k_l \theta} \ \ \ &\text{in $\R \times \R_+$}\\
\Psi(k, 0) = 0  \ \  &\text{on $\R \times \{0\}$.}
\end{cases}
\end{align}
in the sense of distributions, with $f_{j,l} \in L^2(\R_+)$.
\begin{itemize}
 \item  If for some $n \in \N$
\begin{align*}
\iii{(1 +\mu)^{-n}\Psi}_{1,M} \leq 1,
\end{align*}
then, there exist $\{A_{l}\}_{1\leq l\leq N} \subset \C$ such that
\begin{align*}
 \Psi(\theta, \mu) = \sum_{l=1}^{N}\bigl(( A_{l}H_l(k_l, \mu) + i\theta B_{l} H_l(k_l, \mu)  + W_{l}(\mu)) \bigr) e^{-i k_l \theta},
 \end{align*}
where for each $l=1,\cdots, N$ 
\begin{align*}
B_{l}=(\nu'_l(k_l))^{-1} \int_0^{+\infty} H_l(k_l, \mu) f_l(\mu) \d\mu, \ \ W_{l} \perp H_l(k_l, \cdot), \ \ \| W_{l} \|_{L^2} \leq C(\lambda) \biggl( 1 + \|f_l \|_{L^2(\R_+)} \biggr).
\end{align*}

\item  If  for some $n \in \N$
\begin{align}\label{quadratic}
\iii{(1 +\mu)^{-n}\Psi}_{2,M} \leq 1,
\end{align}
then, there exist $\{A_{l}\}_{0\leq j \leq n} \subset \C$ such that
\begin{align*}
 \Psi(\theta, \mu) = \sum_{j=1}^{N}\bigl((A_{l} + i\theta B_{l} - \theta^2 C_{l})H_j(k_j, \mu) + W_{l,1}(\mu) + i\theta W_{l,2}(\mu)) \bigr) e^{-i k_j \theta},
 \end{align*}
where 
\begin{align*}
W_{l,1}, W_{l,2} \perp H_j(k_j, \cdot), \ \  \sum_{j=1}^2\| W_{l,j} \|_{L^2}  \leq C(\lambda) \biggl( 1 +\sum_{j=1}^2 \|f_{l,j} \|_{L^2(\R_+)}\biggr)
\end{align*}
and
\begin{align*}
C_{l}&:=(4 \nu'_l(k_l))^{-1} \int f_{l, 2}(\mu) H_l(k_l, \mu)\d\mu \\
B_{l}&:= (2 \nu'_l(k_l))^{-1} \bigl(\int f_{l,1} H_l(k_l,\mu) \d\mu - 2 \int (\mu + k_l) W_{l,1} H_{l}(k_l, \mu) \d\mu - 2 C_l\bigr).
\end{align*}
\end{itemize}
\end{lem}

\bigskip

The next result is a simple energy estimate that will prove to be useful in many parts of the proofs of Theorems \ref{t.localization}-\ref{t.main}.

\begin{lem}\label{l.energy.estimate}
Let $\Psi, g \in L^2( \T \times \R_+) \cap H^1_{loc}(\T \times \R_+)$ be such that 
\begin{align}
\begin{cases}
(H_0-\lambda) \Psi = g \ \ \ \ &\text{ in $\T \times \R_+$}\\
\Psi = 0  \ \ \ \ &\text{ on $\T \times \{ \mu = 0 \}$}
\end{cases}
\end{align}
in the sense of distributions. Then, for every $n \in \N$ and $r >0$ we have that
\begin{align}\label{iterated.energy}
\int_{|\theta|< r} \int_0^{+\infty} \bigl( |(\partial_\theta + i\mu)^n\Psi(\theta, \mu)|^2 + &|\partial_\mu(\partial_\theta + i \mu)^{n-1}\Psi(\theta, \mu)|^2 \bigr) \d\mu \, \d\theta\notag\\
 &\lesssim \int_{|\theta|< r} \int_0^{+\infty} \bigl(|\Psi|^2 + |(\partial_\theta + i \mu)^{n-1}g(\theta, \mu)|^2\bigr) \d\mu \, \d\theta
\end{align}
\end{lem}

\subsection{Some results on harmonic oscillators on the half-line}
This subsection contains further results on the harmonic oscillators $\O(k)$ defined in \eqref{harmonic.oscillator} of Subsection \ref{s.half.plane.intro}.
For each $k \in \R$, we recall that we denote by $\{ H_l(k; \cdot)\}_{l\in \N} \subset H^1_0(\R_+) \cap H^2(\R_+)$
the eigenfunction associated to the eigenvalues $\{ \nu_l(k)\}_{l\in\N}$ for $\O(k)$ on the half line $\{ \mu > 0\}$, with Dirichlet boundary condition on $\{\mu =0 \}$.
Although many of these results are well-known, we found convenient to tailor them to our needs and give a self-contained proof.

\smallskip

Let $k\in \R$ be fixed. For every $n \in \N$, we define the projection operator  
\begin{align}\label{definition.projection}
P_n(k): L^2(\R_+) \mapsto \mathop{Span}(H_n(k , \cdot)), \ \ \ \ \ \rho \mapsto P_n(k)\rho(\cdot) := \bigl(\int_{0}^{+\infty}\rho(\mu) H_n(k, \mu) \d\mu \bigr) H_n(k, \cdot),
\end{align}
and denote by $P^{\perp}_n(k)$ its orthogonal in $L^2(\R_+)$.

\medskip

The next lemma enumerates some further properties for the eigenfunctions $H_l(k;\cdot)$ and their associated
eigenvalues $\nu_l(k)$ that will be useful in the results of Proposition \ref{proposition.growth}.
\begin{lem}\label{l.eigenfunctions.behaviour}\, Let $k\in \R$ be fixed. Then:
\begin{itemize}
\item[(a)] For every $l \in \N$
\begin{align}\label{asymptotic.nu}
0<\nu_l'(k) \leq C_l |k|,  \ \ \ \ \ \nu_l(k) = k^2 + \Gamma_l k^{\frac 2 3} + O(k^{-\frac 2 3}) \ \ \ \text{for $k \to +\infty$,}
\end{align}
with $\Gamma_l \neq \Gamma_i$ for all $i,l \in \N$ with $i \neq l$ and where the constant in the last error term depends on $l$.
\item[(b)]  For every $m, n \in \N$ there exists a constant $C=C(m,n, l)< +\infty$ and an exponent $0\leq \alpha(m,n)< +\infty$ such that
\begin{align}
\bigl(\int (1+  \mu)^{2n} |\partial_k^m H_l(k, \mu)|^2 \d \mu \bigr)^{\frac 1 2} \leq C(1+ |k|)^{\alpha}. 
\end{align}
\end{itemize}
\end{lem}

\bigskip

The previous estimates rely on the following: 
\begin{lem}\label{l.schwarz.estimate.muk}
Let $\lambda \in\R_+$ be fixed and let $\rho \in \mathcal{S}(\R \times \R_+)$. For every $k\in \R$ and $l\in \N$ let 
$$
D(k, \lambda):=\mathop{dist}(\sigma(\O(k)), \lambda), \ \ \ D_l(k, \lambda)= \mathop{dist}(\sigma(\O(k)) \backslash \{ \nu_l(k) \}, \lambda).
$$
Let $w : \R_+ \to \R$ be any weight such that
\begin{align}\label{linear.weight}
 \|\partial^{j}w \|_{L^\infty(\R_+)} \leq 1, \ \ \ \ \forall j \in \N.
\end{align}
Then 
\begin{itemize}
 \item[(a)] For every $k\in \R$ such that $D(k, \lambda)> 0$, the solution $\eta(k , \cdot) \in L^2(\R_+)$ to
\begin{align*}
\begin{cases}
(\O(k) - \lambda) \eta(k, \cdot) = \rho(k, \cdot) \ \ \ \text{for $\{ \mu > 0\}$}\\
\eta(k, 0) = 0
\end{cases}
\end{align*}
satisfies for any $\alpha \in \N$ and $\beta \in \N \cup \{0\}$
\begin{align}\label{schwarz.estimate.0.1}
\|w^\alpha \partial_\mu \partial_k^\beta \eta(k, \cdot) \|^2_{L^2} + \|w^\alpha& \partial_k^\beta \eta(k, \cdot) \|^2_{L^2} + \|(\mu +k) w^\alpha \partial_k^\beta \eta(k, \cdot) \|^2_{L^2}\notag\\
 &\leq C \sup_{1 \leq n \leq \beta}(\|\partial_k^n \rho(k, \cdot)\|^2_{L^2}+ \| w^\alpha \partial_k^n\rho(k, \cdot)\|^2_{L^2}).
\end{align}
The constant $C=C(\alpha, \lambda, k)$ depends algebraically on $\lambda$ and $D(k, \lambda)^{-1}$.

\medskip

\item[(b)] For $k\in \R$ such that  $D_l(k, \lambda)>0$  for some $l\in \N$, the solution $\eta(k , \cdot) \in L^2(\R_+)$ to
\begin{align*}
\begin{cases}
(\O(k) - \lambda) \eta(k, \cdot) = P_{l}(k)\rho(k, \cdot) \ \ \ \text{for $\{ \mu > 0\}$}\\
\eta(k, 0) = 0
\end{cases}
\end{align*}
with $\eta(k, \cdot) = P_{l}(k)\eta(k, \cdot)$ satisfies for all $\alpha, \beta \in \N$
\begin{align}\label{schwarz.estimate.0.2} 
\|&w^\alpha \partial_k^\beta\eta(k, \cdot) \|^2_{L^2} +  \|(\mu +k) w^\alpha \partial_k^\beta \eta(k, \cdot) \|^2_{L^2}\\
 &\quad\quad \leq C \sup_{1 \leq n,m \leq \beta}\biggl( \|\partial_k^m\rho \|^2 + \| \partial_k^n H_l(k, \cdot)\|^2 + \|w^\alpha\partial_k^m \rho(k, \cdot)\|^2 + \|w^\alpha \partial_k^n H_l(k, \cdot)\|^2 \biggr).
\end{align}
Here, the constant $C=C(\alpha, \beta, k, \lambda)$ may be chosen as in (a), with $D(k,\lambda)$ substituted by $D_l(k, \lambda)$. 
\end{itemize}
\end{lem}

\bigskip

The next lemma is a simple consequence of the energy estimate. 

\begin{lem}\label{l.stupid}
Let $f \in L^2(\R_+)$. Assume that $\alpha \in H^1_{loc}(\R_+)$ solves (in the sense of distributions)
\begin{align}
\begin{cases}
(\O(k) - \lambda) \alpha = f \ \ \ \text{in $\{ \mu > 0 \}$}\\
\alpha(0)=0
\end{cases}
\end{align}
and that
\begin{align}\label{algebraic.growth}
\int_0^{+\infty} (1 + \mu)^{-2n}\bigl(|\alpha|^2 + |\partial_\mu \alpha |^2 \bigr)< +\infty.
\end{align}
Then, $\alpha \in L^2(\R_+)$.  Furthermore, if $f$ is such that $(1 + \mu) f \in L^2(\R_+)$, then $\alpha \in H^1(\R_+)$.
\end{lem}

\bigskip

\subsection{Proofs}
\begin{proof}[Proof of Proposition \ref{proposition.growth}]
With no loss of generality, we assume that $M=1$ and thus simply denote $\iii{\cdot }_{m,1}$ by $\iii{\cdot}_{m}$. In addition, we give the proof only in the cases $N=1$ and that $n=0$. For what concerns the general case, we comment along the proof the few parts which do require a non-trivial adaptation.

\bigskip
{

The proof of the proposition relies on several applications of Lemma \ref{RL} and the easy inequality 
\begin{align}\label{general.inequality.1}
|\int \int_0^{+\infty} g(\theta, \mu) f(\theta, \mu) \d\mu \, \d\theta |& \lesssim  \iii{ g}_{m}  \biggl( \int \int_0^{+\infty}(|\hat f|^2 + |\partial_k^{m+1}\hat f|^2) \d\mu \d k \biggr)^{\frac 1 2}
\end{align}
for $g, f \in L^2(\R \times \R_+)$, where $\hat f= \hat f(k, \mu)$ denotes the Fourier transform of $f$ in the variable $\theta \in \R$.
This inequality is an easy consequence of Cauchy-Schwarz's inequality with the weights $(1 + |\theta|)^{m+1}$ and $(1 + |\theta|)^{-(m+1)}$, the definition of the norm $\iii{g}_m$ and the standard identity  $\partial_k \hat f = i\theta f$ .}

\bigskip

Thanks to \eqref{distance.landau}, we  may denote by $k_1\in \R$ the unique solution to \eqref{roots}. We remark that, again by \eqref{distance.landau}, it holds  $k_1 \neq 0$. We argue that
\begin{align}\label{growth.psi.a}
\limsup_{\eps \to 0} \, \sup_{|\Theta| > 1} |\Theta |^{m+1}\biggl(\fint_{|\theta- \Theta| < 1} \int_0^{+\infty} |\Psi_\eps(\theta, \mu)|^2 \d\mu \d\theta \biggr)^{\frac 1 2} \lesssim 1.
\end{align} 
From this, the statement of Proposition \ref{proposition.growth} immediately follows.

\bigskip

Let us denote by $\hat g_\eps= \hat g_\eps(k ,\mu)$  and $\hat \Psi_\eps= \hat \Psi_\eps(k, \mu)$ the Fourier transforms in the variable $\theta$ of $g_\eps$ and $\Psi_\eps$. Note that by the assumptions on $g_\eps$ and $\Psi_\eps$ we have that also  $\hat \Psi_\eps$ and $\hat g_\eps$ are in $H^2( \R \times \R_+)$ and therefore satisfy for every $k\in \R$ the one-dimensional boundary value problem
\begin{align}\label{equation.psi.hat.full}
\begin{cases}
(\O(k)- \lambda_0)\hat \Psi_{\eps}(k, \cdot) = \hat g_\eps(k, \cdot) \ \ \ \text{ for $\mu > 0$}\\
\hat\Psi_{\eps}(k, 0) = 0.
\end{cases}
\end{align}
In addition, by testing the equation above when $k= k_1$ with $H_1(k_1, \mu)$, we have the compatibility condition
\begin{align}\label{compatiblity.condition}
 \int_0^{+\infty} \hat g_\eps(k_1, \mu) H_1(k_1, \mu) \d\mu =  \int_0^{+\infty} \int g_\eps(\theta, \mu) H_1(k_1,\mu) e^{i k_1\theta} \d\theta \d\mu = 0.
\end{align}
Throughout this proof we skip the index $\eps$ in the notation for $\Psi_\eps$ and $g_\eps$ bearing in mind that all the estimates are independent from $\eps > 0$.

\bigskip

Let $\nu=\nu(k)$ be a (smooth) cut-off function for the set $\{|k - k_1| < 1\}$ in $\{|k-k_1| < 2\}$. By recalling the definition \eqref{definition.projection} of the projection $P_1(k)$, we decompose $\Psi =\Psi_{1,1} + \Psi_{1,2} + \Psi_{2}$ with 
\begin{align*}
\hat \Psi_{1,1}:= \eta P_1(k)\hat \Psi, \ \ \ \hat \Psi_{j,2}: =\eta P^\perp_1(k)\hat \Psi, \ \ \  \hat \Psi_{2} := \hat \Psi - (\hat \Psi_{1,1} + \hat \Psi_{1,2})= (1-\eta) \hat \Psi.
\end{align*}

\bigskip

By \eqref{equation.psi.hat.full}, the fact that both $P_1$ and $\eta$ commute with the operator $\O(k)$ and Lemma \ref{l.basic.flat}, we may rewrite
\begin{align}\label{psi.1.1}
&\hat\Psi_{1,1} \stackrel{\eqref{compatiblity.condition}}{=} \eta  \frac{P_1(k)\hat g- P_1(k_1)\hat g}{\nu_1(k)-\lambda}.
\end{align}
In addition, for every $k \in \R$ we also have that
\begin{align}\label{psi.1.2}
\begin{cases}
(\O(k) - \lambda_0)\hat \Psi_{1,2} = \eta P_1^\perp(K_1)\hat g \\
\hat\Psi_{1,2} = 0
\end{cases},\ \ \ \ \ \int \hat\Psi_{1, 2}(k ,\mu) H_1(k ,\mu) =0
\end{align}
and
\begin{align}\label{psi.2}
\begin{cases}
(\O(k) - \lambda_0)\hat \Psi_{2} = (1-\eta) \hat g\\
\hat\Psi_{2} = 0.
\end{cases}
\end{align}

In next paragraphs, we treat the previous terms separately and show \eqref{growth.psi.a} and \eqref{growth.psi} for each of the them.  

\bigskip

\noindent{\em Term $\Psi_{1,2}$.\, }  We argue by duality that $\Psi_2$ satisfies \eqref{growth.psi.a} and that 
\begin{align}\label{psi.1.2.in.zero}
\fint_{|\theta| < 2} |\Psi_{1,2}|^2 \lesssim 1.
\end{align}
 We claim, indeed that for every $\rho \in C^\infty_0(\R \times \R_+)$ it holds that
\begin{align}\label{estimate.test.function}
|\int_0^{+\infty} \int \Psi_{1,2}(\theta, \mu) \rho(\theta, \mu) \d\theta \, \d\mu |\lesssim \biggl( \int_0^{+\infty} \int (1 + |\theta|)^{2(m+1)} |\rho(\theta, \mu)|^2 \d\theta  \, \d\mu \biggr)^{\frac 1 2}.
\end{align}
This implies in particular, that if for any $\Theta \in \R$, supp$(\rho) \subset \{ |\theta - \Theta| < 1 \} \times \R_+$, then
\begin{align*}
|\int_0^{+\infty} \int \Psi_{1,2}(\theta, \mu) \rho(\theta, \mu) \d\theta \, \d\mu | \lesssim ( 1 + |\Theta|)^{m+1}  \biggl(\int_0^{+\infty} \int |\rho(\theta, \mu)|^2 \d\theta\biggr)^{\frac 1 2},
\end{align*}
which yields \eqref{psi.1.2.in.zero} and \eqref{growth.psi.a} for $\Psi_{1,2}$ by duality.

\bigskip

We argue \eqref{estimate.test.function} in the following way: For each $\rho \in C^\infty_0(\R \times \R_+)$ we use Plancherel's identity to equal
\begin{align*}
\int_0^{+\infty} \int \Psi_{1,2}(\theta, \mu) \rho(\theta, \mu) \d\theta \, \d\mu = \int_{0}^{+\infty} \int \hat \Psi_{1,2}(k, \mu) \hat \rho(k,\mu) \d k \, \d\mu.
\end{align*}
By using \eqref{psi.1.2}, together with the commutation of $P_1(\cdot)$ and $\eta$ with the operator $\O(k)-\lambda$, this also equals to
\begin{align}\label{duality.psi.2}
\int_0^{+\infty} \int \Psi_{1,2}(\theta, \mu) \rho(\theta, \mu) \d\theta \, \d\mu = \int \int_{0}^{+\infty} \hat g_\eps(k, \mu) \eta(k) (\O(k) -\lambda)^{-1}P_1(k) \hat \rho(k, \mu) \d\mu \, \d k.
\end{align}
Again by Plancherel's identity we have that
\begin{align}\label{term.psi.1.2.A}
\int_0^{+\infty} \int \Psi_{1,2}(\theta, \mu) \rho(\theta, \mu) \d\theta \, \d\mu= \int \int_{0}^{+\infty} g_\eps(\theta, \mu) \mathcal{F}^{-1}\bigl(\eta\zeta \bigr)(\theta, \mu) \d\mu \, \d \theta,
\end{align}
with $\zeta:=  (\O(\cdot) -\lambda)^{-1}P(\cdot)\hat \rho$. Hence, by \eqref{general.inequality.1} and \eqref{growth.g} we bound
\begin{align}\label{term.Psi.1.2.B}
|\int_0^{+\infty} \int \Psi_{1,2}(\theta, \mu) \rho(\theta, \mu) \d\theta \, \d\mu|& \lesssim   \biggl( \int \int_0^{+\infty} \bigl(|\eta(k) \zeta(k, \mu)|^2 +   |\partial_k^{m+1}\bigl(\eta(k) \zeta (k, \mu) \bigr)|^2 \bigr) \d\mu \, \d k  \biggr)^{\frac 1 2}.
\end{align}

\bigskip

To conclude  \eqref{estimate.test.function}, it remains to argue that
\begin{align*}
\biggl(&\int \int_0^{+\infty} ( |\eta \zeta(k, \mu)|^2 + |\partial_k^{m+1}( \eta \zeta(k, \mu))|^2 ) \d\mu \d k\biggr)^{\frac 1 2}\lesssim \biggl(\int_0^{+\infty} \int (|\hat \rho(k,\mu)|^2  + |\partial^{1+m}_k\hat \rho(k,\mu)|^2) \d\mu \d k\biggr)^{\frac 1 2}.
\end{align*}
The assumption on the support of $\eta$, Leibniz rule and Lemma \ref{l.schwarz.estimate.muk} with values $\beta \leq m +1$ and $\alpha=0$ applied on the right-hand side imply that
\begin{align*}
\biggl(&\int \int_0^{+\infty} ( |\eta \zeta(k, \mu)|^2 + |\partial_k^{m+1}( \eta \zeta(k, \mu))|^2 ) \d\mu \d k\biggr)^{\frac 1 2}\\
&\lesssim \biggl( 1 + \sup_{|k-k_1|< 2} \frac{1}{|\nu_2(k) - \lambda|}\biggr)^{m + n +2} \biggl(\int_0^{+\infty} \int (|\hat \rho(k,\mu)|^2  + |\partial^{1+m}_k\hat \rho(k,\mu)|^2  +|\partial^{1+m}H_1(k, \mu)|^2) \d\mu \d k\biggr)^{\frac 1 2}.
\end{align*}
By Lemma \ref{l.basic.flat}, the supremum on the right-hand side is bounded. Since the term with $\partial_k^{m+1} H_1^2$  on the right-hand side of the 
second inequality above is integrated only on the bounded set $\{ |k - k_1| < 2\}$, by Lemma \ref{l.schwarz.estimate.muk} we also have 
$$
\biggl(\int_{|k-k_1|<2} \int_0^{+\infty}(1 + \mu)^{n}|\partial^{1+m}H_1(k, \mu)|^2 \d\mu \d k\biggr)^{\frac 1 2} \lesssim 1.
$$
This concludes the proof of  \eqref{estimate.test.function} and yields  \eqref{growth.psi.a}-\eqref{growth.psi}  for $\Psi_{1,2}$.

\bigskip

\noindent{\em Term $\Psi_{2}$.} \, We prove  \eqref{growth.psi} and \eqref{psi.1.2.in.zero} for $\Psi_{2}$ as done for $\Psi_{1,2}$. This case is simpler than the previous one as the function $\Psi_2$ is supported on the frequencies $k$ where the operator $(\O(k) -\lambda)$ is invertible. 

\bigskip

\noindent{\em Term $\Psi_{1,1}$.\,} We now turn to $\Psi_{1,1}$ and claim that for every $|\Theta| > 1$, 
\begin{align}\label{main.estimate.A}
 \biggl(\fint_{|\theta|< 1}&  \int_0^{+\infty} |\Psi_{1,1}(\theta+ \Theta, \mu) - \Psi_{1,1}(\theta, \mu)|^2 \d\mu \, \d\theta \biggr)^{\frac 1 2} 
 \lesssim |\Theta|^{1 + m}.
\end{align}
This, together with the triangle inequality and \eqref{psi.1.2.in.zero} for both $\Psi_{1,2}$ and $\Psi_{2}$ implies \eqref{growth.psi} also for the function $\Psi_{1,1}$.

\smallskip

  Let us define $G(k) = \int P_1(k) \hat g(k, \mu) H_1(k, \mu) \d\mu$ so that, by definition of the
projection $P_1(k)$ (see \eqref{definition.projection}),  we may write
\begin{align}\label{def.GK}
P_1^\perp(k) \hat g(k, \mu) = G(k) H_1(k, \mu).
\end{align}
By inverting in \eqref{psi.1.1} the Fourier transform and using this new notation we thus have
\begin{align*}
\Psi_{1,1}(\theta, \mu) &= \int e^{-i \theta k}\eta(k)\frac{G(k)- G(k_1)}{\nu_1(k)-\lambda} H_1(k, \mu) \d k.
\end{align*}

\bigskip

Since $k_1 \neq 0$, we may fix $\bar R:= \frac{4\pi}{ k_1}$ and observe that it suffices to prove that for every $\Theta \in \frac{2\pi}{k_1}\Z$ with $|\Theta| > 1$, it holds
\begin{align}\label{main.estimate}
 \biggl(\int_{|\theta|< \bar R}&  \int_0^{+\infty} |\Psi_{1,1}(\theta+ \Theta, \mu) - \Psi_{1,1}(\theta, \mu)|^2 \d\mu \, \d\theta \biggr)^{\frac 1 2} 
 \lesssim  |\Theta|^{1 + m}.
\end{align}
Indeed, since for all $\Theta \in \R$ we may find $\Theta_0 \in  \frac{2\pi}{k_1}\Z$ with $|\Theta -\Theta_0| < 2\pi k_1$, inequality \eqref{main.estimate.A} follows immediately from the estimate above.

\bigskip

With no loss of generality, we argue \eqref{main.estimate} in the case of $R=1$. We write the integral
\begin{equation}
\begin{aligned}\label{first.identity}
\int_{|\theta|< \bar R}  \int_0^{+\infty} &|\Psi_{1,1}(\theta+ \Theta, \mu) - \Psi(\theta, \mu)|^2 \d\mu \, \d\theta \\
&= \int_{|\theta|< \bar R} \int_0^{+\infty} |\int (e^{-i (\theta+\Theta) k} - e^{-i\theta k}) \eta(k)\frac{G(k)- G(k_1)}{\nu_1(k)-\lambda} H_1(k, \mu)\, \d k|^2 \, \d\mu \, \d\theta
\end{aligned}
\end{equation}
and begin by claiming that
\begin{equation}
\begin{aligned}\label{psi.1.1.reduced}
\biggl(\int_{|\theta|< 1} & \int_0^{+\infty}|\Psi_{1,1}(\theta+ \Theta, \mu) - \Psi_{1,1}(\theta, \mu)|^2 \d\mu \, \d\theta \biggr)^{\frac 1 2}\\
& \lesssim (1 + |\Theta|)^{2m} \\
& \quad + \int_{|\theta|<1} |\int \int (e^{-i (\theta+\Theta) k} - e^{-i\theta k})e^{- |k-k_1|^2}  (\int g(y, \mu) H_1(k_1, \mu) \d\mu) 
\frac{e^{iy k} - e^{i y k_1}}{k-k_1} \, \d k \, \d y|^2 \, \d\theta.
\end{aligned}
\end{equation}

\medskip

In other words, the previous inequality yields that, up to an error $( 1 + |\Theta|)^{2m}$, we may substitute in the right-hand side in \eqref{first.identity}  the functions  $H_1(k_1,\mu)$ to $H_1(k, \mu)$,  $k-k_1$ to $\nu_1(k) -\lambda$, $e^{-(k-k_1)^2}$ to $\eta$ and finally $\int_0^{+\infty} (\hat g_\eps(k, \tilde \mu) - \hat g_\eps(k_1, \tilde\mu)) H_1(k_1, \tilde\mu) \d\tilde\mu$ to $G_1(k)- G_1(k_1)$. We claim that all the error terms produced by the previous substitutions may be tackled by means of Lemma \ref{RL}. We give the argument for the first error term, i.e.
$$
 \int_{|\theta|<1} \int_0^{+\infty} |\int (e^{-i (\theta+\Theta) k} - e^{-i\theta k}) \eta(k) G_1(k) \frac{H_1(k, \mu)- H_1(k_1,\mu)}{\nu_1(k)-\lambda} \, \d k|^2 \, \d\mu,
$$
as the other remaining terms follow as well by analogous manipulations and the estimates of Lemma \ref{RL}.  We claim that for every $\theta \in \R$
\begin{align}\label{first.term.psi.1.1.pointwise}
 \int_0^{+\infty} |\int e^{-i (\theta k} \eta(k) G_1(k) \frac{H_1(k, \mu)- H_1(k_1,\mu)}{\nu_1(k)-\lambda} \, \d k|^2 \, \d\mu \lesssim |\theta|^{2m}+1.
\end{align}
By the triangle inequality this estimate implies immediately that the error term above is bounded by $(1+ |\Theta|)^{2m}$. Thanks to the definition of $G_1$, we may invert the Fourier transform and rewrite the left-hand side above as
\begin{align}\label{first.term.psi.1.1.pointwise}
 \int_0^{+\infty} |\int \int e^{-i (\theta - y)k} \bigl(\int g(y, x) H_1(k, x) \d x \bigr) \eta(k)  \frac{H_1(k, \mu)- H_1(k_1,\mu)}{\nu_1(k)-\lambda} \, \d y \, \d k|^2 \, \d\mu 
\end{align}

Inequality \eqref{first.term.psi.1.1.pointwise} is an immediate consequence of \eqref{RL.2} of Lemma \ref{RL} for the operator $T$ applied to the function $g$ with  $\rho(k,\mu) = \eta(k) \frac{H_1(k, \mu)- H_1(k_1,\mu)}{\nu_1(k)-\lambda}$. 

\bigskip

It thus remains to tackle the term on the right-hand side of \eqref{psi.1.1.reduced} and show that this is in turn  bounded by the right-hand side of \eqref{main.estimate}. By applying Lemma \ref{l.chi} to the inner two integrals on the right-hand side of \eqref{psi.1.1.reduced}, with the function $u(y)= \int g(y, \mu) H_1(k_1, \mu) \d\mu$, we obtain that
\begin{align*}
&\int_{|\theta|<1} |\int  \int (e^{-i (\theta+\Theta) k} - e^{-i\theta k})e^{-|k-k_1|^2}\biggr(\int g(y, \mu) H_1(k_1, \mu) \d\mu \biggl) \frac{e^{i ky}- e^{ik_1y}}{k-k_1} \d y \, \d k|^2 \, \d\theta\\
&= i \int_{|\theta|<1} |\int_{\theta +\Theta}^{\theta} \int \biggr(\int g(y, \mu) H_1(k_1, \mu) \d\mu \biggl)e^{-i k_1(\theta -y)}  (e^{-\frac {(y-x)^2} {2}} - e^{-\frac {x^2} {2}}) \d y \, \d x  |^2 \, \d\theta.
\end{align*}
Since by \eqref{compatiblity.condition} we have
\begin{align*}
 \int_0^{+\infty} \hat g(k_1, \mu) H_1(k_1, \mu) \d\mu =  \int_0^{+\infty} \int e^{i y k_1} g(y, \mu) H_1(k_1, \mu) \d\mu \d y = 0, 
\end{align*}
the above terms reduce to
\begin{align*}
&\int_{|\theta|<1} |\int  \int (e^{-i (\theta+\Theta) k} - e^{-i\theta k})e^{-|k-k_1|^2}\biggr(\int g(y, \mu) H_1(k_1, \mu) \d\mu \biggl) \frac{e^{i ky}- e^{ik_1y}}{k-k_1} \d y \, \d k|^2 \, \d\theta\\
&\quad = -i \int_{|\theta|<1} |\int_{\theta}^{\theta + \Theta} \int  e^{-c(y-x)^2}\biggr(\int g(y, \mu) H_1(k_1, \mu)e^{-i k_1(\theta -y)} \d\mu \biggl)
\d y \, \d x|^2 \, \d\theta.
\end{align*}
By using Cauchy-Schwarz inequality together with Lemma \ref{l.basic.flat} this implies that 
\begin{align*}
&\int_{|\theta|<1} |\int  \int (e^{-i (\theta+\Theta) k} - e^{-i\theta k})e^{-|k-k_1|^2}\biggr(\int g(y, \mu) H_1(k_1, \mu) \d\mu \biggl) \frac{e^{i ky}- e^{ik_1y}}{k-k_1} \d y \, \d k|^2 \, \d\theta\\
&\lesssim \int_{|\theta|<1} |\int_{\theta}^{\theta + \Theta} \int  e^{-c(y-x)^2}\bigr(\int  |g(y, \mu)|^2 \d\mu \bigl)^{\frac 1 2}\d y \, \d x|^2 \, \d\theta\\
&\lesssim (|\Theta| + 1) \int_{|x| < |\Theta| + 1} \int  e^{-c(y-x)^2}\int  |g(y, \mu)|^2 \d\mu \, \d y \, \d x \, \d\theta \stackrel{\eqref{growth.g}}{\lesssim}(|\Theta| + 1)^{2m + 2}.
\end{align*}
By \eqref{psi.1.1.reduced}, this yields \eqref{main.estimate} and completes the proof of \eqref{growth.psi.a} for $\Psi_{1,1}$. 

\bigskip

By combining the previous three steps, the triangle inequality yields \eqref{growth.psi} also for the function $\Psi$. This concludes the proof of Proposition \ref{proposition.growth}.
\end{proof}

\begin{proof}[Proof of Lemma \ref{l.identification}]
We give the proof in the case $M=1$ and $m=2$. Furthermore, in order to keep a leaner notation, we also assume that $N=1$. We thus suppress the index $l (=1)$ in all the functions considered. Let $g(\mu, \theta):= \bigl( f_{1}(\mu) + i\theta f_{2}(\mu) \bigr) e^{-i k_1 \theta}$. Let $\hat \Psi= \hat\Psi(k, \mu)$ and $\hat g= \hat g(k, \mu)$ be the Fourier transforms of $\Psi$ and $g$ with respect to the variable $\theta$. By the explicit formulation of $g$ we have that $\hat g$ is a Schwarz distribution in $k$ and equals to
\begin{align*}
\hat g(k, \mu) =f_{1}(\mu) \delta(k - k_1) + f_2(\mu) \partial_k\delta( k -k_1),
\end{align*} 
in the sense that it defines a linear bounded functional over functions $\rho \in \mathcal{S}(\R \times \R_+)$. Similarly, by assumption \eqref{quadratic},
we have that $\hat \Psi$ defines a Schwarz distribution in the sense that for every $\rho \in  \mathcal{S}(\R \times \R_+)$ we may define
\begin{align*}
\langle \hat \Psi , \rho \rangle = \int \int_0^{+\infty} \Psi(\theta, \mu) \mathcal{F}(\rho)(\theta, \mu) \d\mu \d\theta,
\end{align*}
and, by Cauchy-Schwarz inequality and Plancherel's identity, bound for any $\alpha> \frac 5 2$
\begin{align*}
|\langle \hat \Psi , \rho \rangle| \leq \bigl(\int  \int_0^{+\infty} \frac{|\Psi(\theta, \mu) |^2}{1 + |\theta|^{2\alpha}}\bigr)^{\frac 1 2} \bigl( \int  \int_0^{+\infty} |(1 -\partial_k^2)^{\frac \alpha 2}\rho( k, \mu)|^2 \bigr)^{\frac 12} {\lesssim} \iii{\Psi}_2 \bigl( \int  \int_0^{+\infty} |(1 -\partial_k^2)^{\frac \alpha 2}\rho( k, \mu)|^2 \bigr)^{\frac 12}.
\end{align*}

\bigskip

Furthermore, from the equation for $\Psi$ we obtain that 
\begin{align*}
\begin{cases}
(\O(k)- \lambda)\hat \Psi(k, \mu ) = f_1(\mu)\delta( k -k_l) + f_2(\mu)  \partial_k^j\delta( k -k_l)\\
\hat \Psi(k, 0) = 0
\end{cases}
\end{align*}
in the sense that for every $\rho \in \mathcal{S}(\R \times \R_+)$ with $\rho(k, 0)=0$
\begin{align}\label{very.weak.formulation}
\langle (\O(k) - \lambda) \rho ; \hat \Psi \rangle =   \int_0^{+\infty} f_1(\mu) \rho(k_1, \mu) \d \mu +  \int_0^{+\infty} f_2(\mu) \partial_k\rho(k_1, \mu) \d\mu.
\end{align}
We begin by observing that the distribution $\hat \Psi$ is supported on $\{ k_1 \} \times \R_+$. This indeed follows by using the previous identity with any $\rho$
being supported in $\R \backslash \{k_1 \} \times \R_+$. This, together with the assumption \eqref{quadratic} on the growth of $\Psi$, implies that
\begin{align}\label{Psi.Fourier.identification}
\hat \Psi(k, \mu) = \alpha_{1}(\mu)\delta(k-k_1) + \alpha_2 \partial_k\delta(k - k_j)+ \alpha_3 \partial^2_k \delta(k-k_1),
\end{align}
with
$$
 \int_0^{+\infty} (1 + \mu)^{-2n}|\alpha_{j}(\mu)|^2 \d\mu < +\infty, \ \ \ \ \alpha_{j}(0)=0,
$$
for all $j= 1,\cdots, 3$.

\bigskip

By inserting \eqref{Psi.Fourier.identification} into the formulation \eqref{very.weak.formulation} the arbitrariness of $\rho \in \mathcal{S}$ and the condition $\alpha_j(0)=0$ for all $j =1, \cdots, 3$ yields that 
\begin{align*}
\begin{cases}
&(\O(k_1) - \lambda)\alpha_{3} = 0, \\
&(\O(k_1) - \lambda)\alpha_{2} = 4( k_1 + \mu) \alpha_{3} + f_{2},\\ 
&(\O(k_1) - \lambda)\alpha_{1} = 2( k_1 + \mu) \alpha_{2} + 2\alpha_{3} + f_{1}.
\end{cases}
\end{align*}
in $\R_+$. By Lemma \ref{l.stupid} applied to the first equation above, implies that $\alpha_m \in L^2(\R_+)$ and thus that
\begin{align}
\alpha_3 =  C H_1(k_1, \mu).
\end{align}
Furthermore, by testing the equation for $\alpha_2$ with $H_1(k_1,\mu)$ and using the above identity for $\alpha_3$, we infer the compatibility condition
\begin{align}\label{compatibility.liouville}
4 C \int_0^{+\infty}(\mu + k_1) |H_1(k_1, \mu)|^2 \d\mu = \int_0^{+\infty} f_{2}(\mu) H_1(k_1, \mu) \d\mu,
\end{align}
which by Lemma \ref{l.basic.flat} turns into the desired estimate for $A_3$.

\bigskip

To conclude the proof, it remains to argue in the same way also for $\alpha_2$ and $\alpha_1$: If we write $\alpha_2:= B H_1(k_1, \mu) + W_1$ with $W_1:= P_1^\perp(k_1)\alpha_2$ then by \eqref{compatibility.liouville} and the equation for $\alpha_2$ we have that $W_1 \in L^2(\R_+)$. Furthermore, inserting the previous decomposition of $\alpha_2$ into the equation for $\alpha_1$ and testing again with $H_1(k_1, \cdot)$, we obtain the desired identity for the term $B$. This, and the decomposition $\alpha_1 = A H_1(k_1, \cdot) + W_2$, $W_2:= P_1^{\perp}(k_1)\alpha_1$, allows also to conclude that $\| W_2 \|_{L^2}$ is bounded by $\| f_1 \|_{L^2}$ and $\| f_2 \|_{L^2}$. The identities obtained for $\alpha_1, \alpha_2$ and $\alpha_3$ and \eqref{Psi.Fourier.identification} establish Lemma \ref{l.identification}.
\end{proof}

\bigskip

\begin{proof}[Proof of Lemma \ref{l.energy.estimate}]
The proof of this result is standard: We first test the equation for $\Psi$ with $\eta_R^2 \Psi$, where $\eta_R$ is a smooth cut-off function in $\{|\theta|< r \}\times \{0 < \mu< R \}$. We apply Cauchy-Schwarz inequality, send $R\to +\infty$ and use the assumptions on assumptions on $\Psi, g$ to infer the inequality for $n=1$. The case $n \geq 1$ follows by iterating the previous procedure.
\end{proof}

\bigskip

\begin{proof}[Proof of Lemma \ref{l.stupid}]
As for Lemma \ref{l.energy.estimate}, also this result is an easy consequence of the energy inequality.  By condition \eqref{algebraic.growth}, the function $(1 + \mu)^{-n}\alpha \in H^1_0(\R_+)$. We denote by $\|\cdot \|$ the $L^2$-norm in $\R_+$. By testing the equation for $\alpha$ with $(1+ \mu)^{-2n} \alpha \in H^1_0(\R_+)$, we obtain
\begin{align}
\|(1 + \mu)^{-n} \partial_\mu \alpha \|^2 + \|(1 + \mu)^{-n}(\mu + k) \alpha\|^2 \leq \lambda \|(1 + \mu)^{-n} \alpha\|^2 + \|(1 + \mu)^{-n}f \|^2 + \|(1 + \mu)^{-n-1}\alpha \|^2 \lesssim 1.
\end{align}
Therefore, this, the triangle inequality and \eqref{algebraic.growth} imply that $\|(1 + \mu)^{-n+1}\alpha\| \leq |k|^{2n} + 1$. This, in turn, yields that also $(1 + \mu)^{-n+1}\alpha \in L^2(\R_+)$. We now test the equation for $\alpha$ with $\eta^2 (1 + \mu)^{-2(n-1)} \alpha$, with $\eta$ any smooth cut-off function: We bound
\begin{align}
\| \eta (1 + \mu)^{-(n-1)}\partial_\mu\alpha \|^2 + \| \eta (\mu + k) (1 +\mu)^{-(n-1)}\alpha \|^2 \lesssim (\lambda +1) \|(1 + \mu)^{-(n-1)}\alpha\|^2 + \| f \|^2.
\end{align}
Hence, $(1 + \mu)^{-(n-1)}\partial_\mu\alpha$ and $(1 + \mu)^{-(n-2)}\alpha$ are in $L^2(\R_+)$. If $f\in L^2(\R_+)$, we may iterate the previous procedure $n$-times and conclude that $\alpha \in L^2(\R_+)$. If, in addition we also assume that $(1 + \mu) f \in L^2(\R_+)$, then another iteration yields $\partial_\mu \alpha \in L^2(\R_+)$, as well as $(1 + \mu) \alpha \in L^2(\R_+)$. 
\end{proof}

\begin{proof}[Proof of Lemma \ref{l.eigenfunctions.behaviour}]
Let $l\in \N$ be fixed. As in the proof of the previous lemma, the notation $\| \cdot \|$ stands for $\| \cdot \|_{L^2(\R_+)}$ where the integration variable
is $\mu$. Bearing in mind that all the implicit constants in the notation $\lesssim$ and $\gtrsim$ depend also on the index $l\in \N$, when no ambiguity occurs we drop it in the 
notation for $\nu_l$, $H_l$ and all the other related quantities. 

\bigskip

Proof of $(a)$. By Lemma \ref{l.basic.flat}, it is immediate that $\nu'(k) \to 0$ for $k \to -\infty$. Furthermore, by \eqref{formula.derivative.nu}, Cauchy-Schwarz 
inequality and the normalization $\|H(k ; \cdot) \|_{L^2}=1$ we have
\begin{align*}
\nu'(k) = \int_0^{+\infty} (\mu + k) |H(k ; \mu)|^2 \d\mu \geq \bigl(\int_{0}^{+\infty} (\mu + k)^2 |H(k; \mu)|^2 \bigr)^{\frac 1 2},
\end{align*}
and, by the energy estimate obtained by testing the equation for $H(k ; \cdot)$ with $H(k ; \cdot)$ itself, we get 
$\nu'(k) \geq \nu(k)^{\frac 1 2}$, i.e. $\nu(k) \leq (k + \nu_l(0) )^2$. This yields the first inequality in $(a)$. We stress that $\nu(0) = 2l + \frac 1 2$. We now turn to proving the asymptotic expansion
for $\nu(\cdot)$ when $k \to +\infty$: To do so, we observe that we may rewrite the equation for $H(k ; \cdot)$ as
\begin{align*}
-\partial_\mu^2 H(k; \cdot) + k \mu H(k; \cdot) + \mu^2 H(k; \mu) = (\nu(k) - k^2) H(k; \cdot),
\end{align*}
so that, when $k>>1$, the function $A(k; \mu) := H(k, (2k)^{-\frac 1 3}\mu)$ solves
\begin{align*}
-\partial_\mu^2 A(k ; \cdot) + \mu A(k ; \cdot) + (2k)^{-\frac 4 3} \mu^2 A(k ; \cdot) = k^{-\frac 2 3} (\nu_l(k) - k^2) A(k ; \cdot),
\end{align*}
with boundary condition $A_l(k ; 0) = 0$. Let us define $\tilde\Gamma(k):=k^{-\frac 2 3} (\nu(k) - k^2)$; For each $m\in \N$, let $\bar A_m= \bar A_m(\mu)$ solve the eigenvalue problem 
\begin{align}\label{Airy.eigenvalues}
\begin{cases}
-\partial_\mu^2 \bar A_m + \mu \bar A_m = \bar\Gamma_m \bar A_m \ \ \ \ \mu > 0 \\
\bar A(0)=0.
\end{cases}
\end{align}
where $\{\bar\Gamma_m\}_{m\in \N} \subset \R_+$ correspond to the zeros of the Airy function $A$ satisfying on $\R$ the equation
$$
-A'' + \mu A = 0. 
$$
{By the Min-Max Theorem for semibounded operators \cite{ReedSimon}[Theorem XIII.1], it is easy to see that $0 < \tilde\Gamma_l - \bar\Gamma_l \lesssim k^{-\frac 4 3}$,
where the constant above depends on $\bar\Gamma_l$, but not on $\tilde\Gamma_l$. By recalling the definition of $\tilde \Gamma_l$, we conclude the second identity
in \eqref{asymptotic.nu}  and conclude the argument for part (a).}

\bigskip

We now turn to $(b)$. We remark that, since the Hamiltonians $\O(k)= -\partial_\mu^2 +(\mu +k)$ depend analytically on $k$, so do the eigenvalues 
$\{ \lambda_j(k)\}_{j\in \N}$ and the associated eigenfunctions $\{H_j(k, \cdot) \}_{j\in \N}$. 
We show that for every $n, m\in \N$ there exists a constant $C=C(m,n,l)< +\infty$ and an exponent $\gamma=\gamma(m,n)$ such that{
\begin{align}\label{decay.eigenfunctions}
\int_0^{+\infty} (\mu + k)^{2n} |\partial_k^m H(k , \mu)|^2 \leq C(1 + k^2)^{\gamma}.
\end{align}}
From this, the statement of the lemma immediately follows by the triangle inequality. 

\bigskip

We start by showing \eqref{decay.eigenfunctions} with $m=0$:  The proof of this is very similar to the proof of \eqref{schwarz.estimate.0.1} with $\alpha > 0$ and $\beta=0$
and where we fixed the weight $w=\mu + k$.  The only difference is that $H(k, \cdot) \in$ ker$(\O(k) - \nu(k))$ and we may not apply any estimate on the inverse of the 
operator. Since we assumed that $\|H(k, \cdot)\|=1$, we may apply the energy estimate and obtain that 
$$
\|\partial_\mu H(k, \cdot) \|^2 + \| (\mu + k) H(k,\cdot)\|^2 \leq \nu(k). 
$$
This, together with (a), yields \eqref{decay.eigenfunctions} in the case $n=1$ and $m=0$. To extend it to higher values of $n$ we inductively apply the energy estimate: 
If we assume that the inequality holds for $n$, then we may test the equation of $H$ with $(\mu + k)^{2n} H \in L^2$ (by hypothesis) and infer that
$$
\|(\mu + k)^n\partial_\mu H(k, \cdot) \|^2 + \| (\mu + k)^{n+1} H(k,\cdot)\|^2 \leq \nu(k)\|(\mu +k)^n H(k, \cdot)\|^2 + \|(\mu + k)^{n-1}H\|^2.
$$
By applying the hypothesis induction, we conclude \eqref{decay.eigenfunctions} for $m=0$ and $n+1$. We remark that at each step the constant gets worse by a multiplicative
factor depending on $\nu_l(k)$ which, by (a), grows at most like $k$. 

\bigskip

We now turn to the case $m>0$ in \eqref{decay.eigenfunctions}. We begin with $m=1$: By differentiating in $k$ the equation solved by $H(k, \cdot)$ we get that for
every $k\in \R$
\begin{align}\label{equation.partial.k}
\begin{cases}
[\O(k) - \nu(k) ] \partial_k H(k, \cdot) = \nu'(k) H(k, \cdot) - 2 (\mu + k) H(k, \cdot) \ \ \ \text{ for $\mu >0$}\\
\partial_k H_k(k, 0)= 0.
\end{cases}
\end{align}
Since the right-hand side above is orthogonal w.r.t. $H_l(k, \cdot)$ (cf. Lemma \ref{l.basic.flat}), and by the normalization $\| H(k, \cdot) \|_{L^2(\R_+)}=1$ for all
$k\in\R$ follows
 \begin{align}\label{orthognality.derivative}
 \int_0^{+\infty} \partial_k H(k ,\mu) H(k, \mu) \, \d\mu = 0,
 \end{align} 
 We may apply Lemma \ref{l.schwarz.estimate.muk}, (b) with $\beta=0$ and $\alpha> 0$ and the weight $w=(\mu +k)$ to obtain
\begin{align*}
\|(\mu +k)^\alpha \partial_k H(k,\cdot) \|^2&\lesssim C_1(k,\alpha,0,\nu_l(k))(\nu'(k)^2 + 1)\|(1 + (\mu + k)^{\alpha})H(k,\cdot) \|^2.
\end{align*}
By Lemma \ref{l.schwarz.estimate.muk} the constant above only depends algebraically on $\nu(k)$ and on $D_l(k, \nu(k))$. By part $(a)$ this implies that $C_1$ above grows
algebraically in $k$. The same holds for the term $(\nu'(k)^2 + 1)$. Thus, by 
 \eqref{decay.eigenfunctions} with $m=0$ and $\alpha=n$ we conclude  $\|(\mu +k)^n \partial_k H(k,\cdot) \|^2\lesssim (1 + |k|)^{2\gamma}$, for some $\gamma=\gamma(\alpha)$, i.e. \eqref{decay.eigenfunctions} for $\beta=1$ and any $n \in \N$.

\bigskip

The case $\beta >1$ follows as above by iterating the same estimates. The only difference is that for $\beta >1$ we have that 
$$
 \int_0^{+\infty} \partial_k^\beta H(k, \mu) H(k,\mu) \neq 0.
$$
Lemma \ref{l.schwarz.estimate.muk} thus only allows to control the terms $(\mu +k)^n P(k)\partial_k^\beta H(k, \mu)$. To control the terms 
$(\mu + k)^n P_l^\perp(k) H_l(k, \cdot)$ we only need to observe that
$$
(\mu + k)^n P^\perp(k) H(k, \cdot)=(\int_0^{+\infty} \partial_k^\beta H(k, \mu) H(k,\mu) ) (\mu + k)^n H(k, \mu)
$$
and that by differentiating $\beta$ times the identity $\| H(k, \cdot) \|=1$, we have
$$
\int_0^{+\infty} \partial_k^\beta H(k, \mu) H(k,\mu) =\frac 1 2 \sum_{j=1}^{\beta-1} \int_0^{+\infty} \partial_k^j H(k, \mu) \partial_k^{\beta-j}H(k,\mu),
$$
which yields that 
$$
\| (\mu + k)^n P^\perp(k) H(k, \cdot)\| \lesssim \bigl(\sup_{j \in \{1, \cdots \beta-1\}} \|\partial_k^j H(k,\cdot) \|\bigr)^2 \|(\mu + k)^n H(k,\cdot) \|.
$$
Since all the term on the right-hand side contain only derivatives lower than $\beta$, we may apply the induction hypothesis and conclude the desired estimate also for the
term $(\mu + k)^n P^\perp(k) H(k, \cdot)$.  This concludes the proof of \eqref{decay.eigenfunctions} and of Lemma \ref{l.eigenfunctions.behaviour}.
\end{proof}

\begin{proof}[Proof of Lemma \ref{l.energy.estimate}]
The proof of this is an easy consequence of an iteration of the energy inequality: We observe that by standard elliptic regularity theory we may infer that also $\Psi \in C^\infty(\R \times \R_+)$. Let $\eta= \eta(\theta)$ be a cut-ff for $\{|\theta| < 1\}$ in $\{|\theta| < 2 \}$ and let us define $\eta_R(\theta) = \eta( \frac \theta R)$. Then by testing the equation for $\Psi$ with $\eta_R^2\Psi$ with itself, we get
\begin{align}
\|\eta_R \partial_\mu\Psi \|_{L^2}^2 + \|\eta_R (\partial_\theta + i\mu) \Psi \|_{L^2}^2 \lesssim \| \eta_R g\|_{L^2}^2 + \| \eta'_R \Psi \|_{L^2}^2, 
\end{align}
which yields, by the assumptions on the support of $\eta$, \eqref{iterated.energy} with $n= 1$. To obtain the estimate in the general case it remains to argue by induction
by writing the equation for $(\partial_\theta + i\mu) \eta \in L^2$ and testing it with $\eta_R^2 (\partial_\theta + i\mu) \eta$.
\end{proof}

\begin{proof}[Proof of Lemma \ref{l.schwarz.estimate.muk}]
To keep our notation leaner, when no ambiguity occurs we write $ \eta$ instead of $\eta(k,\cdot)$ and $\| \cdot \|$ instead of $\| \cdot \|_{L^2(\R_+)}$.
Throughout the proof, the variable $k$ is indeed fixed and all the integrals are in $\mu \in \R_+$.

\smallskip

We begin by showing case (a) and start with case $\beta=0$. We prove it by induction over $\alpha \in \N$. Since $k \notin \{ k_1, \cdots, k_{N}\}$, the operator $(\O(k) -\lambda)$ is invertible. Hence, we immediately have that
\begin{align}\label{energy.0}
\| \eta \|^2 \leq \frac{1}{D(k, \lambda)}\|\rho\|^2.
\end{align}
By the energy inequality, obtained by testing the equation for $\eta$ with $\eta$ itself, we also obtain
\begin{align}\label{schwarz.estimate.energy}
\|\partial_\mu \eta \|^2 + \|(\mu + k)\eta \|^2 \leq \lambda \| \eta \|^2 + \|\rho\|^2 \leq (\frac{\lambda}{D(k)} + 1) \| \rho\|^2.
\end{align}
This, together with \eqref{energy.0} implies \eqref{schwarz.estimate.0.1} with $\beta=0$ and $\alpha =0$. Let us now assume that \eqref{schwarz.estimate.0.1} holds for $\beta=0$ and any $\tilde \alpha \leq \alpha_0$. We prove it for $\alpha + 1$. To do so, we observe that, since $(\mu + k) w^\alpha \eta \in L^2$  by the induction hypothesis, then the assumption \eqref{linear.weight} on the weight $w$ implies that also $w^{\alpha+1} \eta \in L^2$. In addition, this function satisfies
\begin{align*}
\begin{cases}
(\O(k) - \lambda) (w^{\alpha+1} \eta(k, \cdot)) = w^{\alpha +1}\rho(k, \cdot) - 2\alpha w' w^{\alpha}\partial_\mu\eta -\alpha(\alpha -1)(w'' + 2|w'|^2)w^{\alpha -2}\eta \ \ \ \text{for $\{ \mu > 0\}$}\\
w^{\alpha +1}(0)\eta(k, 0) = 0.
\end{cases}
\end{align*}
Thus,
\begin{align*}
\| w^{\alpha+1}\eta\|^2 \leq \frac{1}{D(k, \lambda)} \bigl(\| w^{\alpha +1}\rho(k, \cdot)\|^2 + \|2\alpha w' w^{\alpha}\partial_\mu\eta\|^2 + \|\alpha(\alpha -1)(w'' + 2|w'|^2)w^{\alpha -2}\eta\|^2 \bigr).
\end{align*}
By \eqref{linear.weight} and the induction hypothesis for $\alpha$ and $\alpha-1$ we conclude that
\begin{align}\label{estim.alpha1}
\| w^{\alpha+1}\eta\|^2 \lesssim \| \rho \|^2 + \|w^{\alpha +1}\rho\|^2.
\end{align}  
By testing the equation for $w^{\alpha+1}\eta$ with the function itself,  \eqref{linear.weight} implies also that
\begin{align*}
\| w^{\alpha+1}\partial_\mu \eta \|^2 + \|(\mu +k) w^{\alpha+1} \eta\|^2 \leq \lambda \|w^{\alpha+1}\eta\| + \| w^{\alpha +1}\rho(k, \cdot)\|^2 + \|w^{\alpha}\partial_\mu\eta\|^2 + \|w^{\alpha -2}\eta\|^2 + \| w^\alpha \eta\|^2.
\end{align*} 
We now appeal to, \eqref{estim.alpha1} and the induction hypothesis to conclude \eqref{schwarz.estimate.0.1} for $\beta=0$ and $\alpha +1$. This concludes the proof of the case $\beta=0$. Note that at every iteration the constants multiplied by a finite multiple of $\lambda$ and $\frac{1}{D(k,\lambda)}$.

\bigskip

We now turn to the case $\beta > 0$ and prove it by induction. Let us assume that \eqref{schwarz.estimate.0.1} holds for all $\tilde \beta \leq \beta$ and for all $\alpha$. We show that this implies that \eqref{schwarz.estimate.0.1} is true also for $\beta+1$ and all $\alpha \in \N$. 
By differentiating $\beta + 1$ times the equation for $\eta$ we get
\begin{align*}
\begin{cases}
(\O(k) - \lambda)\partial_k^{\beta+1}\eta = \partial_k^{\beta+1}\rho + 2(\mu + k) \partial^\beta_k\eta + 2\partial_k^{\beta-1}\eta \ \ \ \text{in $\{ \mu > 0\}$} \\
\partial_k^{\beta+1}\eta= 0.
\end{cases}
\end{align*}
Since the operator $(\O(k) - \lambda)$ is invertible by the assumption on $k$, we may estimate
\begin{align*}
\|\partial_k^{\beta+1}\eta\|^2 \leq \frac{1}{D(k)}\bigl(\|\partial_k^{\beta+1}\rho + 2(\mu + k) \partial^\beta_k\eta + 2\partial_k^{\beta-1}\eta\|^2\bigr)
\end{align*}
so that by the induction hypothesis we conclude
\begin{align*}
 \|\partial_k^{\beta+1}\eta\|^2 \lesssim  \frac 1 D \bigl(1 + \frac{\lambda}{D} \bigr)^{\beta+1}\sup_{0 \leq n \leq \beta}\|\partial_k^n\rho\|^2.
\end{align*}
By the energy estimate, this also implies that the same holds for the terms $\partial_\mu \partial_k^{\beta+1}\eta$ and $(\mu+k) \partial_k^{\beta+1}\eta$, i.e. estimate \eqref{schwarz.estimate.0.1} for $\beta+1$ and $\alpha =0$. To extend the estimate to all values of $\alpha$ we may argue by induction over $\alpha$ similarly to the case case $\beta=0$ treated above: Note, indeed, that also in this case $(\mu + k)^\alpha \partial^{\beta+1}\eta$  solves for $\mu > 0$ the equation
\begin{align*}
(\O(k) - \lambda)(\mu + k)^\alpha \partial_k^{\beta+1} \eta &=(\mu + k)^\alpha\partial_k^{\beta+1}\rho + 2 (\mu + k)^{\alpha +1} \partial_k^{\beta}\eta\\
& \quad\quad + 2(\mu + k)^\alpha \partial_k^{\beta-1}\eta 
-2\alpha(\mu + k)^{\alpha-1}\partial_\mu\partial_k^{\beta+1}\eta -\alpha(\alpha-1)(\mu + k)^{\alpha -2}\partial_k^{\beta+1}\eta
\end{align*}
and all the terms on the right-hand side may be treated by the estimates obtained in the previous iteration. This concludes the proof of \eqref{schwarz.estimate.0.1}, (a).

\bigskip

We now turn to case (b). With no loss of generality, let $k$ be such that $D_1(k, \lambda) > 0$. By Lemma \ref{l.basic.flat}, the operator $(\O(k) - \lambda)$ is therefore invertible on the subspace {$\mathcal{H}_1(k):=\overline{\text{span}\{H_l(k, \cdot)\}_{l> 1}}^{L^2}$}
and, since $\eta(k, \cdot) \in \mathcal{H}_1(k)$, it immediately follows that
\begin{align}\label{energy.b}
\| \eta \|^2 \leq \frac{1}{D_1(k, \lambda)}\| \rho \|^2. 
\end{align}
The proof for \eqref{schwarz.estimate.0.2} with $\beta =0$ thus may be argued exactly as in case $(a)$, with the only difference that on the right-hand side the term that appears is $w^\alpha P_{1}\rho$ instead of $w^\alpha \rho$. Since $P_1\rho(k ,\cdot) = \rho(k, \cdot) - (\int \rho(k, \mu) H_1(k ,\mu) \d\mu) H_1(k, \cdot)$, we have
$$
\| w^\alpha P_{1}\rho \| \lesssim \|w^\alpha \rho\| + \|\rho\| \|w^\alpha H_1(k, \cdot )\|.
$$
This yields the proof of (b) in the cases $\beta =0$.

\bigskip

We now turn to the case $\beta >0$: Also in this case, the proof is similar to the one for case (a) with the difference that we need to separately estimate $P_1(k)\partial_k^\beta \eta $ and $P_1^\perp(k)\partial_k^\beta \eta$. We show how this may be done in the case $\beta =1$; the generalization to higher values of $\beta$ is immediate. 

\bigskip

 By differentiating in $k$ the equation for $\eta$ we get
\begin{align}\label{problem.projection}
\begin{cases}
(\O(k) - \lambda) \partial_k \eta(k, \cdot) = \partial_k (P_1(k)\rho) + 2(\mu +k)\eta\ \ \ \text{in $\{ \mu > 0\}$}\\
\partial_k\eta(k, 0) = 0.
\end{cases}
\end{align}
In addition, since the condition $\eta= P_1(k)\eta$ is equivalent to
\begin{align*}
 \int_0^{+\infty} \eta(k, \mu) H_1(k, \mu) = 0, 
\end{align*}
we also get that
\begin{align}\label{orthogonal}
 \int_0^{+\infty} \partial_k \eta(k,\mu) H_1(k, \mu) = -  \int_0^{+\infty} \eta(k, \mu) \partial_k H_1(k ,\mu) \, \d\mu.  
\end{align}
By decomposing $\partial_k \eta = P_1(k)\partial_k\eta + P_1^\perp(k)\partial_k\eta$, identity \eqref{orthogonal} and \eqref{problem.projection} yield
\begin{equation}\label{estimates.partial.eta}
\begin{aligned}
\| P_1^\perp(k)\partial_k\eta \|^2& \lesssim \|\eta\|^2 \| \partial_k H_1(k, \cdot)\|^2\stackrel{\eqref{energy.b}}{\lesssim}\frac{1}{D_1(k,\lambda)}\|\rho\|^2 \| \partial_k H_1(k, \cdot)\|^2,\\
\|P_1(k)\partial_k \eta \|^2&\leq \frac{1}{D_1(k,\lambda)}\| \partial_k(P(k)_1\rho) + 2(\mu +k)\eta\|^2\\
&\stackrel{\text{(b),\, $\beta=0$}}{\lesssim}  \frac{1}{D(k,\lambda)}(1+\frac{\lambda}{D_1(k,\lambda)})\bigl(\|\rho\|^2 + \|\partial_k(P(k)_1\rho) \|\bigr) .
\end{aligned}
\end{equation}
It thus remains to estimate the right-hand side above: We rewrite
\begin{align*}
\partial_k(P_1(k)\rho) = \partial_k\rho - \partial_k(c_1(k) H_1(k, \cdot)), \ \ \ c_1(k):= \int \rho(\mu, k) H_1(k, \mu) \, \d \mu
\end{align*}
so that
\begin{align*}
 \partial_k(P_1(k)\rho) = \partial_k\rho - \partial_k( c_1(k) H_1(k, \cdot))= \partial_k \rho - (\partial_k c_1) H_1(k,\cdot) - c_1 \partial_kH_1(k,\cdot)
\end{align*}
and hence
\begin{align}\label{derivative.of.projection}
\| \partial_k(P_1(k)\rho) \|^2 \lesssim \|\partial_k \rho\|^2 + \bigl( |\partial_k c_1|^2  + |c_1(k)| \|\partial_k H_1(k, \cdot)\|^2 \bigr).
\end{align}
By appealing to the explicit formulation of $c_1$ using Cauchy-Schwarz inequality in the integrals, we conclude that
\begin{align*}
\| \partial_k(P_1(k)\rho) \|^2 \lesssim \bigl(\|\partial_k \rho\|^2 + \|\rho\|^2 \bigr)\bigl( 1 + \|\partial_k H_1(k, \cdot)\|^2\bigr)
\end{align*}
By inserting this into \eqref{estimates.partial.eta} we obtain
\begin{align*}
 \|\partial_k\eta\| \lesssim \frac{1}{D_1(k,\lambda)}(1 + \frac{\lambda}{D_1(k,\lambda)}) \bigl(\|\partial_k \rho\|^2 + \|\rho\|^2 \bigr)\bigl( 1 + \|\partial_k H_1(k, \cdot)\|^2\bigr),
\end{align*}
i.e. \eqref{schwarz.estimate.0.2} for $\alpha = 0$ and $\beta =1$. The case $\alpha > 0$ follows as case (a), again with the only difference that the terms in which the right-hand side $\rho$ appears are of the form $w^\alpha \partial_k^\beta P_1(k)\rho$ and need to be estimated similarly to \eqref{derivative.of.projection}. The proof of Lemma \ref{l.schwarz.estimate.muk} is therefore complete.
\end{proof}

\subsection{Technical lemmas}
The proof of Proposition \ref{proposition.growth} relies on the following two technical Lemmas. For the sake of simplicity, we give the statements for the norms $\iii{ \cdot}_{m,M}$ in the case $M=1$ and write $\iii{\cdot}_m$ instead of $\iii{\cdot}_{m,1}$.
\begin{lem}\label{RL}
 Let $g \in L^2(\R \times \R_+)$ and let $\rho \in \mathcal{S}(\R \times \R_+)$. We denote by 
\begin{align}\label{def.T}
F_\rho(g)(\theta) &:= \int \int \int_0^{+\infty}  e^{i(\theta- y) k} g(y, \mu) \rho(k, \mu) \, \d\mu \, \d y \, \d k \\
T_\rho(g)(\theta, \mu) &:= \int \int e^{i(\theta- y) k} \bigl( \int_0^{+\infty} g(y, x) H_1(k, x) \d x \bigr) \rho(k,\mu) \d y \d k.
\end{align}

Then there exist $\alpha=\alpha(m)< +\infty$ such that
\begin{align}
\sup_{\theta \in \R}(1 + |\theta|)^{-m} |F_\rho(g)(\theta)| +  \iii{T_\rho(g)}_{m} &\lesssim \iii{g}_m \, \int (1 + |k|^\alpha)\sup_{l=0, \cdots, m+2} \biggl( \int_0^{+\infty} |\partial^{l}\rho(k, \mu)|^2 \d\mu \biggr)^{\frac 1 2} \d k, \label{RL.2}\\
\sup_{\mu > 0}\sup_{\theta \in \R}( 1 + |\theta|)^{-m} |T_\rho(g)(\theta, \mu)| &\lesssim \iii{g}_m \, \sup_{\mu > 0}\int (1 + |k|^\alpha)\sup_{l=0, \cdots, m+2}|\partial^{l}\rho(k, \mu)| \d k, \label{RL.1}
\end{align}
\end{lem}

\begin{lem}\label{l.chi}
Let $u \in L^1_{loc}( \R)$ and $k_1 \in \R \backslash \{ 0\}$. Then, there exists a constant $C= C(d) < +\infty$ such that for every $\theta \in \R$ and $\Theta \in \frac{2\pi}{k_1}$ it holds 
\begin{align*}
\int \int (e^{-i (\theta + \Theta) k} - e^{i \theta k}) u(y) e^{-(k-k_1)^2}\frac{e^{i ky}- e^{ik_1 y}}{k-k_1} \, d y \, d k
= i C \int_{\theta}^{\theta + \Theta} \int u(y) e^{-i k_1 (\theta -y)} \bigl( e^{-\frac{(y-x)^2}{2}}- e^{-\frac{x^2}{2}} \bigr) \, \d x \, \d y. 
\end{align*}

\end{lem}

\bigskip

\begin{proof}[Proof of Lemma \ref{RL}]
The proof of \eqref{RL.2} and \eqref{RL.1} relies on an argument similar to the one for the standard Riemann-Lebesgue Lemma \cite{Grafakos}[Proposition 2.2.17.]. We focus only on the operator $T_\rho$; the operator $F_\rho$ may be treated in a similar way. We begin by splitting the integral in $y$ on the right-hand side of \eqref{def.T} into 
an integral over $\{ |y - \theta| < 1 \}$ and $\{|y -\theta| > 1 \}$. We start with the first integral: An application of Cauchy-Schwarz's inequality in the inner integral 
and the normalization assumption for $H_1(k, \cdot)$ yields 
that
\begin{align*}
 | \int_{|y-\theta|< 1} \int e^{i(\theta- y) k} \bigl( \int_0^{+\infty} g(y, x) H_1(k, x) \d x \bigr) \rho(k,\mu) \d y \d k| \lesssim \int |\rho(k, \mu)| \int_{|\theta -y|< 1} \bigl(\int_0^{+\infty} |g(y, x)|^2 \d x \bigr)^{\frac 1 2} \d \theta \d k 
\end{align*}
So that, by definition \eqref{norm.3.bars.easy} we get
\begin{align}\label{RL.2.1}
  | \int_{|y-\theta|< 1} \int e^{i(\theta- y) k} \bigl( \int_0^{+\infty} g(y, x) H_1(k, x) \d x \bigr) \rho(k,\mu) \d y \d k|
  \lesssim \iii{g}_m \int |\rho(k, \mu)|  \d k .
\end{align}
We now tackle the second integral: Since for every $\mu \in (0, +\infty)$ the function $\rho( \cdot, \mu)$ is assumed to be a Schwarz function in the variable $k$, we 
may integrate by parts $m+2$ times in the variable $k$ and get that
\begin{align*}
 |& \int_{|y-\theta| > 1} \int e^{i(\theta- y) k} \bigl( \int_0^{+\infty} g(y, x) H_1(k, x) \d x \bigr) \rho(k,\mu) \d y \d k|\\
 &= |\int_{|y-\theta| > 1} \int (\theta - y)^{-(m+2)} e^{i(\theta- y) k} \partial_k^{m+2}\biggl( \bigl( \int_0^{+\infty} g(y, x) H_1(k, x) \d x \bigr) \rho(k,\mu)\biggr) \d y \d k|.
\end{align*}
The chain rule, and Fubini's theorem thus imply that
\begin{align*}
 |& \int_{|y-\theta| > 1} \int e^{i(\theta- y) k} \bigl( \int_0^{+\infty} g(y, x) H_1(k, x) \d x \bigr) \rho(k,\mu) \d y \, \d k|\\
 &\lesssim \sum_{l=0}^{m+2}  |\int_{|y-\theta| > 1} \int (\theta - y)^{-(m+2)} e^{i(\theta- y) k}\bigl( \int_0^{+\infty} g(y, x) \partial_k^l H_1(k, x) \d x \bigr) \partial_k^{m+2-l}\rho(k,\mu)\biggr) \d y \d k|.
\end{align*}
For each $l$ fixed, the corresponding term of the sum above may be bound by
\begin{align*}
 |\int_{|y-\theta| > 1} \int &(\theta - y)^{-(m+2)} e^{i(\theta- y) k}\bigl( \int_0^{+\infty} g(y, x) \partial_k^l H_1(k, x) \d x \bigr) 
 \partial_k^{m+2-l}\rho(k,\mu)\biggr) \d y \d k|\\
 &\lesssim \int | \partial_k^{m+2-l}\rho(k,\mu)| \bigl( \int_0^{+\infty} |\partial_k^{l}H_1(k, x)|^2 \d x\bigr)^{\frac 1 2}\int_{|y-\theta| > 1} |\theta - y|^{-(m+2)} \bigl( \int_0^{+\infty} |g(y,x)|^2 \d x\bigr)^{\frac 1 2}  
 \end{align*}
By Lemma \ref{l.eigenfunctions.behaviour}, definition \eqref{norm.3.bars.easy} we thus have
\begin{align*}
 |\int_{|y-\theta| > 1} \int (\theta - y)^{-(m+2)} e^{i(\theta- y) k}\bigl( \int_0^{+\infty} g(y, x)& \partial_k^l H_1(k, x) \d x \bigr) 
 \partial_k^{m+2-l}\rho(k,\mu)\biggr) \d y \, \d k|\\
 &\lesssim \iii{g}_m |\theta|^m \int |k|^{\alpha} | \partial_k^{m+2-l}\rho(k,\mu)|
 \end{align*}
for some $\alpha=\alpha(m) < +\infty$. By this, we thus estimate 
\begin{align*}
 |\int_{|y-\theta| > 1} \int e^{i(\theta- y) k} \bigl( \int_0^{+\infty} g(y, x) H_1(k, x) \d x \bigr) \rho(k,\mu) \d y \, \d k|\lesssim \iii{g}_m |\theta|^m \int |k|^{\alpha} (\sum_{l=0}^{m+2}| \partial_k^{l}\rho(k,\mu)| )
 \end{align*}
 so that by \eqref{RL.2.1} and the definition \eqref{def.T} of $T_\rho(g)$ we control
 \begin{align*}
 |T_\rho(g)(\theta, \mu)|\lesssim \iii{g}_m ( 1 + |\theta|^m)  \int |k|^{\alpha} (\sum_{l=0}^{m+2}| \partial_k^{l}\rho(k,\mu)| ).
 \end{align*}
  This yields immediately \eqref{RL.1}.  By applying on the left-hand side the norm $\iii{ \cdot }_m$ and using Minkowski's inequality we conclude also \eqref{RL.2} for $T_\rho$.
\end{proof}

\bigskip

\begin{proof}[Proof of Lemma \ref{l.chi}]
By exchanging the order of the integrals in the left-hand and changing the coordinates $k \mapsto k -k_1$ we rewrite
\begin{align*}
&\int  \int (e^{-i (\theta+\Theta) k} - e^{-i\theta k})e^{-|k-k_1|^2}u(y) \frac{e^{i ky}- e^{ik_1y}}{k-k_1} \d y \, \d k \\
&= \int u(y) \biggl(\int ( e^{-i k_1(\theta + \Theta -y)}e^{-i (\theta+\Theta) k} - e^{-i k_1(\theta -y)}e^{-i\theta k})e^{-k^2} \frac{e^{i k y}- 1}{k}\, \d k \biggr) \, \d y
\end{align*}
so that, since $\Theta \in \frac{2\pi}{k_1}\Z$, we get
\begin{align}\label{chi.1}
&\int  \int (e^{-i (\theta+\Theta) k} - e^{-i\theta k})e^{-|k-k_1|^2}u(y) \frac{e^{i ky}- e^{ik_1y}}{k-k_1} \d y \, \d k = \int u(y)e^{-i k_1(\theta -y)}
\biggl(\chi(\theta + \Theta, y) - \chi(\theta, y)\biggr) \, \d y,
\end{align}
where we defined for $\theta, y \in \R$ the function
\begin{align*}
\chi(y, \theta):= \int e^{- k^2} \frac{e^{i(y-\theta) k} - e^{-i k \theta}}{k} \d k.
\end{align*}
We now observe that 
\begin{align*}
\partial_\theta \chi(y, \theta) =  -i \int e^{- k^2}( e^{i(y-\theta) k} - e^{-i k \theta} ) \d k = e^{-\frac {(y-\theta)^2} {2}} - e^{-\frac {\theta^2} {2}}.
\end{align*}
This, together with \eqref{chi.1}, the Fundamental theorem of calculus and Fubini's theorem allows to conclude the desired identity.
\end{proof}

\section*{Acknowledgements} We thank R. Frank for his insightful comments and suggestions  and  B. Zaslanski for bringing to our attention the problem studied in
this paper. We acknowledge support through the CRC 1060 (The Mathematics of Emergent Effects) that
is funded through the German Science Foundation (DFG), and the Hausdorff Center for Mathematics
(HCM) at the University of Bonn.

\end{document}